\documentclass[1 [leqno,10pt]{amsart}
\usepackage{amsmath,mathrsfs}
\usepackage{amssymb, amsmath}
\usepackage{graphicx}

\numberwithin{equation}{section}

\def\inte#1{
\displaystyle\mathop{#1\kern0pt}^\circ }

\let\pa=\partial

\let\f=\frac

\def\pa{\partial}

\def\virgp{\raise 2pt\hbox{,}}
\def\cdotpv{\raise 2pt\hbox{;}}

\def\eqdefa{\buildrel\hbox{\footnotesize def}\over =}

\def\C{\mathop{\mathbb C\kern 0pt}\nolimits}
\def\DD{\mathop{\mathbb D\kern 0pt}\nolimits}
\def\EE{\mathop{{\mathbb E \kern 0pt}}\nolimits}
\def\K{\mathop{\mathbb K\kern 0pt}\nolimits}
\def\N{\mathop{\mathbb N\kern 0pt}\nolimits}
\def\Q{\mathop{\mathbb Q\kern 0pt}\nolimits}
\def\R{\mathop{\mathbb R\kern 0pt}\nolimits}
\def\SS{\mathop{\mathbb S\kern 0pt}\nolimits}
\def\ZZ{\mathop{\mathbb Z\kern 0pt}\nolimits}
\def\TT{\mathop{\mathbb T\kern 0pt}\nolimits}
\def\P{\mathop{\mathbb P\kern 0pt}\nolimits}

\newcommand{\beq}{\begin{equation}}
\newcommand{\eeq}{\end{equation}}
\newcommand{\ben}{\begin{eqnarray}}
\newcommand{\een}{\end{eqnarray}}
\newcommand{\beno}{\begin{eqnarray*}}
\newcommand{\eeno}{\end{eqnarray*}}

\newtheorem{thm}{Theorem}[section]
\newtheorem{lem}{Lemma}[section]
\newtheorem{rmk}{Remark}[section]

\newtheorem{prop}{Proposition}[section]

\begin{document}
\title[High order approximation for the Boltzmann equation with cutoff]
{High order approximation for the Boltzmann equation without angular cutoff}

\author[L. He and Y. Zhou]{Lingbing He and Yulong Zhou}
\address[L. He]{Department of Mathematical Sciences, Tsinghua University,
Beijing 100084,  P. R.  China.} \email{lbhe@math.tsinghua.edu.cn}

\address[Y. Zhou]{Department of Mathematics, Hong Kong Baptist University,
Hong Kong,  P. R.  China.} \email{yulong.zhou2012@gmail.com}

\maketitle

\begin{abstract} In order to solve the Boltzmann equation numerically, in the present work, we  propose a new model equation to approximate the Boltzmann equation without angular cutoff.
Here the approximate equation incorporates Boltzmann collision operator with angular cut-off  and  the Landau collision operator. As a first step,
we prove the well-posedness theory for our approximate equation. Then in the next step
 we  show the error estimate between the solutions to the approximate equation and the original equation.  Compared to the standard angular cut-off approximation method, our method results in higher order of accuracy.
 \end{abstract}



\vskip 0.2cm



\section{Introduction}
\subsection{The Boltzmann equation}
Our interest is to consider the numerical method for the spatially homogeneous Boltzmann equation with long-range interaction in the case of hard potentials. Here, the spatial homogeneity  means the unknown function is  assumed to be independent of the position variables.
In this case, the Boltzmann equation  reads:
\begin{eqnarray}\label{homb}
\partial _t f =Q(f,f),
\end{eqnarray}
where $f(t,v)\geq 0$ is the   distribution
function   of collision particles which at time
$t\geq 0$  move with velocity
$v\in\R^3$. The Boltzmann collision operator $Q$ is a bilinear
operator which acts only on the velocity variables $v$, that is,
\beno Q(g,h)(v)\eqdefa
\int_{\R^3}\int_{\SS^{2}}B(v-v_*,\sigma)(g'_*h'-g_*h)d\sigma dv_*.
\eeno Here we use the standard shorthand $h=h(v)$, $g_*=g(v_*)$,
$h'=h(v')$, $g'_*=g(v'_*)$ where $v'$, $v_*'$ are given by
\begin{eqnarray}\label{e3}
v'=\frac{v+v_{*}}{2}+\frac{|v-v_{*}|}{2}\sigma\ ,\ \ \
v'_{*}=\frac{v+v_{*}}{2}-\frac{|v-v_{*}|}{2}\sigma\
,\qquad\sigma\in\SS^{2}.
\end{eqnarray}

The nonnegative function $B(v-v_*,\sigma)$  in the collision
operator is called the Boltzmann collision kernel. It is always
  assumed to depend only on $|v-v_{*}|$ and $\langle\frac{v-v_{*}}{|v-v_{*}|},\sigma
  \rangle$. We introduce the angle variable $\theta$ through
  $\cos\theta=\langle\frac{v-v_{*}}{|v-v_{*}|},\sigma
  \rangle$.  Without loss of generality, we may
assume that $B(v-v_{*},\sigma)$ is supported in the set
 $0\leq\theta\leq\frac{\pi}{2}$ , i.e,
$\langle\frac{v-v_{*}}{|v-v_{*}|},\sigma
  \rangle\ge0
$, for otherwise $B$ can be replaced by its symmetrized form:\beno
\bar{B}(v-v_{*},\sigma)=[B(v-v_{*},\sigma)+B(v-v_{*},-\sigma)]\mathbf{1}_{\langle\frac{v-v_{*}}{|v-v_{*}|},\sigma
  \rangle\ge0}.\eeno
Here, $\mathbf{1}_E$ is the characteristic function of the set $E$.

\subsection{Assumptions on the collision kernel} We consider the  collision kernel
satisfying the following assumptions:
\begin{itemize}
\item (A-1) The cross-section $B(v-v_{*},\sigma)$ takes a product form of
\beno  B(v-v_{*},\sigma)=\Phi(|v-v_*|)b(\cos\theta),\eeno
 where both $\Phi$ and $b$ are nonnegative functions.
\item (A-2) The angular function $b(t)$ is not locally integrable and it
satisfies \begin{eqnarray*}  K^{-1}\theta^{-1-2s} \le\sin\theta
b(\cos\theta)\le K
 \theta^{-1-2s},  \quad
\mbox{with}\   0<s<1,\ K \geq 1.
\end{eqnarray*}
\item (A-3) The kinetic factor $\Phi$ takes the form of \beno
\Phi(|v-v_*|)=|v-v_{*}|^{\gamma}.\eeno
\item (A-4) The parameter $\gamma$ verifies that $0 < \gamma\leq 2$.
\end{itemize}
We remark that
under assumption (A-2), we have $A_{2} \eqdefa \int_{\SS^2}b(\cos\theta)\sin^2\theta d\sigma < \infty$.

The solutions of the Boltzmann equation \eqref{homb} have the fundamental physical properties of conserving the total mass, momentum and kinetic energy, that is, for all $t\ge0$,
\beno  \int_{\R^3} f(t,v)\phi(v)dv=\int_{\R^3} f(0,v)\phi(v)dv,\quad \phi(v)=1,v,|v|^2.\eeno
Moreover, there exists a quantity called entropy satisfying the Boltzmann's $H$ theorem, which formally is
\beno -\f{d}{dt} \int_{\R^3} f\log f dv= -\int_{\R^3} Q(f,f)\log f dv\ge0.\eeno

\subsection{Existing results, motivations and difficulties}
The well-posedness of the spatially homogeneous Boltzmann equation with angular cut-off, that is when $\int_{0}^{\pi/2} \sin\theta b(\cos\theta)d\theta < \infty$, had been investigated by many authors. For the hard potentials, Arkeryd [8] and Mischler-Wennberg [20] established the existence and uniqueness of the solutions in weighted $L^{1}$ space. Recently, Lu-Mouhot in \cite{lm} extended the results to the space of non-negative measure with finite non-increasing kinetic energy. For the well-posedness of the spatially homogeneous Boltzmann equation without angular cut-off, we refer to \cite{he3} and the references therein. As for the regularity theory of the equation, we refer to \cite{mv} for the analysis of the positive part of the collision operator and the propagation of smoothness in the case of angular cut-off and  refer to \cite{amuxy6}, \cite{chenhe}, \cite{hmuy}  and \cite{sl} in the case of long-range interaction.

For any $0 <\epsilon \leq \frac{\sqrt{2}}{2}$, let $b^{\epsilon} = b\mathbf{1}_{\sin\frac{\theta}{2}\geq\epsilon}$, and $Q^{\epsilon}$ be the operator associated to the angular cut-off kernel $B^{\epsilon}(v-v_{*},\sigma)=|v-v_{*}|^{\gamma}b^{\epsilon}(\cos\theta)$. That is,
\beno Q^{\epsilon}(g,h)(v)\eqdefa
\int_{\R^3}\int_{\SS^{2}}B^{\epsilon}(v-v_*,\sigma)(g'_*h'-g_*h)d\sigma dv_*.
\eeno
Then the angular cut-off Boltzmann equation
\begin{eqnarray}\label{cutoffboltzmann}
\partial _t f =Q^{\epsilon}(f,f)
\end{eqnarray}
is well-posed(see \cite{HeJiang}). And moreover if $f$ and $f^{\epsilon}$ are solutions to the Boltzmann equation \eqref{homb} and the cutoff Boltzmann equation \eqref{cutoffboltzmann} with the same initial datum $f_{0}$ respectively, then one has
\beno
f = f^{\epsilon} + O(\epsilon^{2-2s}).
\eeno

The cut-off Boltzmann operator $Q^{\epsilon}$ omits all grazing collisions and then results in an error of order $2-2s$. We emphasize that the cutoff Boltzmann equation \eqref{cutoffboltzmann} is not a good approximation to the Boltzmann equation \eqref{homb} as the singularity parameter $s$ approaches to $1$.

The effect of grazing collisions has been studied extensively, and we refer to \cite{ld1} and \cite{he1}. It is proved that the limit of concentrating grazing collisions leads to the Landau collision operator. Mathematically, if denote $b_{\epsilon} = b\mathbf{1}_{\sin\frac{\theta}{2}\leq\epsilon}$, and let $Q_{\epsilon}$ be the operator associated to $B_{\epsilon}(v-v_{*},\sigma)=|v-v_{*}|^{\gamma}b_{\epsilon}(\cos\theta)$, according to \cite{he1}, we shall have
\begin{eqnarray}\label{appqwithql}
 ||\epsilon^{2-2s}Q_{L}(f,f)-Q_{\epsilon}(f,f)||_{L^{1}} \lesssim \epsilon^{3-2s}||f||^{2}_{H^{5}_{\gamma+12}},
\end{eqnarray}
where  the Landau collision operator $Q_{L}$ is defined as
\beno Q_{L}(g,h)(v) \eqdefa
\nabla_{v}\cdot\{\int_{\R^3}a(v-v_{*})[g(v_{*})\nabla_{v}h(v)-\nabla_{v}g(v_{*})h(v)]dv_{*}\}.
\eeno
Here the symmetrical matrix $a$ is given by
\begin{eqnarray}\label{matrix}
a(v)  = \Lambda |v|^{\gamma + 2}(I - \frac{v \otimes v}{|v|^{2}}),
\end{eqnarray}
where $\Lambda$ is a constant.

This motivates us to compensate the omission of grazing collisions by Landau operator. Specifically,
we consider the operator
\begin{eqnarray}\label{approximatedoperator}
  M^{\epsilon}(g,h) \eqdefa Q^{\epsilon}(g,h)+\epsilon^{2-2s}Q_{L}(g,h),
\end{eqnarray}
and propose our approximate equation,
\begin{eqnarray}\label{approximated}
\partial _t f = M^{\epsilon}(f,f).
\end{eqnarray}
If $\tilde{f}^{\epsilon}$ is the solution to equation \eqref{approximated}, we will prove
\begin{eqnarray}\label{errororder}
f = \tilde{f}^{\epsilon} + O(\epsilon^{3-2s}).
\end{eqnarray}
That is, by adding Landau operator to the cutoff Boltzmann equation, we increase the order of error from $2-2s$ to $3-2s$. The accuracy of approximation of the Boltzmann equation \eqref{homb} by equation \eqref{approximated} remains even if the singularity parameter $s$ goes to $1$.
Another motivation for studying equation \eqref{approximated} is the recent development of numerical methods. We believe that our approximate equation can be solved numerically. In this regard, see next subsection for a detailed discussion.
We emphasize that the solutions of our approximate equation \eqref{approximated} also have the above mentioned properties, namely, conservation of mass, moment, energy and entropy dissipation.

In the current paper, we study the well-posedness of equation \eqref{approximated} and then give the error analysis of the approximate equation (\ref{approximated}) and the original Boltzmann equation \eqref{homb}.
There are two main difficulties in the current paper. One is to show the existence of a non-negative  solution to equation \eqref{approximated}. We proceed by constructing a sequence of convergent non-negative functions with its limit being the solution. Since we consider hard potentials ($\gamma > 0$), there will be an increase of weight at each iteration. Observing the coefficient before the weight increased term is strictly less $1$, we prove that, on a whole level, the increased weight is limited. The other difficulty is related to the estimate of the error function $F^{\epsilon}_{R}$ as defined in (\ref{errorf}). Again, weight increase problem happens here and another problem  is no sign information of $F^{\epsilon}_{R}$. We circumvent the problem  of lacking sign information by writing the equation of error function in a suitable way.  The weight increase problem is dealt with by carefully separating the integration region such that either  the  increased  weight is eliminated or the coefficient before the weight increased term is controlled as desired.

\subsection{Existing numerical results and future work}
Our approximate equation contains both the angular cut-off Boltzmann operator $Q^{\epsilon}$ and Landau operator $Q_{L}$. Numerical methods of the Boltzmann equation and Landau equation have been investigated extensively. The most famous one is Kac's program. Kac started from the Markov process corresponding to collisions only, and try to prove the limit towards the
spatially homogeneous Boltzmann equation. For Kac's program approximating Boltzmann equation, we refer to the recent work \cite{mm} and the references therein. In \cite{mm}, the authors proved the propagation of chaos quantitatively in an abstract framework by proving stability and convergence estimates between linear semigroups. They then applied their results to prove the propagation of chaos of Kac's program in the cases of hard sphere model ($B(v-v_{*}, \cos\theta) = |v-v_{*}|$) and true Maxwell molecules ($B(v-v_{*}, \cos\theta) = b(\cos\theta)$).

As for particle system approximating the Landau equation, we refer to \cite{fg2}  and the references therein. The authors in \cite{fg2} proved quantitatively the propagation of chaos for a $N$-particle continuous drift diffusion process under the cases of Maxwell molecules ($\gamma =0$) and hard potentials ($0 < \gamma \leq 1$).

As one can see from above, the Boltzmann equation corresponds to the limit of jump processes, while the Lanau equation corresponds to the limit of continuous processes. If we are to numerically solve our approximate equation (\ref{approximated}), we need some jump-diffusion processes.  Actually, the method in \cite{mm} is general and robust to deal with mixture of
jump and diffusion processes. As shown to be successful in \cite{mmw},  the authors considered the Boltzmann equation for diffusively excited granular media, used jump-diffusion processes to approximate it, and then proved the  propagation of chaos. The jump part is the Boltzmann operator with an integrable kernel, while the diffusive part is a Laplace operator. We know that the Landau operator behaves like the Laplace operator, except with some compensation to conserve energy.

In the recent work \cite{fg3}, the authors replaced the small collisions by a small diffusion term to approximate the Kac equation without cutoff, and successfully built a stochastic particle system to approximate the solution of the Kac equation without cutoff. The Kac equation is a one-dimensional case of the Boltzmann equation.

Thanks to the above breakthroughs, our approximate equation (\ref{approximated}) has great potential to be solved numerically. In our future work, we will build a particle system based on equation (\ref{approximated}) and prove the  propagation of chaos.

\subsection{Notations and main results}
Let us introduce the function spaces and notations
which we shall use throughout the paper.
\begin{itemize}
\item For
integer $N\geq 0$, we define the Sobolev space
\begin{equation*} H^N =\bigg\{f( v): \sum_{|\alpha | \leq N}\|
\partial _v^{\alpha }f\|_{L^2 }<+\infty\bigg\},
\end{equation*}
where the multi-index $\alpha =(\alpha _1,\alpha _2,\alpha _3)$ with
$|\alpha |=\alpha _1+\alpha _2+\alpha _3$ and $\partial _v^{\alpha } = \partial_{v_{1}}^{\alpha_{1}}\partial_{v_{2}}^{\alpha_{2}}\partial_{v_{3}}^{\alpha_{3}}$.
\item  For   real number $m, l$, we define the weighted Sobolev space
\begin{equation*}
H^{m}_l =\bigg\{f(v):  \|  \langle D \rangle^{m} \langle \cdot \rangle^{l}f
 \|_{L^2}<+\infty\bigg\},
\end{equation*}
where $\langle v \rangle \eqdefa (1+|v|^{2})^{\frac{1}{2}}$, and $a(D)$ is the pesudo-differential operator with symbol $a(\xi)$ defined by
\beno
(a(D)f)(v) \eqdefa \frac{1}{(2\pi)^{3}}\int_{\R^{3}}\int_{\R^{3}} e^{i(v-u)\cdot\xi}a(\xi)f(u)dud\xi.
\eeno
\item We also introduce the standard notations \beno
\|f\|_{L^p_q}=\big(\int_{\R^3} |f(v)|^p\langle v
\rangle^{q p}dv\big)^{\f1{p}},\quad \|f\|_{L\log L}=\int_{\R^3}
|f|\log (1+|f|)dv .\eeno
\item For the ease of notation, let us define a new norm $||\cdot||_{\epsilon,m,l}$ for any $\epsilon, l > 0$ and $m \in \N$ as:
\beno
||f||^{2}_{\epsilon,m,l} \eqdefa ||f||^{2}_{H^{m+s}_{l}} + \epsilon^{2-2s}||f||^{2}_{H^{m+1}_{l}},
\eeno
If $m=0$, we simply write $||\cdot||_{\epsilon,l}$ instead of $||\cdot||_{\epsilon,0,l}$. If $m=l=0$, we simply write $||\cdot||_{\epsilon}$ instead of $||\cdot||_{\epsilon,0}$. Then for any $\epsilon \leq 1,  ||\cdot||_{H^{m+s}_{l}}\leq||\cdot||_{\epsilon,m,l}\leq 2||\cdot||_{H^{m+1}_{l}}$.
\item Let us define the symbol $W^{\epsilon}(\xi)$ by
\beno
W^{\epsilon}(\xi) = \langle \xi \rangle^{s} \textbf{1}_{|\xi|\leq \frac{1}{\epsilon}} + \epsilon^{-s} \textbf{1}_{|\xi| > \frac{1}{\epsilon}},
\eeno
which comes from the coercivity estimate of the cut-off Boltzmann operator $Q^{\epsilon}$.
\item For any $f, g \in L^{2}(\R^{3})$, we denote by $\langle f, g\rangle$ the inner product of $f$ and $g$.
\item By $a\lesssim b$, we mean that there is a uniform constant $C,$ which may be different on different lines, such that $a\leq Cb$. We write $a\sim b$ if both $a\lesssim b$ and $b\lesssim a$.
\end{itemize}
We do not bother to distinguish a function and its value at a point. For example, we do not distinguish weight function $\langle \cdot \rangle^{l}$ and the value $\langle v \rangle^{l}$ it takes at a point $v$.

We recall Young's inequality for use in future. For $a, b \geq 0$ and $p, q > 1$, with $\frac{1}{p} + \frac{1}{q} = 1$, there holds
\begin{eqnarray}\label{youngineq}
a b \leq \frac{a^{p}}{p}  + \frac{b^{q}}{q}.
\end{eqnarray}
As a result, for any $\eta > 0$, we have the basic inequality
\begin{eqnarray}\label{basicineq}
a b \leq \eta a^{p} + (p\eta)^{-\frac{q}{p}}\frac{b^{q}}{q}.
\end{eqnarray}
We also recall the Gronwall's inequality. For any $a , b \in \R $, and a function $y$ defined on $\R_{+}$ satisfying
\begin{eqnarray}\label{assofgron}
\frac{dy}{dt} \leq a + b y(t), \nonumber
\end{eqnarray}
then
\begin{eqnarray}\label{gronwall}
y(t) \leq y(0)e^{bt} + \frac{a}{b}(e^{bt}-1).
\end{eqnarray}
There is also an integral type of Gronwall's inequality. Let $y, \alpha, \beta$  be functions defined on $\R_{+}$. If $\beta$ is nonnegative and for any $t>a\geq0$,  $y$ satisfies
\begin{eqnarray}\label{assofgron2}
y(t)\leq \alpha(t) + \int_{a}^{t}\beta(r)y(r)dr, \nonumber
\end{eqnarray}
then
\begin{eqnarray}\label{gronwall2}
y(t)\leq \alpha(t) + \int_{a}^{t}\alpha(r)\beta(r)\exp{(\int^{t}_{r}\beta(u)du)}dr.
\end{eqnarray}
If, in addition, the function $\alpha$ is non-decreasing, then
\begin{eqnarray}\label{gronwall3}
y(t)\leq \alpha(t)\exp{(\int_{a}^{t}\beta(r)dr)}.
\end{eqnarray}
\medskip

Before stating our main results, let us give the definition of $\phi$ which is related to the weight function:
\begin{equation}\label{functionphi} \left \{ \begin{aligned}
&\phi(0,l) = 2l+5;\\
&\phi(s,l) = \frac{(2l+4)(2+s)-2l}{s};\\
&\phi(1,l) = \max\{\phi(s,x(l)),y(l)\};\\
&\phi(m, l)= \max\{u(m,l),\phi(m-1,z(l)), ~~~~~~~~ m \geq 2,
\end{aligned} \right.
\end{equation}
where
\begin{equation}\label{functionxyzu}
\left \{
\begin{aligned}
&x(l) = \frac{2l+7}{s} - \frac{1-s}{s}(l+\frac{\gamma}{2});\\
&y(l) = \frac{3x(l) - (s+2)l}{1-s};\\
&z(l) = 2l+7+\frac{l+7}{s};\\
&u(m,l) = (m+2)z(l)-(m+1)l.
\end{aligned} \right.
\end{equation}
We begin with the first result concerns the propagation of the moments and smoothness for the solution to our approximate equation.
\begin{thm}\label{main2}  Let $\phi: \N \times \R_{+} \rightarrow \R$ be the function defined as in (\ref{functionphi}). Let $N \in \N$ and $l \geq 0$.  If $f_{0} \in L^{1}_{q} \cap H^{N}_{l}$ with $q \geq \phi(N,l)$,
then \eqref{approximated}  admits a non-negative and unique solution $f^\epsilon$ in $L^\infty([0,\infty]; L^{1}_{q} \cap H^{N}_{l})$ and morevover
there exists a constant $C$, depending only on $||f_{0}||_{L^{1}_{q}}$ and $||f_{0}||_{H^{N}_{l}}$, such that for any $t \geq 0$ and $\epsilon$ small enough,
\begin{eqnarray}
\label{uniforml1}
||f^\epsilon(t)||_{L^{1}_{q}}& \leq &C + ||f_{0}||_{L^{1}_{q}};\\
\label{uniformhn}
||f^\epsilon(t)||^{2}_{H^{N}_{l}} + \int^{t+1}_{t} ||f^\epsilon(r)||^{2}_{\epsilon,N,l+\gamma/2} dr & \leq &C(||f_{0}||_{L^{1}_{q}},||f_{0}||_{H^{N}_{l}}).
\end{eqnarray}
\end{thm}

\begin{rmk}\label{alsotrue}
The result of Theorem  \ref{main2} is also true when $\epsilon=0$, which corresponds to the propagation of moments and smoothness of solution of the original Boltzmann equation (\ref{homb}).
\end{rmk}

The last two theorems describe the error between solutions of the Boltzmann equation and our approximate equation.
\begin{thm}\label{main3}
Let $ l\geq 0$ such that $ (\frac{4}{\pi})^{2l-2s}(l-s) \geq \frac{2^{4-2s} \pi K}{A_{2}}$ and $2l \geq \frac{s}{1-s}(\gamma+2)+\gamma$. Suppose $f_{0} \in L^{1}_{q} \cap H^{5}_{2l+\gamma+12}$ with $q\geq\phi(5,2l+\gamma+12)$.
Let $f$ and $f^{\epsilon}$ be solutions to the Boltzmann equation \eqref{homb} and the approximated equation \eqref{approximated} with the same initial datum $f_{0}$ respectively, then we have for any $t \geq 0$,
\begin{eqnarray}\label{errorl1}
||f(t)-f^{\epsilon}(t)||_{L^{1}_{2l}}\leq C(f_{0},t)\epsilon^{3-2s},
\end{eqnarray}
where $C(f_{0},t)$ is a constant depending only on $||f_{0}||_{L^{1}_{q}}, ||f_{0}||_{H^{5}_{2l+\gamma+12}}$ and time $t$.
\end{thm}

Let us introduce the definition of $\psi$:
\begin{equation}\label{functionpsi} \left \{ \begin{aligned}
&\psi(0,l) = 2l+\gamma+17;\\
&\psi(m, l)= l+\gamma+10, ~~~~~~~~ m \geq 1.
\end{aligned} \right.
\end{equation}
and $\varphi$:
\begin{equation}\label{functionvarphi}
\left \{
\begin{aligned}
&\varphi(0,l) = \phi(5,2l+\gamma+17);\\
&\varphi(m,l) = \max\{\varphi(m-1,z(l)), \rho(m,l)\},~~~~ m \geq 1.
\end{aligned} \right.
\end{equation}
Then we have:
\begin{thm}\label{main4}  Let $N \in \N$ and $l \geq 0$ such that $ (\frac{4}{\pi})^{2l+5-2s}(2l+5-2s) \geq \frac{2^{5-2s} \pi K}{A_{2}}$ and $2l + 5 \geq \frac{s}{1-s}(\gamma+2)+\gamma$. Let $\psi, \varphi: \N \times \R_{+} \rightarrow \R$ be functions defined as in (\ref{functionpsi}) and (\ref{functionvarphi}). Suppose $f_{0} \in L^{1}_{q} \cap H^{N+5}_{\psi(N,l)}$ with $q \geq \varphi(N,l)$.
Let $f$ and $f^{\epsilon}$ be solutions to the Boltzmann equation \eqref{homb} and the approximated equation \eqref{approximated} with the same initial datum $f_{0}$ respectively, then we have for any $t \geq 0$,
\begin{eqnarray}\label{errorl2}
||f(t)-f^{\epsilon}(t)||_{H^{N}_{l}} \leq C(f_{0},t)\epsilon^{3-2s},
\end{eqnarray}
where $C(f_{0},t)$ is a constant depending only on $||f_{0}||_{L^{1}_{q}}, ||f_{0}||_{H^{N+5}_{\psi(N,l)}}$ and time $t$.
\end{thm}

\subsection{Plan of the paper} In section 2, we state three estimates (upper bound, coercivity, commutator) of the operator $M^{\epsilon}$. Section 3 is devoted to the well-posedness theory  of our approximate equation, namely, uniqueness and existence of non-negative solution. In the last section, we prove the high order convergence of solutions between the Boltzmann equation and our approximate equation.

\section{Estimates of the collision operators}
In this section, we state three estimates of the operator $M^{\epsilon}$, as defined in (\ref{approximatedoperator}) which will used frequently in next sections. We begin with upper bound of the collision operator.

\begin{thm}\label{upcb}
Suppose the collision kernel $B$ satisfies the Assumption (A-1)-(A-4), and $Q^{\epsilon}$ is the collision operator associated to the collision kernel $B^{\epsilon}$. Let $w_{1}, w_{2} \in \R$ with $w_{1}+w_{2} \geq \gamma + 2$, $a_{1}, a_{2} \geq 0$ with $a_{1}+a_{2} = 2s$ and $b_{1}, b_{2} \geq 0$ with $b_{1}+b_{2} = 2$. Then for smooth functions $g, h$ and $f$, the following
estimate holds uniformly with respect to $\epsilon$:
\begin{eqnarray}\label{upcboltz}
\quad \quad \quad   |\langle M^{\epsilon}(g,h), f \rangle| \lesssim ||g||_{L^{1}_{\gamma+2+(-w_{1})^{+}+(-w_{2})^{+}}} (||h||_{H^{a_{1}}_{w_{1}}} ||f||_{H^{a_{2}}_{w_{2}}}+ \epsilon^{2-2s}||h||_{H^{b_{1}}_{w_{1}}} ||f||_{H^{b_{2}}_{w_{2}}}).
\end{eqnarray}
\end{thm}
\begin{proof}
For the cut-off Boltzmann operator $Q^{\epsilon}$,  as in \cite{he2}, for any $w_{1}, w_{2} \in \R$ with $w_{1}+w_{2} \geq \gamma + 2$, there holds
\begin{eqnarray}\label{upcboltz1}
|\langle Q^{\epsilon}(g,h), f \rangle| \lesssim ||g||_{L^{1}_{\gamma+2s+(-w_{1})^{+}+(-w_{2})^{+}}} ||h||_{H^{a_{1}}_{w_{1}}} ||f||_{H^{a_{2}}_{w_{2}}}.
\end{eqnarray}
Again from \cite{he2}, we have
\begin{eqnarray}\label{uplandau}
|\langle Q_{L}(g,h), f \rangle| \lesssim ||g||_{L^{1}_{\gamma+2+(-w_{1})^{+}+(-w_{2})^{+}}} ||h||_{H^{b_{1}}_{w_{1}}} ||f||_{H^{b_{2}}_{w_{2}}}.
\end{eqnarray}
Patching together the above two estimates, the estimate (\ref{upcboltz}) follows accordingly.
\end{proof}

We now turn to coercivity estimate of the operator.
\begin{thm}\label{coercb}
Suppose the collision kernel $B$ satisfies the Assumption (A-1)-(A-4), and $Q^{\epsilon}$ is the collision operator associated to the collision kernel $B^{\epsilon}$. Suppose function $g$ is nonnegative and satisfies
\begin{eqnarray}\label{assong}
||g||_{L^{1}_{2}} + ||g||_{LlogL} < \infty,
\end{eqnarray}
then there exists constants $C_{1}(g)$ and $C_{2}(g)$ depending only on $||g||_{L^{1}_{1}}$ and $||g||_{LlogL}$ such that
\begin{eqnarray}\label{coercivityboltz}
-\langle M^{\epsilon}(g,f), f \rangle \geq C_{1}(g)||f||^{2}_{\epsilon,\gamma/2} - C_{2}(g)||f||^{2}_{L^{2}_{\gamma/2}}.
\end{eqnarray}
\end{thm}
\begin{proof}
For the cut-off Boltzmann operator $Q^{\epsilon}$,  with a similar argument as in \cite{advw}, one has
\beno
-\langle Q^{\epsilon}(g,f), f \rangle \geq C_{1}(g)||W^{\epsilon}(D)f||^{2}_{L^{2}_{\gamma/2}} - C_{2}(g)||f||^{2}_{L^{2}_{\gamma/2}}.
\eeno
For the Landau operator $Q_{L}$, by \cite{dv1}, there holds
\begin{eqnarray}\label{coerlandau}
-\langle Q_{L}(g,f), f \rangle_{v} \geq C_{1}(g)||f||^{2}_{H^{1}_{\gamma/2}} - C_{2}(g)||f||^{2}_{L^{2}_{\gamma/2}}.
\end{eqnarray}
The coercivity estimate (\ref{coercivityboltz}) follows by noting that
\beno
||W^{\epsilon}(D)f||^{2}_{L^{2}_{\gamma/2}} + \epsilon^{2-2s}||f||^{2}_{H^{1}_{\gamma/2}} \sim ||f||^{2}_{\epsilon,\gamma/2}.
\eeno
\end{proof}

In the last, we move to commutator estimates. We first give the commutator estimate of the cut-off Boltzmann operator $Q^{\epsilon}$ as a lemma.
\begin{lem}\label{commcb}  Suppose the collision kernel $B$ satisfies the Assumption (A-1)-(A-4), and $Q^{\epsilon}$ is the collision operator associated to the collision kernel $B^{\epsilon}$. Let $N_{2}, N_{3}\in \R$ and $l \geq 0$ with $N_{2}+N_{3} \geq l + \gamma $, and let $N_{1} = |N_{2}|+|N_{3}| + \max\{|l-1|,|l-2|\}$. Then for smooth functions $g, h$ and $f$, the following
estimate holds uniformly with respect to $\epsilon$:
\begin{eqnarray}\label{commcboltz1}
|\langle Q^{\epsilon}(g, h \langle v \rangle^{l}) - Q^{\epsilon}(g,h)\langle v \rangle^{l}, f \rangle| \lesssim ||g||_{L^{1}_{N_{1}}} ||h||_{H^{s}_{N_{2}}} ||f||_{L^{2}_{N_{3}}}.
\end{eqnarray}
\end{lem}
\begin{proof}
One may refer to \cite{chenhe} for a proof.
\end{proof}
The next lemma is the commutator estimate of the Landau operator $Q_{L}$.
\begin{lem}\label{commld}
Let $N_{2}, N_{3}\in \R$ and $l \geq 0$ with $N_{2}+N_{3} \geq l + \gamma $.
Then for smooth functions $g, h$ and $f$, the following estimate holds true:
\begin{eqnarray}\label{commlandau}
|\langle Q_{L}(g, h \langle v \rangle^{l}) - Q_{L}(g,h)\langle v \rangle^{l}, f \rangle| \leq \Lambda C(l)||g||_{L^{1}_{\gamma+3}} ||h||_{H^{1}_{N_{2}}} ||f||_{L^{2}_{N_{3}}},
\end{eqnarray}
where $C(l) = \max\{2l^{2}+12l, 20l-2l^{2}\}$.
\end{lem}
\begin{proof}  We define as usual the following quantities in 3-dimension:
\beno
b_{i}(z) = \sum_{j=1}^{3} \partial_{j}a_{ij}(z) = -2\Lambda|z|^{\gamma}z_{i},~~~~~ c(z)= \sum_{i,j=1}^{3}\partial_{ij}a_{ij}(z) = -2\Lambda(\gamma+3)|z|^{\gamma}.
\eeno
Hence the Landau operator $Q_{L}$ can be rewritten as:
\beno
Q_{L}(g,h) = \sum_{i,j=1}^{3}(a_{ij}*g)\partial_{ij}h - (c*g)h = \sum_{i=1}^{3}\partial_{i}[\sum_{j=1}^{3}(a_{ij}*g)\partial_{j}h-(b_{i}*g)h].
\eeno
Then we have
\beno
D(g,h,f;l) \eqdefa \langle Q_{L}(g, h \langle v \rangle^{l}) - Q_{L}(g,h)\langle v \rangle^{l}, f \rangle = \sum_{i, j=1}^{3} \langle a_{ij}*g, f\partial_{ij}(h \langle v \rangle^{l}) - f \langle v \rangle^{l}\partial_{ij}h  \rangle.
\eeno
It is easy to check
\beno
\partial_{ij}(h \langle v \rangle^{l}) - \langle v \rangle^{l}\partial_{ij}h   = l\langle v \rangle^{l-2}(v_{i}\partial_{j}h + v_{j}\partial_{i}h) +
l\langle v \rangle^{l-2}[(l-2)\frac{v_{i} v_{j}}{\langle v \rangle^{2}} + \delta_{ij}]h.
\eeno
Thus we have

\beno
D(g,h,f;l) &=& l \int_{\R^{6}}g_{*}f \langle v \rangle^{l-2}[\sum_{i,j}a_{ij}(v-v_{*})(v_{i}\partial_{j}h + v_{j}\partial_{i}h)]dv dv_{*} \\
&& + l (l-2) \int_{\R^{6}}g_{*} h f \langle v \rangle^{l-2}\frac{\sum_{i,j}a_{ij}(v-v_{*})v_{i}v_{j}}{\langle v \rangle^{2}}dv dv_{*}\\
&& + l \int_{\R^{6}}g_{*} h f \langle v \rangle^{l-2}  \sum_{i}a_{ii}(v-v_{*})dv dv_{*}.
\eeno
Considering the following facts
\beno
\sum_{i,j=1}^{3}a_{ij}(v-v_{*})v_{i}\partial_{j}h = \sum_{i,j=1}^{3}a_{ij}(v-v_{*})v_{j}\partial_{i}h = (\nabla h)^{T}a(v-v_{*})v = (\nabla h)^{T}a(v-v_{*})v_{*},
\eeno
and
\beno
\sum_{i,j=1}^{3}a_{ij}v_{i}v_{j} = \Lambda|v-v_{*}|^{\gamma}(|v|^{2}|v_{*}|^{2}-(v \cdot v_{*})^{2}),
\eeno
and
\beno
\sum_{i}a_{ii} = 2\Lambda |v-v_{*}|^{\gamma+2},
\eeno
we arrive at
\beno
D(g,h,f;l) &=& 2l \int_{\R^{6}}g_{*}f \langle v \rangle^{l-2} (\nabla h)^{T}a(v-v_{*})v_{*} dv dv_{*}
\\ && + \Lambda l (l-2) \int_{\R^{6}}g_{*} h f \langle v \rangle^{l-2}\frac{|v-v_{*}|^{\gamma}(|v|^{2}|v_{*}|^{2}-(v \cdot v_{*})^{2})}{\langle v \rangle^{2}}dv dv_{*}
\\&& + 2\Lambda l \int_{\R^{6}}g_{*} h f \langle v \rangle^{l-2}  |v-v_{*}|^{\gamma+2} dv dv_{*}
\\&\eqdefa&  \mathfrak{I}_{1} + \mathfrak{I}_{2} + \mathfrak{I}_{3}.
\eeno
Thanks to \beno |a(v-v_{*})v_{*}|\leq 4\Lambda\langle v_{*} \rangle^{\gamma+3}\langle v \rangle^{\gamma+2}, \eeno
we have \beno |\mathfrak{I}_{1}| \leq 8\Lambda l||g||_{L^{1}_{\gamma+3}}||h||_{H^{1}_{N_{2}}} ||f||_{L^{2}_{N_{3}}},\eeno
provided $N_{2}+N_{3} \geq l+\gamma$.
Similarly, if $N_{2}+N_{3} \geq l+\gamma$, there holds
\beno |\mathfrak{I}_{3}| \leq 8\Lambda l||g||_{L^{1}_{\gamma+2}}||h||_{L^{2}_{N_{2}}} ||f||_{L^{2}_{N_{3}}}. \eeno
With the help of the fact
\beno
\frac{|v-v_{*}|^{\gamma}(|v|^{2}|v_{*}|^{2}-(v \cdot v_{*})^{2})}{\langle v \rangle^{2}} \leq 2\langle v_{*} \rangle^{\gamma+2}\langle v \rangle^{\gamma},
\eeno
we have
\beno
|\mathfrak{I}_{2}| \leq 2\Lambda l|l-2| ||g||_{L^{1}_{\gamma+2}}||h||_{L^{2}_{N_{2}}} ||f||_{L^{2}_{N_{3}}},
\eeno
provided $N_{2}+N_{3} \geq l-2+\gamma$. Patching together the above estimates, if $N_{2}+N_{3} \geq l+\gamma$, we have
\beno
|D(g,h,f;l)| \leq  \Lambda \max\{2l^{2}+12l, 20l-2l^{2}\} ||g||_{L^{1}_{\gamma+3}}||h||_{H^{1}_{N_{2}}} ||f||_{L^{2}_{N_{3}}}.
\eeno
\end{proof}
In the end of this section, we state the commutator estimate of the operator $M^{\epsilon}$.
\begin{thm}\label{commcb1}  Suppose the collision kernel $B$ satisfies the Assumption (A-1)-(A-4), and $Q^{\epsilon}$ is the collision operator associated to the collision kernel $B^{\epsilon}$. Let $N_{2}, N_{3}\in \R$ and $l \geq 0$ with $N_{2}+N_{3} \geq l + \gamma $, and let $N_{1} = \max\{|N_{2}|+|N_{3}| + \max\{|l-1|,|l-2|\}, \gamma+3\}$. Then for smooth functions $g, h$ and $f$, the following
estimate holds uniformly with respect to $\epsilon$:
\begin{eqnarray}\label{commcboltz}
|\langle M^{\epsilon}(g, h \langle v \rangle^{l}) - M^{\epsilon}(g,h)\langle v \rangle^{l}, f \rangle| \lesssim ||g||_{L^{1}_{N_{1}}} (||h||_{H^{s}_{N_{2}}} + \epsilon^{2-2s}||h||_{H^{1}_{N_{2}}})||f||_{L^{2}_{N_{3}}}.
\end{eqnarray}
\end{thm}
\begin{proof}
The commutator estimate (\ref{commcboltz}) follows from lemma \ref{commcb} and \ref{commld}.
\end{proof}

\section{Well-posedness  for approximate equation \eqref{approximated}: existence and uniqueness}
In this section, we will show that  \eqref{approximated} admits a non-negative, unique and smooth solution if the initial data is smooth.  To do that, we separate the proof into three steps.  In the first step, we prove that the linear equation to \eqref{approximated} admits a non-negative and smooth solution. Then in the next step, by using Picard iteration scheme, we get the well-posedness result. In the final step, we   improve the well-posedness result by applying the symmetric property of the collision operators.

 \subsection{Well-posedness of linear equation to \eqref{approximated}}
 Throughout this subsection, $\epsilon>0$ is a fixed but small enough number. In the following, we construct a non-negative solution to the linear equation:
\begin{equation}\label{forpositive2} \left\{ \begin{aligned}
&\partial_{t}f = Q^{\epsilon}(g,f) + \epsilon^{2-2s}Q_{L}(g,f)\\
&f|_{t=0} = f_{0}.
 \end{aligned} \right.
\end{equation}

 Let us define two operators:
\beno
Q^{\epsilon +}(g,h) \eqdefa \int_{\R^3}\int_{\SS^{2}}B^{\epsilon}(v-v_*,\sigma)g'_*h'd\sigma dv_*,
\eeno
\beno
Q^{\epsilon -}(g,h) \eqdefa \int_{\R^3}\int_{\SS^{2}}B^{\epsilon}(v-v_*,\sigma)g_*hd\sigma dv_* = \mathcal{L}(g)h.
\eeno
Then we have $Q^{\epsilon} = Q^{\epsilon +} - Q^{\epsilon -}$, so we call $Q^{\epsilon +}$ the gain operator and $Q^{\epsilon -}$ the loss operator.

We first give a proposition, which shall be used in both the current section and the next section.
\begin{prop}\label{lemma1}Let $p \geq 2, n = \frac{v-v_{*}}{|v-v_{*}|}, u = \frac{v+v_{*}}{|v+v_{*}|}, j = \frac{u-(u \cdot n)n}{|u-(u \cdot n)n|},
h = \sqrt{|v|^{2}|v_{*}|^{2}-(v\cdot v_{*})^{2}}$, and $E(\theta) = \langle v \rangle^{2}\cos^{2}\frac{\theta}{2}+\langle v_{*} \rangle^{2}\sin^{2}\frac{\theta}{2}$. Suppose $\omega$ is the vector such that $\sigma = \cos\theta n + \sin\theta \omega$, then there holds
\begin{eqnarray}\label{primededuct}
\langle v^{\prime} \rangle^{2p} - \langle v \rangle^{2p} &\leq& -\langle v \rangle^{2p}(1-\cos^{2p}\frac{\theta}{2})+\langle v_{*} \rangle^{2p}\sin^{2p}\frac{\theta}{2}+ p(E(\theta))^{p-1}h(j\cdot\omega)\sin\theta\\
&&+(\frac{1}{2}\max\{2^{p-3},1\}p(p-1)+2^{p-1})\langle v_{*} \rangle^{2p-2}\langle v \rangle^{2p-2}\sin^{2}\theta. \nonumber
\end{eqnarray}
\end{prop}
\begin{proof} It is easy to check $\langle v^{\prime} \rangle^{2} = E(\theta)+h(j\cdot\omega)\sin\theta$.
By Taylor expansion, we have
\beno
\langle v^{\prime} \rangle^{2p} &=& (E(\theta))^{p}+p(E(\theta))^{p-1}h(j\cdot\omega)\sin\theta \\
&&+p(p-1)(h(j\cdot\omega)\sin\theta)^{2}\int^{1}_{0}(1-\kappa)(E(\theta)+\kappa h(j\cdot\omega)\sin\theta)^{p-2}d\kappa. \\
&\eqdefa&\mathfrak{M}_1+\mathfrak{M}_2+\mathfrak{M}_3.
\eeno
For the last term $\mathfrak{M}_3$, we have for any $\kappa \in [0, 1]$:
\beno
E(\theta)+\kappa h(j\cdot\omega)\sin\theta &\leq& (\langle v \rangle^{2}+\langle v_{*} \rangle^{2})(1 - \frac{1-\kappa}{4}\sin^{2}\theta) \\
&\leq& \langle v \rangle^{2}+\langle v_{*} \rangle^{2}.
\eeno
Together with $h^{2}\leq \langle v \rangle^{2} \langle v_{*} \rangle^{2}$, we arrive at
\beno
\mathfrak{M}_3 &\leq& p(p-1)\langle v \rangle^{2} \langle v_{*} \rangle^{2}(\langle v \rangle^{2}+\langle v_{*} \rangle^{2})^{p-2}\sin^{2}\theta
\int^{1}_{0}(1-\kappa)d\kappa \\
&\leq&\frac{1}{2}\max\{2^{p-3},1\}p(p-1)\langle v \rangle^{2p-2}\langle v_{*} \rangle^{2p-2}\sin^{2}\theta.  \nonumber
\eeno
For the term $\mathfrak{M}_1$, we have
\begin{eqnarray}\label{firstterm}
&&(\langle v \rangle^{2}\cos^{2}\frac{\theta}{2}+\langle v_{*} \rangle^{2}\sin^{2}\frac{\theta}{2})^{p}
\\&\leq&
\sum^{k_{p}}_{k=1}\binom{p}{k}\{\langle v \rangle^{2k}\cos^{2k}\frac{\theta}{2}\langle v_{*} \rangle^{2(p-k)}\sin^{2(p-k)}\frac{\theta}{2}+\langle v \rangle^{2(p-k)}\cos^{2(p-k)}\frac{\theta}{2}\langle v_{*} \rangle^{2k}\sin^{2k}\frac{\theta}{2}\} \nonumber \\
&\leq& \langle v \rangle^{2p}\cos^{2p}\frac{\theta}{2}+\langle v_{*} \rangle^{2p}\sin^{2p}\frac{\theta}{2}+
2^{p}\langle v \rangle^{2p-2}\langle v_{*} \rangle^{2p-2}\sin^{2}\frac{\theta}{2}.\nonumber
\end{eqnarray}
Combining $\mathfrak{M}_1, \mathfrak{M}_2, \mathfrak{M}_3$, we arrive at (\ref{primededuct}).
\end{proof}

We begin with an equation which shall be used to construct solution to  the linear equation  to \eqref{forpositive2}. 

\begin{lem}\label{lemmapositive1} Let $g,h \geq  0$ be smooth functions. Suppose $f^{\epsilon}$ is the solution to the following equation
\begin{equation}\label{forpositive1} \left\{ \begin{aligned}
&\partial_{t}f = Q^{\epsilon +}(g,h) - Q^{\epsilon -}(g,f) + \epsilon^{2-2s}Q_{L}(g,f)\\
&f|_{t=0} = f_{0} \geq 0.
 \end{aligned} \right.
\end{equation}
Then $f^{\epsilon}(t) \geq 0$ for any $t\geq0$.
\end{lem}
\begin{proof} Denote $f^{\epsilon}_{-} = \min \{0,f^{\epsilon}\} \leq 0$, then we have $f^{\epsilon}_{-}|_{t=0} = 0$, and \beno
\frac{d}{dt}(\frac{1}{2}||f^{\epsilon}_{-}||^{2}_{L^{2}}) + \int_{\R^3}\mathcal{L}(g)(f^{\epsilon}_{-})^{2}dv = \int_{\R^3}Q^{\epsilon +}(g,h)f^{\epsilon}_{-}dv + \epsilon^{2-2s}\langle Q_{L}(g,f^{\epsilon}), f^{\epsilon}_{-} \rangle.
\eeno
Since $g,h \geq  0$ and $f^{\epsilon}_{-} \leq 0$, it is clear that
\beno
 \int_{\R^3}Q^{\epsilon +}(g,h)f^{\epsilon}_{-}dv \leq 0.
\eeno
By the definition of $Q_{L}$, we have
\beno
\langle Q_{L}(g,f^{\epsilon}), f^{\epsilon}_{-} \rangle &=& -\int_{\R^6}g_{*}(\nabla f^{\epsilon}_{-})^{T}a(v-v_{*})\nabla f^{\epsilon}_{-}dvdv_{*} \\&&+
\Lambda(\gamma+3)\int_{\R^6}|v-v_{*}|^{\gamma}g_{*}(f^{\epsilon}_{-})^{2}dvdv_{*} \\
&\eqdefa& \mathfrak{I}_{1} + \mathfrak{I}_{2}.
\eeno
Since $a$ is a positive semi-definite matrix, we have $\mathfrak{I}_{1} \leq 0$. By assumption (A-2), there holds $\int_{\SS^{2}} b^{\epsilon}(\cos\theta)d\sigma \sim \frac{\epsilon^{-2s}}{s}$. Therefore, there exists $\epsilon_{*} > 0$ such that, for any $0 < \epsilon \leq \epsilon_{*}$,
\beno
 \epsilon^{2-2s}\mathfrak{I}_{2} \leq \frac{1}{2}\int_{\R^3}\mathcal{L}(g)(f^{\epsilon}_{-})^{2}dv.
\eeno
Finally, we arrive at
\beno
\frac{d}{dt}(\frac{1}{2}||f^{\epsilon}_{-}||^{2}_{L^{2}}) + \frac{1}{2}\int_{\R^3}\mathcal{L}(g)(f^{\epsilon}_{-})^{2}dv \leq 0.
\eeno
Thus $||f^{\epsilon}_{-}(t)||_{L^{2}} = 0$ for any $t \geq 0$, which implies $f^{\epsilon}(t) \geq 0$ for any $t\geq0$.
\end{proof}

Now we are ready to construct a  solution to the linear equation  \eqref{approximated}.

\begin{lem}\label{lemmapositive2} Let $l \geq 4, T>0$ be real numbers. Suppose the non-negative datum $f_{0} \in  H^{5}_{l+3\gamma/2+10} \cap L^{1}_{l+5\gamma/2+16}$ with $||f_{0}||_{L^{1}} > 0$.  Suppose $g(t,v)$ is a non-negative function satisfying
\beno
M = \sup_{0\leq t \leq T}||g(t)||_{H^{5}_{2l+3\gamma+22} \cap L^{1}_{l+5\gamma/2+16}} + \int_{0}^{T}||g(t)||_{L^{1}_{l+7\gamma/2+16}}dt < \infty~~\text{and}~~ m = \inf_{0\leq t \leq T}||g(t)||_{L^{1}} > 0,
\eeno
then \eqref{forpositive2}
admits a unique non-negative solution $f$ in $L^{\infty}([0,T]; L^{1}_{l+5\gamma/2+16}\cap H^5_{l+\gamma+10})\cap L^{1}([0,T]; L^{1}_{l+7\gamma/2+16})$.
\end{lem}
\begin{proof} Define a sequence of approximate solutions $\{f^{n}\}_{n \in \N}$ by
\begin{equation}\label{chauchysequence1} \left\{ \begin{aligned}
&f^{0}(t) = f_{0}, ~~ \text{for any } t \geq 0 ; \\
&\partial_{t}f^{n} = Q^{\epsilon +}(g,f^{n-1}) - Q^{\epsilon -}(g,f^{n}) + \epsilon^{2-2s}Q_{L}(g,f^{n}), ~~ n \geq 1\\
&f^{n}|_{t=0} = f_{0}.
 \end{aligned} \right.
\end{equation}
According to the previous lemma, we have $f^{n} \geq 0$.\\
\noindent {\textit {Step 1}:} (Uniform Upper Bound)\\
\noindent {\textit {Step 1.1}:} (Uniform Upper Bound in $L^{1}_{l}$)\\
In this step, we shall use the energy method to get the uniform
upper bound of $L^{1}_{l}$ norm of $\{f^{n}\}_{n}$ with respect to $n$.
Applying the basic inequality (\ref{basicineq}), for any $\eta>0$, there holds
\beno
|v-v_{*}|^{\gamma} \leq (|v|^{2}+2|v||v_{*}|+|v_{*}|^{2})^{\frac{\gamma}{2}} &\leq& ((1+\eta)|v|^{2}+(1+\frac{1}{\eta})|v_{*}|^{2})^{\frac{\gamma}{2}} \\
&\leq&(1+\eta)^{\frac{\gamma}{2}} \langle v \rangle^{\gamma} + (1+\frac{1}{\eta})^{\frac{\gamma}{2}}\langle v_{*} \rangle^{\gamma}.
\eeno
Also one has
\begin{eqnarray}\label{roughonvprime}
\langle v^{\prime} \rangle^{l} \leq (1+|v|^{2}+|v_{*}|^{2})^{\frac{l}{2}} &\leq& \langle v_{*} \rangle^{l} + 2^{l}\langle v \rangle^{l-2}\langle v_{*} \rangle^{l-2} + \langle v \rangle^{l}.
\end{eqnarray}
Thanks to the above two facts, we obtain
\begin{eqnarray}\label{fnl1lfirstterm}
\int_{\R^{3}}Q^{\epsilon+}(g,f^{n-1})(v)\langle v \rangle^{l}dv &=& \int_{\R^{6}}\int_{\SS^{2}}b^{\epsilon}(\cos\theta)|v-v_{*}|^{\gamma}g_{*}f^{n-1}\langle v^{\prime} \rangle^{l}dvdv_{*}d\sigma \\
&\leq&(1+\eta)^{\frac{\gamma}{2}}A^{\epsilon}||g||_{L^{1}}||f^{n-1}||_{L^{1}_{l+\gamma}} \nonumber\\&&+
(1+\frac{1}{\eta})^{\frac{\gamma}{2}}A^{\epsilon}||g||_{L^{1}_{l+\gamma}}||f^{n-1}||_{L^{1}} \nonumber\\
&&+C(l,\gamma,\eta)A^{\epsilon}||g||_{L^{1}_{l}}||f^{n-1}||_{L^{1}_{l}}, \nonumber
\end{eqnarray}
where $A^{\epsilon} = \int_{\SS^{2}}b^{\epsilon}(\cos\theta)d\sigma$.
It is easy to check
\beno
\langle v \rangle^{\gamma} &=& (1+|v-v_{*}+v_{*}|^{2})^{\frac{\gamma}{2}} \leq (1+(1+\frac{1}{\eta})|v_{*}|^{2}+(1+\eta)|v-v_{*}|^{2})^{\frac{\gamma}{2}} \\
&\leq& (1+\frac{1}{\eta})^{\gamma/2}\langle v_{*} \rangle^{\gamma}+(1+\eta)^{\gamma/2}|v-v_{*}|^{\gamma}.
\eeno
That is, for any $\eta>0$, there holds
\begin{eqnarray}\label{usfulinlowerbound}
|v-v_{*}|^{\gamma} \geq \frac{\langle v \rangle^{\gamma}}{(1+\eta)^{\gamma/2}}- \eta^{-\gamma/2}\langle v_{*} \rangle^{\gamma}.
\end{eqnarray}
Then we obtain
\begin{eqnarray}\label{fnl1lsecondterm}
~~~~~~~~\int_{\R^{3}}Q^{\epsilon-}(g,f^{n})(v)\langle v \rangle^{l}dv \geq \frac{A^{\epsilon}}{(1+\eta)^{\gamma/2}}||g||_{L^{1}}||f^{n}||_{L^{1}_{l+\gamma}} - \eta^{-\gamma/2}A^{\epsilon}||g||_{L^{1}_{\gamma}}||f^{n}||_{L^{1}_{l}}.
\end{eqnarray}
For the Landau operator, referring to \cite{dv1}, there holds
\begin{eqnarray}\label{fnl1lthirdterm}
\int_{\R^{3}}Q_{L}(g,f^{n})(v)\langle v \rangle^{l}dv &\leq& l \Lambda \int_{\R^{3}}g_{*}f^{n}|v-v_{*}|^{\gamma}\langle v \rangle^{l-2}(-2|v|^{2}+l|v_{*}|^{2})dvdv_{*} \\
&\leq& -l\Lambda||g||_{L^{1}}||f^{n}||_{L^{1}_{l+\gamma}} + (4l+2)l\Lambda||g||_{L^{1}_{4}}||f^{n}||_{L^{1}_{l}}. \nonumber
\end{eqnarray}
Patching together the above estimates, we arrive at
\beno
\frac{d}{dt}||f^{n}||_{L^{1}_{l}} &\leq& - (\frac{A^{\epsilon}}{(1+\eta)^{\gamma/2}}+\epsilon^{2-2s}l\Lambda)||g||_{L^{1}}||f^{n}||_{L^{1}_{l+\gamma}}+
(1+\eta)^{\frac{\gamma}{2}}A^{\epsilon}||g||_{L^{1}}||f^{n-1}||_{L^{1}_{l+\gamma}} \\ &&+(1+\frac{1}{\eta})^{\frac{\gamma}{2}}A^{\epsilon}||g||_{L^{1}_{l+\gamma}}||f^{n-1}||_{L^{1}} +C(l,\gamma,\eta)A^{\epsilon}||g||_{L^{1}_{l}}||f^{n-1}||_{L^{1}_{l}} \\
&&+\eta^{-\gamma/2}A^{\epsilon}||g||_{L^{1}_{\gamma}}||f^{n}||_{L^{1}_{l}}+\epsilon^{2-2s}(4l+2)l\Lambda||g||_{L^{1}_{4}}||f^{n}||_{L^{1}_{l}}
\eeno
Observing that
\beno
\lim_{\eta \downarrow 0} \{(1+\eta)^{\frac{\gamma}{2}}A^{\epsilon} - \frac{A^{\epsilon}}{(1+\eta)^{\gamma/2}}\} =0,
\eeno
thus we can take an $\eta>0$ small enough such that,
\beno
(1+\eta)^{\frac{\gamma}{2}}A^{\epsilon} \leq \frac{A^{\epsilon}}{(1+\eta)^{\gamma/2}}+\frac{1}{2}\epsilon^{2-2s}l\Lambda.
\eeno
With such a small $\eta$, let us denote $a = \frac{A^{\epsilon}}{(1+\eta)^{\gamma/2}}+\frac{1}{2}\epsilon^{2-2s}l\Lambda,~ \delta = \frac{1}{2}\epsilon^{2-2s}l\Lambda,~ K_{1} = (1+\frac{1}{\eta})^{\frac{\gamma}{2}}A^{\epsilon},~ K_{2} = \sup_{0 \leq s \leq T}C(l,\gamma,\eta)A^{\epsilon}||g(s)||_{L^{1}_{l}},~ K_{3} = \sup_{0 \leq s \leq T}\{\eta^{-\gamma/2}A^{\epsilon}||g(s)||_{L^{1}_{\gamma}} + \epsilon^{2-2s}(4l+2)l\Lambda||g(s)||_{L^{1}_{4}}\}$.
Therefore, we arrive at a neater inequality on the interval $[0, T]$,
\beno
\frac{d}{dt}||f^{n}||_{L^{1}_{l}}  + (a+\delta)||g||_{L^{1}}||f^{n}||_{L^{1}_{l+\gamma}}&\leq&
a||g||_{L^{1}}||f^{n-1}||_{L^{1}_{l+\gamma}} + (K_{1}||g||_{L^{1}_{l+\gamma}}+K_{2})||f^{n-1}||_{L^{1}_{l}} \\ &&+ K_{3}||f^{n}||_{L^{1}_{l}}.
\eeno
By defining $y^{n}(t) = e^{-K_{3}t}||f^{n}(t)||_{L^{1}_{l}}$ and $x^{n}(t) = \int_{0}^{t}e^{-K_{3}s}||g(s)||_{L^{1}}||f^{n}(s)||_{L^{1}_{l+\gamma}}ds$ for any $0\leq t\leq T$ and $n \geq 0$, we derive that
\beno
y^{n}(t)  + (a+\delta)x^{n}(t)&\leq& ||f_{0}||_{L^{1}_{l}} + a x^{n-1}(t) +\int_{0}^{t}(K_{1}||g(s)||_{L^{1}_{l+\gamma}}+K_{2})y^{n-1}(s)ds.
\eeno
Now denote $S^{n}(t) = \sum_{i = 0}^{n} (\frac{a}{a+\delta})^{i} y^{n-i}(t)$ for $n \geq 0$, by recursive derivation
and noting that $y^{0}(t) \leq ||f_{0}||_{L^{1}_{l}}$ and $x^{0}(t) \leq M\frac{1-e^{-K_{3}t}}{K_{3}}||f_{0}||_{L^{1}_{l+\gamma}}$, we obtain
\beno
S^{n}(t)  + (a+\delta)x^{n}(t)&\leq& \sum_{i=0}^{n-1}(\frac{a}{a+\delta})^{i}||f_{0}||_{L^{1}_{l}} +(\frac{a}{a+\delta})^{n}y^{0}(t) + (\frac{a}{a+\delta})^{n-1}a x^{0}(t) \\&&+\int_{0}^{t}(K_{1}||g(t_{n-1})||_{L^{1}_{l+\gamma}}+K_{2})S^{n-1}(t_{n-1})dt_{n-1}\\
&\leq& (\frac{a}{\delta}+1)||f_{0}||_{L^{1}_{l}} + a M\frac{1-e^{-K_{3}t}}{K_{3}}||f_{0}||_{L^{1}_{l+\gamma}}(\frac{a}{a+\delta})^{n-1} \\
&&+ \int_{0}^{t}(K_{1}||g(t_{n-1})||_{L^{1}_{l+\gamma}}+K_{2})S^{n-1}(t_{n-1})dt_{n-1}.
\eeno
By further recursive derivation, we have
\beno
&&S^{n}(t)  + (a+\delta)x^{n}(t)\\
&\leq& (\frac{a}{\delta}+1)||f_{0}||_{L^{1}_{l}}\sum_{i=0}^{n-1} \frac{(\int_{0}^{t}(K_{1}||g(s)||_{L^{1}_{l+\gamma}}+K_{2})ds)^{i}}{i!} \\
&&\quad + a M \frac{1-e^{-K_{3}t}}{K_{3}}||f_{0}||_{L^{1}_{l+\gamma}}(\frac{a}{a+\delta})^{n-1} \sum_{i=0}^{n-1} \frac{(\frac{a+\delta}{a}\int_{0}^{t}(K_{1}||g(s)||_{L^{1}_{l+\gamma}}+K_{2})ds)^{i}}{i!}\\
&& \quad + \int_{0}^{t}\int_{0}^{t_{n-1}}\cdots \int_{0}^{t_{1}}S^{0}(t_{0})\prod_{i=0}^{n-1}(K_{1}||g(t_{i})||_{L^{1}_{l+\gamma}}+K_{2})dt_{n-1}dt_{n-2} \cdots dt_{0} \\
&&\leq (\frac{a}{\delta}+1)||f_{0}||_{L^{1}_{l}} \exp(\int_{0}^{t}(K_{1}||g(s)||_{L^{1}_{l+\gamma}}+K_{2})ds)\\
&&\quad+  a M \frac{1-e^{-K_{3}t}}{K_{3}}||f_{0}||_{L^{1}_{l+\gamma}}(\frac{a}{a+\delta})^{n-1} \exp(\frac{a+\delta}{a}\int_{0}^{t}(K_{1}||g(s)||_{L^{1}_{l+\gamma}}+K_{2})ds).
\eeno
Noting that
\beno
 ||f^{n}(t)||_{L^{1}_{l}}  \leq e^{K_{3}t}S^{n}(t),
\eeno
and
\beno
 \int_{0}^{t}||f^{n}(s)||_{L^{1}_{l+\gamma}} ds \leq  m^{-1}  e^{K_{3}t}x^{n}(t),
\eeno
and recalling the definition of constants $K_{1}, K_{2}, K_{3}$,  we obtain
\begin{eqnarray}\label{fnl1luniform}
&&\sup_{n}(||f^{n}(t)||_{L^{1}_{l}} + \int_{0}^{t}||f^{n}(s)||_{L^{1}_{l+\gamma}} ds) \\
&\leq& C(||f_{0}||_{L^{1}_{l+\gamma}},t, \sup_{0 \leq s \leq t} ||g(s)||_{L^{1}_{l}}, \int_{0}^{t}||g(s)||_{L^{1}_{l+\gamma}}ds). \nonumber
\end{eqnarray}
\noindent {\textit {Step 1.2}:} (Uniform Upper Bound in $L^{2}_{l}$)\\
In this step, we show the  uniform upper bound of $L^{2}_{l}$ norm of $\{f^{n}\}_{n}$ with respect to $n$. It is easy to check that
\begin{eqnarray} \label{uniinl2normweightl}
\frac{d}{dt}(\frac{1}{2}||f^{n}||^{2}_{L^{2}_{l}}) &=& \langle  Q^{\epsilon +}(g,f^{n-1}) - Q^{\epsilon -}(g,f^{n}) + \epsilon^{2-2s}Q_{L}(g,f^{n}), f^{n} \langle v \rangle^{2l} \rangle 
\\&\eqdefa&\mathfrak{I}_1-\mathfrak{I}_2+\epsilon^{2-2s}\mathfrak{I}_3. \nonumber
\end{eqnarray}
By Cauchy-Schwartz inequality, there holds 
\begin{eqnarray} \label{uniinl2lterm1}
\mathfrak{I}_1 &=& \int B^{\epsilon}g_{*}f^{n-1}f^{n}(v^{\prime})\langle v^{\prime} \rangle^{2l}dvdv_{*}d\sigma \\ 
&\lesssim& (\int B^{\epsilon}g_{*}(f^{n-1})^{2}\langle v^{\prime} \rangle^{2l}dvdv_{*}d\sigma)^{1/2} \times (\int B^{\epsilon}g_{*}(f^{n})^{2}(v^{\prime}) \langle v^{\prime} \rangle^{2l}dv^{\prime} dv_{*}d\sigma)^{1/2}
\nonumber \\&\lesssim& (A^{\epsilon}||g||_{L^{1}_{2l+\gamma}}||f^{n-1}||^{2}_{L^{2}_{l+\gamma/2}})^{1/2} \times
(A^{\epsilon}||g||_{L^{1}}||f^{n}||^{2}_{L^{2}_{l+\gamma/2}})^{1/2}\nonumber
\\&\lesssim& A^{\epsilon}||g||_{L^{1}_{2l+\gamma}} ||f^{n-1}||_{L^{2}_{l+\gamma/2}}||f^{n}||_{L^{2}_{l+\gamma/2}}, \nonumber
\end{eqnarray}
where we have used the estimate (\ref{roughonvprime}) and the usual change of variable $ v \rightarrow v^{\prime}$.
By direct calculation, we have 
\begin{eqnarray} \label{uniinl2lterm2}
\mathfrak{I}_2 = \int B^{\epsilon}g_{*}(f^{n})^{2}\langle v \rangle^{2l}dvdv_{*}d\sigma 
\leq A^{\epsilon}||g||_{L^{1}_{\gamma}}||f^{n}||^{2}_{L^{2}_{l+\gamma/2}}. 
\end{eqnarray}
By coercivity estimate (\ref{coerlandau}) and commutator (\ref {commlandau}) estimate of the Landau operator, thus we have 
\begin{eqnarray} \label{uniinl2lterm3}
\mathfrak{I}_3 &=&  \langle Q_{L}(g,f^{n}\langle v \rangle^{l}),  f^{n} \langle v \rangle^{l}  \rangle +  \{\langle Q_{L}(g,f^{n})\langle v \rangle^{l} - Q_{L}(g,f^{n}\langle v \rangle^{l}), f^{n} \langle v \rangle^{l}\rangle\}\\
&\leq& -C_{1}(g)||f^{n}||^{2}_{H^{1}_{l+\gamma/2}} + C_{2}(g)||f^{n}||^{2}_{L^{2}_{l+\gamma/2}} + \Lambda C(l)||g||_{L^{1}_{\gamma+3}} ||f^{n}||_{H^{1}_{l+\gamma/2}} ||f^{n}||_{L^{2}_{l+\gamma/2}}.\nonumber
\end{eqnarray}
Now patching together the inequalities (\ref{uniinl2lterm1}), (\ref{uniinl2lterm2}) and (\ref{uniinl2lterm3}), and using the basic inequality (\ref{basicineq}), we have 
\beno
\frac{d}{dt}(\frac{1}{2}||f^{n}||^{2}_{L^{2}_{l}}) + \frac{C_{1}}{2}\epsilon^{2-2s}||f^{n}||^{2}_{H^{1}_{l+\gamma/2}} \leq \frac{C_{1}}{8}\epsilon^{2-2s}||f^{n-1}||^{2}_{L^{2}_{l+\gamma/2}}  + K_{1} ||f^{n}||^{2}_{L^{2}_{l+\gamma/2}},
\eeno
where $C_{1}, K_{1}$ are some positive constants depending on $m, M, \epsilon$. For any $\lambda, s>0$, one has
\begin{eqnarray}\label{l2hsl1}
||f||^{2}_{L^{2}} \leq \lambda ||f||^{2}_{H^{s}} + \frac{4\pi}{3} \lambda^{-\frac{3}{2s}} ||f||^{2}_{L^{1}}.
\end{eqnarray}
With the help of the above inequality, we have 
\beno
\frac{d}{dt}||f^{n}||^{2}_{L^{2}_{l}} + \frac{C_{1}}{2}\epsilon^{2-2s}||f^{n}||^{2}_{H^{1}_{l+\gamma/2}} \leq \frac{C_{1}}{4}\epsilon^{2-2s}||f^{n-1}||^{2}_{L^{2}_{l+\gamma/2}} + K_{1} ||f^{n}||^{2}_{L^{1}_{l+\gamma/2}},
\eeno
for some new constant $K_{1}$. By the previous step, with the uniform upper bound of $||f^{n}||_{L^{1}_{l+\gamma/2}}$, we have 
\begin{eqnarray}\label{l2llimtedincrease}
\frac{d}{dt}||f^{n}||^{2}_{L^{2}_{l}} + \frac{C_{1}}{2}\epsilon^{2-2s}||f^{n}||^{2}_{L^{2}_{l+\gamma/2}} \leq \frac{C_{1}}{4}\epsilon^{2-2s}||f^{n-1}||^{2}_{L^{2}_{l+\gamma/2}} + K_{1}K_{2},
\end{eqnarray}
where $K_{2}$ is some constant depending on $||f_{0}||_{L^{1}_{l+3\gamma/2}}$ and uniform upper bound of $||g||_{L^{1}_{l+3\gamma/2}}$. 
Now we use the same technique as in the previous step. Integrating both sides with respect to time, for any $t_{n} \in [0,t]$, we obtain
\beno
||f^{n}(t_{n})||^{2}_{L^{2}_{l}} + \frac{C_{1}}{2}\epsilon^{2-2s}\int_{0}^{t_{n}}||f^{n}(r)||^{2}_{L^{2}_{l+\gamma/2}}dr &\leq& ||f_{0}||^{2}_{L^{2}_{l}} + \frac{C_{1}}{4}\epsilon^{2-2s}\int_{0}^{t_{n}}||f^{n-1}(r)||^{2}_{L^{2}_{l+\gamma/2}}dr \\ && +K_{1}K_{2}t.
\eeno
Now denote $S^{n}(t_{n}) = \sum_{i = 0}^{n} (\frac{1}{2})^{i} ||f^{n-i}(t_{n})||^{2}_{L^{2}_{l}}$  and $x^{n}(t_{n}) = C_{1}\epsilon^{2-2s}\int_{0}^{t_{n}}||f^{n}(r)||^{2}_{L^{2}_{l+\gamma/2}}dr$ for $n \geq 0$, by recursive derivation
and noting that $x^{0}(t_{n}) \leq C_{1}\epsilon^{2-2s}t||f_{0}||^{2}_{L^{2}_{l+\gamma/2}}$, we obtain, for $n \geq 1$, 
\beno
S^{n}(t_{n})  + \frac{1}{2}x^{n}(t_{n})&\leq& \sum_{i=0}^{n}(\frac{1}{2})^{i}(||f_{0}||^{2}_{L^{2}_{l}}+K_{1}K_{2}t ) + \frac{C_{1}}{2^{n+1}}\epsilon^{2-2s}t||f_{0}||^{2}_{L^{2}_{l+\gamma/2}} \\
&\leq& 2||f_{0}||^{2}_{L^{2}_{l}} + 2K_{1}K_{2}t  +  \frac{C_{1}}{4}\epsilon^{2-2s}t||f_{0}||^{2}_{L^{2}_{l+\gamma/2}}. \\
\eeno
By tracking the definitions of constants $K_{1}, K_{2}$,  we obtain
\begin{eqnarray}\label{fnl2luniform}
\sup_{0\leq s \leq t}\sup_{n}||f^{n}(s)||_{L^{2}_{l}}  
\leq C(||f_{0}||_{L^{1}_{l+3\gamma/2}},||f_{0}||_{L^{2}_{l+\gamma/2}}, t, \sup_{0 \leq s \leq t} ||g(s)||_{L^{2}_{2l+\gamma+2}}). 
\end{eqnarray}
\noindent {\textit {Step 1.3}:} (Uniform Upper Bound in $H^{m}_{l}$ with $m \geq 1$)\\
Fix an $\alpha$ with $|\alpha| \leq m$, one has
\beno
\partial_{t}\partial^{\alpha}_{v}f^{n} = \sum_{\alpha_{1}+\alpha_{2}=\alpha} \binom{\alpha}{\alpha_{1}}[Q^{\epsilon +}(\partial^{\alpha_{1}}_{v}g,\partial^{\alpha_{2}}_{v}f^{n-1}) - Q^{\epsilon -}(\partial^{\alpha_{1}}_{v}g,\partial^{\alpha_{2}}_{v}f^{n})
+\epsilon^{2-2s}Q_{L}(\partial^{\alpha_{1}}_{v}g,\partial^{\alpha_{2}}_{v}f^{n})].
\eeno
Then we have
\beno
\frac{d}{dt}(\frac{1}{2}||\partial^{\alpha}_{v}f^{n}||^{2}_{L^{2}_{l}})&=& \sum_{\alpha_{1}+\alpha_{2}=\alpha} \binom{\alpha}{\alpha_{1}}
[\langle Q^{\epsilon +}(\partial^{\alpha_{1}}_{v}g,\partial^{\alpha_{2}}_{v}f^{n-1}), \partial^{\alpha}_{v}f^{n} \langle v \rangle^{2l} \rangle \\&&- \langle Q^{\epsilon -}(\partial^{\alpha_{1}}_{v}g,\partial^{\alpha_{2}}_{v}f^{n}), \partial^{\alpha}_{v}f^{n} \langle v \rangle^{2l}  \rangle+\epsilon^{2-2s}\langle Q_{L}(\partial^{\alpha_{1}}_{v}g,\partial^{\alpha_{2}}_{v}f^{n}), \partial^{\alpha}_{v}f^{n} \langle v \rangle^{2l} \rangle]
\\&\eqdefa& \sum_{\alpha_{1}+\alpha_{2}=\alpha} \binom{\alpha}{\alpha_{1}} [\mathfrak{I}_{1}(\alpha_{1},\alpha_{2})-\mathfrak{I}_{2}(\alpha_{1},\alpha_{2})+\epsilon^{2-2s}\mathfrak{I}_{3}(\alpha_{1},\alpha_{2})].
\eeno
As the same as (\ref{uniinl2lterm1}), we have 
\beno
|\mathfrak{I}_{1}(\alpha_{1},\alpha_{2})| \lesssim  A^{\epsilon}||g||_{H^{m}_{2l+\gamma+2}} ||f^{n-1}||_{H^{m}_{l+\gamma/2}}||f^{n}||_{H^{m}_{l+\gamma/2}}
\eeno
As the same as (\ref{uniinl2lterm2}), we have 
\beno
|\mathfrak{I}_{2}(\alpha_{1},\alpha_{2})| \lesssim  A^{\epsilon}||g||_{H^{m}_{\gamma+2}} ||f^{n}||^{2}_{H^{m}_{l+\gamma/2}}.
\eeno
When $|\alpha_{2}| \leq |\alpha|-1 \leq m-1$, by upper bound estimate (\ref{uplandau}) and commutator (\ref {commlandau}) estimate of the Landau operator, we have
\beno
|\mathfrak{I}_{3}(\alpha_{1},\alpha_{2})| &\leq& |\langle Q_{L}(\partial^{\alpha_{1}}_{v}g,\partial^{\alpha_{2}}_{v}f^{n}\langle v \rangle^{l}),  \partial^{\alpha}_{v}f^{n} \langle v \rangle^{l}  \rangle| \\&&+  |\{\langle Q_{L}(\partial^{\alpha_{1}}_{v}g,f^{n})\langle v \rangle^{l} - Q_{L}(\partial^{\alpha_{1}}_{v}g,\partial^{\alpha_{2}}_{v}f^{n}\langle v \rangle^{l}), \partial^{\alpha}_{v}f^{n} \langle v \rangle^{l}\rangle\}|
\\&\lesssim& ||g||_{H^{m}_{\gamma+4}} ||f^{n}||_{H^{m}_{l+\gamma/2+2}}||f^{n}||_{H^{m+1}_{l+\gamma/2}} +
||g||_{H^{m}_{\gamma+5}} ||f^{n}||^{2}_{H^{m}_{l+\gamma/2}}.
\eeno
When $\alpha_{2} = \alpha$, as the same as (\ref{uniinl2lterm3}), we have  
\beno
\mathfrak{I}_{3}(0,\alpha) \leq -C_{1}(g)||\partial^{\alpha}_{v}f^{n}||^{2}_{H^{1}_{l+\gamma/2}} + C_{2}(g)||\partial^{\alpha}_{v}f^{n}||^{2}_{L^{2}_{l+\gamma/2}} + \Lambda C(l)||g||_{H^{m}_{\gamma+5}} ||f^{n}||_{H^{m+1}_{l+\gamma/2}} ||f^{n}||_{H^{m}_{l+\gamma/2}}
\eeno
Now patching together the above estimates and taking sum over $|\alpha| \leq m$, we have 
\beno
\frac{1}{2}\frac{d}{dt}||f^{n}||^{2}_{H^{m}_{l}} + \frac{C_{1}}{2}\epsilon^{2-2s}||f^{n}||^{2}_{H^{m+1}_{l+\gamma/2}} \leq \frac{C_{1}}{4}\epsilon^{2-2s}||f^{n-1}||^{2}_{H^{m}_{l+\gamma/2}} + K_{1} ||f^{n}||^{2}_{H^{m}_{l+\gamma/2+2}},
\eeno
where $C_{1}, K_{1}$ are some positive constants depending on uniform upper bound of $||g||_{H^{m}_{2l+\gamma+2}}$ and uniform lower bound of $||g||_{L^{1}}$. Thanks to interpolation theory and the basic inequality (\ref{basicineq}), for any $\eta > 0$, there exists some constant $ C_{\eta}$ such that 
\beno
||f^{n}||^{2}_{H^{m}_{l+\gamma/2+2}} \leq \eta ||f^{n}||^{2}_{H^{m+1}_{l+\gamma/2}} + C_{\eta} ||f^{n}||^{2}_{L^{1}_{l+\gamma/2+2m+6}},
\eeno
thus we have 
\beno
\frac{d}{dt}||f^{n}||^{2}_{H^{m}_{l}} + \frac{C_{1}}{2}\epsilon^{2-2s}||f^{n}||^{2}_{H^{m}_{l+\gamma/2}} \leq \frac{C_{1}}{8}\epsilon^{2-2s}||f^{n-1}||^{2}_{H^{m}_{l+\gamma/2}} + K_{1} ||f^{n}||^{2}_{L^{1}_{l+\gamma/2+2m+6}},
\eeno
for some new constant $K_{1}$. By the previous step, with the uniform upper bound of $||f^{n}||_{L^{1}_{l+\gamma/2+2m+6}}$, we have
\begin{eqnarray}\label{hmllimtedincrease}
\frac{d}{dt}||f^{n}||^{2}_{L^{2}_{l}} + \frac{C_{1}}{2}\epsilon^{2-2s}||f^{n}||^{2}_{L^{2}_{l+\gamma/2}} \leq \frac{C_{1}}{4}\epsilon^{2-2s}||f^{n-1}||^{2}_{L^{2}_{l+\gamma/2}} + K_{1}K_{2},
\end{eqnarray}
where $K_{2}$ is some constant depending on $||f_{0}||_{L^{1}_{l+3\gamma/2+2m+6}}$ and uniform upper bound of $||g||_{L^{1}_{l+3\gamma/2+2m+6}}$. Noticing that inequality (\ref{hmllimtedincrease}) has exactly the same structure as inequality (\ref{l2llimtedincrease}), we have 
\begin{eqnarray}\label{fnhmluniform}
&&\sup_{0\leq s \leq t}\sup_{n}||f^{n}(s)||_{H^{m}_{l}}
\\&\leq& C(||f_{0}||_{L^{1}_{l+3\gamma/2+2m+6}},||f_{0}||_{H^{m}_{l+\gamma/2}}, t, \sup_{0 \leq s \leq t} ||g(s)||_{H^{m}_{2l+\gamma+2}}, \sup_{0 \leq s \leq t} ||g(s)||_{L^{1}_{l+3\gamma/2+2m+6}}). \nonumber
\end{eqnarray}
\\
\noindent {\textit {Step 2}:} (Cauchy Sequence)\\
In this step, we prove that $\{f^{n}(t)\}_{n \in \N}$ is a Cauchy sequence in $L^{1}_{l}$ for any $ t\geq 0$. Set $h^{n} = f^{n} - f^{n-1}$ for $n \geq 1$. Then for $n \geq 2$, we have
\begin{equation}\label{diffequation1} \left\{ \begin{aligned}
&\partial_{t}h^{n} = Q^{\epsilon +}(g,h^{n-1}) - Q^{\epsilon -}(g,h^{n}) + \epsilon^{2-2s}Q_{L}(g,h^{n}),\\
&h^{n}|_{t=0} = 0.
 \end{aligned} \right.
\end{equation}
Because we are uncertain about the sign of $h^{n}$, we have to introduce the sign function $sgn(h^{n})$. Similar as in (\ref{fnl1lfirstterm}), we obtain
\begin{eqnarray}\label{hnl1lfirstterm}
&&\int_{\R^{3}}Q^{\epsilon+}(g,h^{n-1})(v)sgn(h^{n})\langle v \rangle^{l}dv \\&\leq& \int_{\R^{6}}\int_{\SS^{2}}b^{\epsilon}(\cos\theta)|v-v_{*}|^{\gamma}g_{*}|h^{n-1}|\langle v^{\prime} \rangle^{l}dvdv_{*}d\sigma \nonumber \\
&\leq&(1+\eta)^{\frac{\gamma}{2}}A^{\epsilon}||g||_{L^{1}}||h^{n-1}||_{L^{1}_{l+\gamma}} \nonumber\\&&+
(1+\frac{1}{\eta})^{\frac{\gamma}{2}}A^{\epsilon}||g||_{L^{1}_{l+\gamma}}||h^{n-1}||_{L^{1}} \nonumber\\
&&+C(l,\gamma,\eta)A^{\epsilon}||g||_{L^{1}_{l}}||h^{n-1}||_{L^{1}_{l}}. \nonumber
\end{eqnarray}
Similar as in (\ref{fnl1lsecondterm})
\begin{eqnarray}\label{hnl1lsecondterm}
\int_{\R^{3}}Q^{\epsilon-}(g,h^{n})(v)sgn(h^{n})\langle v \rangle^{l}dv &=& \int_{\R^{3}}b^{\epsilon}|v-v_{*}|^{\gamma}g_{*}|h^{n}|\langle v \rangle^{l}dvdv_{*}d\sigma \\&\geq& \frac{A^{\epsilon}}{(1+\eta)^{\gamma/2}}||g||_{L^{1}}||h^{n}||_{L^{1}_{l+\gamma}} \nonumber\\&&- \eta^{-\gamma/2}A^{\epsilon}||g||_{L^{1}_{\gamma}}||h^{n}||_{L^{1}_{l}}. \nonumber
\end{eqnarray}
For the inner product $\langle Q_{L}(g, h^{n}),sgn(h^{n})\langle v \rangle^{l}\rangle $, we can approximate Landau operator by Boltzmann operators.
Let $b_{\lambda}(\cos\theta) = \lambda^{2s-2} b(\cos\theta)\textbf{1}_{\theta\leq\lambda} $ for each $\lambda \leq \frac{\pi}{2}$, such that
\beno
 \lim_{\lambda \downarrow 0}\int_{\SS^2}b_{\lambda}(\cos\theta)\sin^{2}\theta d\sigma = \Lambda.
\eeno
Let $Q_{\lambda}$ be the Boltzmann operator associated to the kernel $b_{\lambda}(\cos\theta)|v-v_{*}|^{\gamma}$, then by lemma 7.1 in \cite{he1}, there holds
\begin{eqnarray}\label{grazinglimit}
&&|\langle Q_{L}(g,h^{n}) ,  sgn(h^{n})\langle v \rangle^{l} \rangle_{v} -\langle Q_{\lambda}(g,h^{n}) ,  sgn(h^{n})\langle v \rangle^{l} \rangle_{v}|\lesssim \lambda\|g\|_{H^3_{l+\gamma+12}}\|h^n\|_{H^5_{l+\gamma+10}}.
\end{eqnarray}
By the uniform estimate (\ref{fnhmluniform}) and our assumption on $g$ and $f_{0}$, we have 
\begin{eqnarray}\label{hnhmluniform}
&&\sup_{0\leq t \leq T}\sup_{n} \|h^n(t)\|_{H^5_{l+\gamma+10}}  
\\&\leq& C(||f_{0}||_{L^{1}_{l+5\gamma/2+16}},||f_{0}||_{H^{5}_{l+3\gamma/2+10}}, t, \sup_{0 \leq s \leq t} ||g(s)||_{H^{5}_{2l+3\gamma+22}}, \sup_{0 \leq s \leq t} ||g(s)||_{L^{1}_{l+5\gamma/2+16}}).  \nonumber
\end{eqnarray}

Thanks to proposition \ref{lemma1}, for $l \geq 4$, we derive that
\beno
&&\langle Q_{\lambda}(g,h^{n}) ,  sgn(h^{n})\langle v \rangle^{l}  \rangle_{v} \\&=&  \int_{\R^{6}}\int_{\SS^{2}}b_{\lambda}|v-v_{*}|^{\gamma}g_{*}h^{n}(sgn(h^{n}(v^{\prime}))\langle v^{\prime} \rangle^{l}-sgn(h^{n}(v))\langle v \rangle^{l})dvdv_{*}d\sigma \\
&\leq& \int_{\R^{6}}\int_{\SS^{2}}b_{\lambda}|v-v_{*}|^{\gamma}g_{*}|h^{n}|(\langle v^{\prime} \rangle^{l}-\langle v \rangle^{l})dvdv_{*}d\sigma \\
&\leq& -\int_{\R^{6}}\int_{\SS^{2}}b_{\lambda}|v-v_{*}|^{\gamma}g_{*}|h^{n}|\langle v \rangle^{l}(1-\cos^{l}\frac{\theta}{2})dvdv_{*}d\sigma\\
&&+ \int_{\R^{6}}\int_{\SS^{2}}b_{\lambda}|v-v_{*}|^{\gamma}g_{*}|h^{n}|\langle v_{*} \rangle^{l}\sin^{l}\frac{\theta}{2}dvdv_{*}d\sigma \\
&&+ C(l)\int_{\R^{6}}\int_{\SS^{2}}b_{\lambda}|v-v_{*}|^{\gamma}g_{*}|h^{n}|\langle v \rangle^{l-2}\langle v_{*} \rangle^{l-2}\sin^{2}\frac{\theta}{2}dvdv_{*}d\sigma.
\eeno
For $\lambda$ small enough, we have
\beno
\frac{\Lambda}{2} \leq \int_{\SS^2}b_{\lambda}(\cos\theta)\sin^{2}\theta d\sigma \leq 2 \Lambda.
\eeno
Thus we have
\beno
\int_{\SS^2}b_{\lambda}(\cos\theta)(1-\cos^{l}\frac{\theta}{2}) d\sigma \geq \int_{\SS^2}b_{\lambda}(\cos\theta)\sin^{2}\frac{\theta}{2} d\sigma \geq \frac{\Lambda}{8},
\eeno
and
\beno
\int_{\SS^2}b_{\lambda}(\cos\theta)\sin^{2}\frac{\theta}{2} d\sigma \leq \int_{\SS^2}b_{\lambda}(\cos\theta)\frac{\sin^{2}\theta}{2} d\sigma \leq \Lambda,
\eeno
and finally
\beno
\int_{\SS^2}b_{\lambda}(\cos\theta)\sin^{l}\frac{\theta}{2} d\sigma \leq \frac{\lambda^{2}}{4}\int_{\SS^2}b_{\lambda}(\cos\theta)\frac{\sin^{2}\theta}{2} d\sigma \leq \frac{\lambda^{2}}{4}\Lambda.
\eeno
With the help of the above three inequalities, we arrive at
\beno
\langle Q_{\lambda}(g,h^{n}) ,  sgn(h^{n})\langle v \rangle^{l}  \rangle_{v} &\leq& -\frac{\Lambda}{16}||g||_{L^{1}}||h^{n}||_{L^{1}_{l+\gamma}}+ \frac{\Lambda}{8}||g||_{L^{1}_{\gamma}}||h^{n}||_{L^{1}_{l}} \\&& + C(l)\Lambda||g||_{L^{1}_{l}}||h^{n}||_{L^{1}_{l}} + \frac{\lambda^{2}}{4}\Lambda ||g||_{L^{1}_{l+\gamma}}||h^{n}||_{L^{1}_{\gamma}}.
\eeno
Let $\lambda$ tend to $0$, by (\ref{grazinglimit}) and the uniform estimate (\ref{hnhmluniform}), we have
\begin{eqnarray}\label{hnl1lthirdterm}
\langle Q_{L}(g,h^{n}) ,  sgn(h^{n})\langle v \rangle^{l} \rangle_{v} &\leq& -\frac{\Lambda}{16}||g||_{L^{1}}||h^{n}||_{L^{1}_{l+\gamma}}+ \frac{\Lambda}{8}||g||_{L^{1}_{\gamma}}||h^{n}||_{L^{1}_{l}} \\&& + C(l)\Lambda||g||_{L^{1}_{l}}||h^{n}||_{L^{1}_{l}}. \nonumber
\end{eqnarray}
Choose $\eta$ small enough such that
\beno
(1+\eta)^{\frac{\gamma}{2}}A^{\epsilon} \leq \frac{A^{\epsilon}}{(1+\eta)^{\gamma/2}}+\frac{1}{32}\epsilon^{2-2s}\Lambda,
\eeno
and denote $a = \frac{A^{\epsilon}}{(1+\eta)^{\gamma/2}}+\frac{1}{32}\epsilon^{2-2s}\Lambda, ~\delta = \frac{1}{32}\epsilon^{2-2s}\Lambda$.
Patch altogether (\ref{hnl1lfirstterm}), (\ref{hnl1lsecondterm}) and (\ref{hnl1lthirdterm}), we obtain
\beno
\frac{d}{dt}||h^{n}||_{L^{1}_{l}} &\leq& -(a+\delta)||g||_{L^{1}}||h^{n}||_{L^{1}_{l+\gamma}} + a||g||_{L^{1}}||h^{n-1}||_{L^{1}_{l+\gamma}} \\ &&+(1+\frac{1}{\eta})^{\frac{\gamma}{2}}A^{\epsilon}||g||_{L^{1}_{l+\gamma}}||h^{n-1}||_{L^{1}}
+C(l,\gamma,\eta)A^{\epsilon}||g||_{L^{1}_{l}}||h^{n-1}||_{L^{1}_{l}} \\
&&+\eta^{-\gamma/2}A^{\epsilon}||g||_{L^{1}_{\gamma}}||h^{n}||_{L^{1}_{l}}+\frac{\Lambda}{8}\epsilon^{2-2s}||g||_{L^{1}_{\gamma}}||h^{n}||_{L^{1}_{l}}
\\&&+C(l)\epsilon^{2-2s}\Lambda||g||_{L^{1}_{l}}||h^{n}||_{L^{1}_{l}} \\
\eeno
For ease of notation, denote $K_{1} = (1+\frac{1}{\eta})^{\frac{\gamma}{2}}A^{\epsilon}$, $K_{2}=C(l,\gamma,\eta)A^{\epsilon}\sup_{0\leq s \leq t}||g(s)||_{L^{1}_{l}}$ and $K_{3} = (\eta^{-\gamma/2}A^{\epsilon}+\frac{\Lambda}{8}\epsilon^{2-2s}+C(l)\epsilon^{2-2s}\Lambda)\sup_{0\leq s \leq t}||g(s)||_{L^{1}_{l}}$. Then we have a much neater inequality on the interval $[0, t]$,
\begin{eqnarray}\label{hnhnminus1}
&&\frac{d}{dt}||h^{n}||_{L^{1}_{l}} + (a+\delta)||g||_{L^{1}}||h^{n}||_{L^{1}_{l+\gamma}} \\&\leq&   a||g||_{L^{1}}||h^{n-1}||_{L^{1}_{l+\gamma}}+ K_{3}||h^{n}||_{L^{1}_{l}} +(K_{1}||g||_{L^{1}_{l+\gamma}}+K_{2})||h^{n-1}||_{L^{1}_{l}}. \nonumber
\end{eqnarray}
Using the same technique as in the previous step, by defining $y^{n}(t_{n}) = e^{-K_{3}t_{n}}||h^{n}(t_{n})||_{L^{1}_{l}}$ and $x^{n}(t_{n}) = \int_{0}^{t_{n}}e^{-K_{3}s}||g(s)||_{L^{1}}||h^{n}(s)||_{L^{1}_{l+\gamma}}ds$, for $n \geq 1$ and $t_{n} \in [0, t]$.
Then for $n \geq 2$, we derive that
\beno
y^{n}(t_{n})  + (a+\delta)x^{n}(t_{n})&\leq& a x^{n-1}(t_{n}) +\int_{0}^{t_{n}}(K_{1}||g(s)||_{L^{1}_{l+\gamma}}+K_{2})y^{n-1}(s)ds,
\eeno
where we have used the initial condition $h^{n}(0) = 0$.
Now denote $S^{n}(s) = \sum_{i = 0}^{n-1} (\frac{a}{a+\delta})^{i} y^{n-i}(s)$ for $n \geq 1$ and $s \in [0, t]$, by recursive derivation, we obtain
\beno
S^{n}(t_{n})  + (a+\delta)x^{n}(t_{n})&\leq& (\frac{a}{a+\delta})^{n-1}y^{1}(t_{n}) + (\frac{a}{a+\delta})^{n-2}a x^{1}(t_{n}) \\&&+\int_{0}^{t}(K_{1}||g(t_{n-1})||_{L^{1}_{l+\gamma}}+K_{2})S^{n-1}(t_{n-1})dt_{n-1}.
\eeno
By previous estimates (\ref{fnl1luniform}), we have
\beno
\sup_{0 \leq t_{n} \leq t} \{y^{1}(t_{n})+x^{1}(t_{n})\} &\leq& C(||f_{0}||_{L^{1}_{l+\gamma}},t, \sup_{0 \leq s \leq t} ||g(s)||_{L^{1}_{l}}, \int_{0}^{t}||g(s)||_{L^{1}_{l+\gamma}}ds)\\
&\eqdefa& C(t). \eeno
For ease of notation, for $n \geq 1$, let us define
\beno
b^{n}(t) = ((\frac{a}{a+\delta})^{n-1} + (\frac{a}{a+\delta})^{n-2}a)C(t)
.\eeno
Thus, by further recursive derivation, for any $t_{n} \in [0,t]$, we obtain
\beno
&&S^{n}(t_{n})  + (a+\delta)x^{n}(t_{n}) \\&\leq& b^{n}(t)+\int_{0}^{t_{n}}(K_{1}||g(t_{n-1})||_{L^{1}_{l+\gamma}}+K_{2})S^{n-1}(t_{n-1})dt_{n-1}\\
&\leq& \sum_{i=2}^{n} b^{i}(t)\frac{(\int_{0}^{t_{n}}(K_{1}||g(s)||_{L^{1}_{l+\gamma}}+K_{2})ds)^{n-i}}{(n-i)!} \\
&&+\int_{0}^{t_{n}}\int_{0}^{t_{n-1}}\cdots\int_{0}^{t_{2}} S^{1}(t_{1})\prod_{i=1}^{n-1}(K_{1}||g(t_{i})||_{L^{1}_{l+\gamma}}+K_{2}) dt_{n-1}dt_{n-2}\cdots dt_{1} \\
&\leq& \sum_{i=1}^{n} b^{i}(t)\frac{(\int_{0}^{t_{n}}(K_{1}||g(s)||_{L^{1}_{l+\gamma}}+K_{2})ds)^{n-i}}{(n-i)!},
\eeno
where we used the fact $S^{1}(t_{1}) \leq C(t) \leq (a+\delta+1)C(t) = b^{1}(t)$. Note that $b^{n}(t)$ is a geometric sequence and $b^{n}(t) = b^{1}(t)(\frac{a}{a+\delta})^{n-1}$ for any $n \geq 1$, thus we have
\beno
S^{n}(t_{n})  + (a+\delta)x^{n}(t_{n})
&\leq& b^{1}(t) \sum_{i=1}^{n} (\frac{a}{a+\delta})^{i-1} \frac{(\int_{0}^{t_{n}}(K_{1}||g(s)||_{L^{1}_{l+\gamma}}+K_{2})ds)^{n-i}}{(n-i)!} \\
&=& b^{1}(t)(\frac{a}{a+\delta})^{n-1}\sum_{i=1}^{n} \frac{(\frac{a+\delta}{a}\int_{0}^{t_{n}}(K_{1}||g(s)||_{L^{1}_{l+\gamma}}+K_{2})ds)^{n-i}}{(n-i)!}
\\&\leq& b^{1}(t)(\frac{a}{a+\delta})^{n-1} \exp(\frac{a+\delta}{a}\int_{0}^{t_{n}}(K_{1}||g(s)||_{L^{1}_{l+\gamma}}+K_{2})ds).
\eeno
By recalling the definitions of $S^{n}$ and $x^{n}$, we arrive at
\beno
&&\sup_{0\leq s \leq t}||h^{n}(s)||_{L^{1}_{l}}  + \int_{0}^{t}e^{K_{3}(t-s)}||g(s)||_{L^{1}}||h^{n}(s)||_{L^{1}_{l+\gamma}}ds
\\ &\leq& b^{1}(t)(\frac{a}{a+\delta})^{n-1} \exp(\frac{a+\delta}{a}\int_{0}^{t}(K_{1}||g(s)||_{L^{1}_{l+\gamma}}+K_{2})ds+K_{3}t).
\eeno
Since the series $\sum_{n}(\frac{a}{a+\delta})^{n-1}$ is finite, we conclude that $\{f^{n}\}_{n \in \N}$ is a Cauchy sequence in $L^{\infty}([0,t]; L^{1}_{l})\cap L^{1}([0,t]; L^{1}_{l+\gamma})$. Due to the arbitrariness of $t \in [0, T]$, there is a function $f \in L^{\infty}([0,T]; L^{1}_{l})\cap L^{1}([0,T]; L^{1}_{l+\gamma})$ such that
\beno
\lim_{n \rightarrow \infty} \{\sup_{0\leq s \leq T}||f^{n}(s)-f(s)||_{L^{1}_{l}}  + \int_{0}^{T}||f^{n}(s)-f(s)||_{L^{1}_{l+\gamma}}ds\} = 0
\eeno
It is obvious that $f$ is the solution to (\ref{forpositive2}). Thus the non-positivity of $f$ is ensured by the non-positivity of $f^{n}$.\\
\noindent {\textit {Step 3}:} (High Order Moments and Smoothness)\\
In this step, we prove the solution $f$ constructed in the previous step actually lies in $L^{\infty}([0,T]; L^{1}_{l+5\gamma/2+16}\cap H^5_{l+\gamma+10})\cap L^{1}([0,T]; L^{1}_{l+7\gamma/2+16})$. Let $q = l+5\gamma/2+16$. By lemma \ref{lemma1} and inequality (\ref{usfulinlowerbound}), we first have
\beno
\int_{\R^{3}}Q^{\epsilon}(g,f)(v)\langle v \rangle^{q}dv &=& \int_{\R^{6}}\int_{\SS^{2}}b^{\epsilon}(\cos\theta)|v-v_{*}|^{\gamma}g_{*}f(\langle v^{\prime} \rangle^{q} - \langle v \rangle^{q})dvdv_{*}d\sigma
\\&\leq& -\frac{C_{\epsilon}}{2} ||g||_{L^{1}_{}} ||f||_{L^{1}_{q+\gamma}} + C_{\epsilon}||g||_{L^{1}_{\gamma}} ||f||_{L^{1}_{q}} + A_{2}||g||_{L^{1}_{q+\gamma}}||f||_{L^{1}_{}} +  A_{2}2^{l}||g||_{L^{1}_{q}}||f||_{L^{1}_{q}}.
\eeno
Next, according to \cite{dv1}, one has 
\beno
\langle Q_{L}(g,h), \langle v \rangle^{q}\rangle \leq - \Lambda q ||g||_{L^{1}_{}} ||f||_{L^{1}_{l+\gamma}} + \Lambda (4q+2)q ||g||_{L^{1}_{4}} ||f||_{L^{1}_{q}}.
\eeno
Therefore we have
\beno
\frac{d}{dt}||f||_{L^{1}_{q}} + \frac{C_{\epsilon}}{2} ||g||_{L^{1}_{}} ||f||_{L^{1}_{q+\gamma}} \leq C(M,\Lambda,q)||g||_{L^{1}_{q}}||f||_{L^{1}_{q}} + A_{2}||f_{0}||_{L^{1}_{}} ||g||_{L^{1}_{q+\gamma}}.
\eeno
By Gronwall's inequality, it is not difficult to derive
\beno
\sup_{0\leq s \leq T}||f||_{L^{1}_{q}}  + \int_{0}^{T}||f(t)||_{L^{1}_{q+\gamma}}dt \leq C(||f_{0}||_{L^{1}_{q}}, \sup_{0\leq t \leq T}||f||_{L^{1}_{q}},   \int_{0}^{T}||f(t)||_{L^{1}_{q+\gamma}}dt).
\eeno
Recalling the uniform estimate (\ref{hnhmluniform}) and the convergence of $\{f^{n}\}_{n \in \N}$ in $L^{\infty}([0,T]; L^{1}_{l})$, we also have $f \in L^{\infty}([0,T]; H^5_{l+\gamma+10})$.

\noindent {\textit {Step 4}:} (Uniqueness)\\
Suppose $f^{1}, f^{2} \in L^{\infty}([0,T]; L^{1}_{l+5\gamma/2+16}\cap H^5_{\gamma+10+l})\cap L^{1}([0,T]; L^{1}_{l+7\gamma/2+16})$ are two non-negative solutions of equation (\ref{forpositive2}), set $h = f^{1}-f^{2}$. Then $h$ is a solution to the following equation,
\begin{equation}\label{linearzeroinitial} \left\{ \begin{aligned}
&\partial_{t}h = Q^{\epsilon}(g,h) + \epsilon^{2-2s}Q_{L}(g,h)\\
&h|_{t=0} = 0.
\end{aligned} \right.
\end{equation}
Observe that the above equation is as the same as the equation (\ref{diffequation1}) if $h^{n-1} = h^{n}$. With the same argument until inequality (\ref{hnhnminus1}), we have
\beno
\frac{d}{dt}||h||_{L^{1}_{l}} + C_{1}||h^{n}||_{L^{1}_{l+\gamma}} &\leq&  C_{2}||h||_{L^{1}_{l}},
\eeno
where $C_{1}$ and $C_{2}$ are some positive constants depending on $M$ and $m$. Then we have
\beno
||h(t)||_{L^{1}_{l}} \leq ||h(0)||_{L^{1}_{l}} e^{C_{2}t},
\eeno
which gives the uniqueness.
\end{proof}

\subsection{First result on the well-posedness of approximate equation \eqref{approximated}} Based on the Picard iteration scheme, we derive that

\begin{lem}\label{lemmapositive3} Let $l \geq 4$ be a real number and $N$ be an nonnegative integer. Let $w_{H}, w_{L}, w$ be functions defined by
\begin{eqnarray}\label{functionwh}
w_{H}(N,l) = \max \{w(N,l)+3\gamma/2+4, 2l + 3 +\gamma/2\},
\end{eqnarray}
\begin{eqnarray}\label{functionwl}
w_{L}(N,l) &=&\max\{q(2,w(N,l)+\gamma+4),q(N,2l+3),q(N+1,l+\gamma/2+2)\}
\\&&+\gamma, \nonumber
\end{eqnarray}
\begin{eqnarray}\label{functionw}
w(N,l) = \frac{(N+s+2)(2l+3)-(N+2)(l+\gamma/2)}{s}.
\end{eqnarray}
 Suppose the non-negative datum $f_{0} \in H^{(N+2)\vee3}_{w_{H}(N,l)} \cap L^{1}_{w_{L}(N,l)}$ with $||f_{0}||_{L^{1}} > 0$,
then our approximate equation \eqref{approximated}
admits a non-negative solution $f$ in $L^{\infty}([0,T^{*}]; H^{N}_{l}\cap L^{1}_{w(N,l)})$ for some $T^{*} > 0$. Moreover, if $N \geq 2$ and $l \geq 8+\gamma$, the solution is unique.
\end{lem}
\begin{proof} Consider the sequence of functions $\{f^{n}\}_{n \in \N}$ defined by
\begin{equation}\label{sequencetooriginal} \left\{ \begin{aligned}
&f^{0}(t) = f_{0}, ~~ \text{for any }  t \geq 0 ; \\
&\partial_{t}f^{n} = Q^{\epsilon }(f^{n-1},f^{n}) + \epsilon^{2-2s}Q_{L}(f^{n-1},f^{n}), ~~ n \geq 1,\\
&f^{n}|_{t=0} = f_{0}.
 \end{aligned} \right.
\end{equation}
We first mention that equation (\ref{sequencetooriginal}) conserves mass, that is, $||f^{n}(t)||_{L^{1}} = ||f_{0}||_{L^{1}}$ for any $n \geq 0$ and $ t \geq 0$. By previous lemma, $f^{n} \geq 0$ for any $n \in \N$. \\
\noindent {\textit {Step 1}:} (Uniform $L^{1}_{l}$ Upper Bound)\\
In this step we prove that $\{f^{n}\}_{n}$ has uniform upper bound in $L^{\infty}([0,T^{*}(l)]; L^{1}_{l})$ with respect to $n$ for some $T^{*}(l) > 0$ if $f_{0} \in L^{1}_{l+\gamma}$.
Thanks to proposition \ref{lemma1}, for any $ l\geq 4$, we have
\beno
\langle Q^{\epsilon }(f^{n-1},f^{n}), \langle v \rangle^{l} \rangle
&=& \int_{\R^{6}}\int_{\SS^{2}}b^{\epsilon}|v-v_{*}|^{\gamma}f^{n-1}_{*}f^{n}(\langle v^{\prime} \rangle^{l}-\langle v \rangle^{l})dvdv_{*}d\sigma \\
&\leq& -\int_{\R^{6}}\int_{\SS^{2}}b^{\epsilon}|v-v_{*}|^{\gamma}f^{n-1}_{*}f^{n}\langle v \rangle^{l}(1-\cos^{l}\frac{\theta}{2})dvdv_{*}d\sigma  \\
&&+ \int_{\R^{6}}\int_{\SS^{2}}b^{\epsilon}|v-v_{*}|^{\gamma}f^{n-1}_{*}f^{n}\langle v_{*} \rangle^{l}\sin^{l}\frac{\theta}{2}dvdv_{*}d\sigma \\
&&+ C(l)\int_{\R^{6}}\int_{\SS^{2}}b^{\epsilon}|v-v_{*}|^{\gamma}f^{n-1}_{*}f^{n}\langle v \rangle^{l-2}\langle v_{*} \rangle^{l-2}\sin^{2}\frac{\theta}{2}dvdv_{*}d\sigma.
\eeno
In the following, denote $A^{\epsilon}_{2} = \int_{\SS^2}b^{\epsilon}\sin^{2}\frac{\theta}{2} d\sigma \leq \frac{A_{2}}{2}$, then we have
\beno
\int_{\SS^2}b^{\epsilon}(\cos\theta)(1-\cos^{l}\frac{\theta}{2}) d\sigma \geq \frac{3}{2}A^{\epsilon}_{2},
\eeno
and
\beno
\int_{\SS^2}b^{\epsilon}(\cos\theta)\sin^{l}\frac{\theta}{2} d\sigma \leq \frac{1}{2}A^{\epsilon}_{2},
\eeno
where we used $1-\cos^{l}\frac{\theta}{2} \geq 1-\cos^{4}\frac{\theta}{2} \geq \frac{3}{2}\sin^{2}\frac{\theta}{2}$ and $\sin^{l}\frac{\theta}{2} \leq \frac{1}{2}\sin^{2}\frac{\theta}{2}$. Together with $|v-v_{*}|^{\gamma}\geq \frac{3}{4}\langle v \rangle^{\gamma} - c_{1}\langle v_{*} \rangle^{\gamma}$,
$|v-v_{*}|^{\gamma}\leq 2(\langle v \rangle^{\gamma} + \langle v_{*} \rangle^{\gamma})$, and $|v-v_{*}|^{\gamma}\leq \langle v \rangle^{\gamma}  \langle v_{*} \rangle^{\gamma}$, we obtain
\begin{eqnarray}\label{lastfnl1lqep}
\langle Q^{\epsilon }(f^{n-1},f^{n}), \langle v \rangle^{l} \rangle
&\leq& -\frac{9}{8}A^{\epsilon}_{2}||f^{n-1}||_{L^{1}_{}}||f^{n}||_{L^{1}_{l+\gamma}}+\frac{3}{2}c_{1}A^{\epsilon}_{2}||f^{n-1}||_{L^{1}_{\gamma}}||f^{n}||_{L^{1}_{l}} \\
&& + A^{\epsilon}_{2}||f^{n-1}||_{L^{1}_{l+\gamma}}||f^{n}||_{L^{1}_{}} + A^{\epsilon}_{2}||f^{n-1}||_{L^{1}_{l}}||f^{n}||_{L^{1}_{\gamma}} \nonumber \\&&+ C(l)A^{\epsilon}_{2}||f^{n-1}||_{L^{1}_{l}}||f^{n}||_{L^{1}_{l}}.
\nonumber
\end{eqnarray}
Recalling (\ref{fnl1lthirdterm}), we have
\begin{eqnarray}\label{lastfnl1lql}
\langle Q_{L}(f^{n-1},f^{n}), \langle v \rangle^{l} \rangle
&\leq& -l\Lambda||f^{n-1}||_{L^{1}}||f^{n}||_{L^{1}_{l+\gamma}} + (4l+2)l\Lambda||f^{n-1}||_{L^{1}_{4}}||f^{n}||_{L^{1}_{l}}.
\end{eqnarray}
With (\ref{lastfnl1lqep}) and (\ref{lastfnl1lql}) in hand, we have
\begin{eqnarray}\label{fnl1lpartial1}
&&\frac{d}{dt}||f^{n}||_{L^{1}_{l}} + \frac{9}{8}A^{\epsilon}_{2}||f_{0}||_{L^{1}}||f^{n}||_{L^{1}_{l+\gamma}} \\&\leq& A^{\epsilon}_{2}||f_{0}||_{L^{1}}||f^{n-1}||_{L^{1}_{l+\gamma}} +
C(\epsilon,l,\Lambda)||f^{n-1}||_{L^{1}_{l}}||f^{n}||_{L^{1}_{l}}, \nonumber
\end{eqnarray}
where we denote $C(\epsilon,l,\Lambda) = A^{\epsilon}_{2}+\frac{3}{2}c_{1}A^{\epsilon}_{2}+C(l)A^{\epsilon}_{2}+ \epsilon^{2-2s}(4l+2)l\Lambda$. For simplicity, denote $m(l) = ||f_{0}||_{L^{1}_{l}}$.
For any $n \in \N$ and $l \geq 0$, define
\beno
T^{*}(l) = \frac{\min\{\log\{\frac{11C(\epsilon,l,\Lambda)m^{2}(l)}{A^{\epsilon}_{2}m(0)m(l+\gamma)}+1\}, \log(10/9)\}}{{11C(\epsilon,l,\Lambda)m(l)}},
\eeno
and
\beno
C_{n,l} = \sup_{0 \leq t \leq T^{*}(l)}\sup_{0 \leq k \leq n} ||f^{k}(t)||_{L^{1}_{l}}.
\eeno
We claim that for any $n \in \N$,
\begin{eqnarray}\label{ubl1claim}
C_{n,l} \leq 11m(l).
\end{eqnarray}
We will prove (\ref{ubl1claim}) by induction. First, it is obvious $C_{0,l} \leq 11m(l)$. Next, fix a $n \geq 1$, suppose $C_{n-1,l} \leq 11m(l)$, then on the interval $[0, T^{*}(l)]$, for any $1 \leq k \leq n$, from (\ref{fnl1lpartial1}), we have
\begin{eqnarray}\label{fnl1lpartial2}
&&\frac{d}{dt}||f^{k}||_{L^{1}_{l}} + \frac{9}{8}A^{\epsilon}_{2}m(0)||f^{k}||_{L^{1}_{l+\gamma}}  \\&\leq& A^{\epsilon}_{2}m(0)||f^{k-1}||_{L^{1}_{l+\gamma}} +
11C(\epsilon,l,\Lambda)m(l)||f^{k}||_{L^{1}_{l}}, \nonumber
\end{eqnarray}
Thus for any $t \in [0, T^{*}(l)]$ and $1 \leq k \leq n$, we derive that
\begin{eqnarray}\label{fkl1llintegral}
e^{-11C(\epsilon,l,\Lambda)m(l)t}||f^{k}(t)||_{L^{1}_{l}} + \frac{9}{8}A^{\epsilon}_{2}m(0)\int_{0}^{t}e^{-11C(\epsilon,l,\Lambda)m(l)s}||f^{k}(s)||_{L^{1}_{l+\gamma}}ds \nonumber \\
\leq A^{\epsilon}_{2}m(0)\int_{0}^{t}e^{-11C(\epsilon,l,\Lambda)m(l)s}||f^{k-1}(s)||_{L^{1}_{l+\gamma}}ds +
m(l). \nonumber
\end{eqnarray}
Multiplying the above inequality by $(\frac{8}{9})^{n-k}$ and taking sum over $1\leq k \leq n$, we obtain
\beno
e^{-11C(\epsilon,l,\Lambda)m(l)t}\sum_{k=1}^{n}(\frac{8}{9})^{n-k}||f^{k}(t)||_{L^{1}_{l}}+\frac{9}{8}A^{\epsilon}_{2}m(0)
\int_{0}^{t}e^{-11C(\epsilon,l,\Lambda)m(l)s}||f^{n}(s)||_{L^{1}_{l+\gamma}}ds \\
\leq (\frac{8}{9})^{n-1}A^{\epsilon}_{2}m(0)m(l+\gamma)\frac{1-e^{-11C(\epsilon,l,\Lambda)m(l)t}}{11C(\epsilon,l,\Lambda) m(l)} + m(l)\sum_{k=1}^{n}(\frac{8}{9})^{n-k}.
\eeno
Observing that $\sum_{k=1}^{n}(\frac{8}{9})^{n-k} \leq 9$, we arrive at
\beno
\sum_{k=1}^{n}(\frac{8}{9})^{n-k}||f^{k}(t)||_{L^{1}_{l}}+\frac{9}{8}A^{\epsilon}_{2}m(0)
\int_{0}^{t}e^{11C(\epsilon,l,\Lambda)m(l)(t-s)}||f^{n}(s)||_{L^{1}_{l+\gamma}}ds \\
\leq A^{\epsilon}_{2}m(0)m(l+\gamma)||f_{0}||_{L^{1}_{l+\gamma}}\frac{e^{11C(\epsilon,l,\Lambda)m(l)t}-1}{11C(\epsilon,l,\Lambda)m(l)} + 9m(l)e^{11C(\epsilon,l,\Lambda)m(l)t}.
\eeno
Thus we have
\beno
\sup_{0 \leq t \leq T^{*}(l)}||f^{n}(t)||_{L^{1}_{l}} &\leq& A^{\epsilon}_{2}m(0)||f_{0}||_{L^{1}_{l+\gamma}}\frac{e^{11C(\epsilon,l,\Lambda)m(l)T^{*}(l)}-1}{11C(\epsilon,l,\Lambda)m(l)} + 9||f_{0}||_{L^{1}_{l}}e^{11C(\epsilon,l,\Lambda)m(l)T^{*}(l)} \\&\leq& 11m(l),
\eeno
by the definition of $T^{*}(l)$. That is, $C_{n} \leq 11m(l)$. Therefore the claim  (\ref{ubl1claim}) is proved, which impiles
\begin{eqnarray}\label{fnuniform11ml}
 \sup_{0 \leq t \leq T^{*}(l)}\sup_{n \geq 0} ||f^{n}(t)||_{L^{1}_{l}} \leq 11m(l).
\end{eqnarray}

\noindent {\textit {Step 2}:} (Uniform $H^{N}_{l}$ Upper Bound)\\
In this step, we shall use the energy estimate to get the uniform
upper bound of $L^{N}_{l}$ norm of $f^{n}$ with respect to $n$. Fix an $\alpha$ with $|\alpha| \leq N$, one has
\beno
\partial_{t}\partial^{\alpha}_{v}f^{n} = \sum_{\alpha_{1}+\alpha_{2}=\alpha} \binom{\alpha}{\alpha_{1}}[Q^{\epsilon}(\partial^{\alpha_{1}}_{v}f^{n-1},\partial^{\alpha_{2}}_{v}f^{n})
+\epsilon^{2-2s}Q_{L}(\partial^{\alpha_{1}}_{v}f^{n-1},\partial^{\alpha_{2}}_{v}f^{n})].
\eeno

As before, we have
\beno
&&\langle  M^{\epsilon}(f^{n-1},\partial^{\alpha}_{v}f^{n}), \partial^{\alpha}_{v}f^{n} \langle v \rangle^{2l} \rangle \\&=&
\langle M^{\epsilon}(f^{n-1},\partial^{\alpha}_{v}f^{n}\langle v \rangle^{l}), \partial^{\alpha}_{v}f^{n} \langle v \rangle^{l} \rangle \nonumber \\&&+\{
\langle  M^{\epsilon}(f^{n-1},\partial^{\alpha}_{v}f^{n})\langle v \rangle^{l} -M^{\epsilon}(f^{n-1},\partial^{\alpha}_{v}f^{n}\langle v \rangle^{l}), \partial^{\alpha}_{v}f^{n} \langle v \rangle^{l} \rangle\}
\eeno

By coercivity estimate (\ref{coercivityboltz}) and commutator estimates (\ref{commcboltz1}), (\ref{commlandau}), we have
\begin{eqnarray}\label{goodtermbycoer}
&&\langle  M^{\epsilon}(f^{n-1},\partial^{\alpha}_{v}f^{n}), \partial^{\alpha}_{v}f^{n} \langle v \rangle^{2l} \rangle + C_{1}(f_{0})||\partial^{\alpha}_{v}f^{n}||^{2}_{\epsilon,l+\gamma/2} \\&\lesssim&
C_{2}(f_{0})||\partial^{\alpha}_{v}f^{n}||^{2}_{L^{2}_{l+\gamma/2}}
 +||f^{n-1}||_{L^{1}_{2l+1}}||\partial^{\alpha}_{v}f^{n}||_{\epsilon,l+\gamma/2}
||\partial^{\alpha}_{v}f^{n}||_{L^{2}_{l+\gamma/2}}. \nonumber
\end{eqnarray}
By upper bound estimate (\ref{upcboltz}) and commutator estimates (\ref{commcboltz1}), (\ref{commlandau}), for $|\alpha_{2}| \leq N-1$, we have
\begin{eqnarray}\label{btbyupperandcommu}
&&\langle  M^{\epsilon}(\partial^{\alpha_{1}}_{v}f^{n-1},\partial^{\alpha_{2}}_{v}f^{n}), \partial^{\alpha}_{v}f^{n} \langle v \rangle^{2l} \rangle \\&=&
\langle M^{\epsilon}(\partial^{\alpha_{1}}_{v}f^{n-1},\partial^{\alpha_{2}}_{v}f^{n}\langle v \rangle^{l}), \partial^{\alpha}_{v}f^{n} \langle v \rangle^{l} \rangle \nonumber \\ &&+\{
\langle  M^{\epsilon}(\partial^{\alpha_{1}}_{v}f^{n-1},\partial^{\alpha_{2}}_{v}f^{n})\langle v \rangle^{l}, \partial^{\alpha}_{v}f^{n} \langle v \rangle^{l} \rangle  \nonumber \\
&& -\langle M^{\epsilon}(\partial^{\alpha_{1}}_{v}f^{n-1},\partial^{\alpha_{2}}_{v}f^{n}\langle v \rangle^{l}), \partial^{\alpha}_{v}f^{n} \langle v \rangle^{l} \rangle\} \nonumber\\
&\lesssim& ||\partial^{\alpha_{1}}_{v}f^{n-1}||_{L^{1}_{\gamma+2}}||\partial^{\alpha_{2}}_{v}f^{n}||_{H^{s}_{l+\gamma/2+2}}
||\partial^{\alpha}_{v}f^{n}||_{H^{s}_{l+\gamma/2}} \nonumber\\&&
+\epsilon^{2-2s}||\partial^{\alpha_{1}}_{v}f^{n-1}||_{L^{1}_{\gamma+2}}||\partial^{\alpha_{2}}_{v}f^{n}||_{H^{1}_{l+\gamma/2+2}}
||\partial^{\alpha}_{v}f^{n}||_{H^{1}_{l+\gamma/2}} \nonumber\\
&&+||\partial^{\alpha_{1}}_{v}f^{n-1}||_{L^{1}_{2l+1}}||\partial^{\alpha_{2}}_{v}f^{n}||_{H^{s}_{l+\gamma/2}}
||\partial^{\alpha}_{v}f^{n}||_{L^{2}_{l+\gamma/2}} \nonumber\\&&+ \epsilon^{2-2s}||\partial^{\alpha_{1}}_{v}f^{n-1}||_{L^{1}_{\gamma+3}}||\partial^{\alpha_{2}}_{v}f^{n}||_{H^{1}_{l+\gamma/2}}
||\partial^{\alpha}_{v}f^{n}||_{L^{2}_{l+\gamma/2}}. \nonumber
\end{eqnarray}

When $N = 0$, by (\ref{goodtermbycoer}), we have
\beno
&&\langle  M^{\epsilon}(f^{n-1},f^{n}), f^{n} \langle v \rangle^{2l} \rangle + \frac{C_{1}(f_{0})}{2}||f^{n}||^{2}_{\epsilon,l+\gamma/2}
\\&\lesssim& C_{2}(f_{0})||f^{n}||^{2}_{L^{2}_{l+\gamma/2}}+\frac{1}{C_{1}(f_{0})}||f^{n-1}||^{2}_{L^{1}_{2l+1}}
||f^{n}||^{2}_{L^{2}_{l+\gamma/2}}.
\eeno
By (\ref{fnuniform11ml}), there holds
\beno
\sup_{0 \leq t \leq T^{*}(2l+1)}\sup_{n \geq 0}||f^{n}(t)||_{L^{1}_{2l+1}} \leq 11m(2l+1),
\eeno
so we have
\beno
\frac{d}{dt}||f^{n}||^{2}_{L^{2}_{l}} + C_{1}(f_{0})||f^{n}||^{2}_{\epsilon,l+\gamma/2} \lesssim C(||f_{0}||_{L^{1}_{2l+1}}, ||f_{0}||_{L\log L})  ||f^{n}||^{2}_{L^{2}_{l+\gamma/2}}.
\eeno
Thanks to the fact
\beno
||f^{n}||^{2}_{L^{2}_{l+\gamma/2}} \leq \eta||f^{n}||^{2}_{H^{s}_{l+\gamma/2}} + C(\eta) ||f^{n}||^{2}_{L^{1}_{l+\gamma/2}},
\eeno
we have
\beno
\frac{d}{dt}||f^{n}||^{2}_{L^{2}_{l}} + \frac{C_{1}(f_{0})}{2}||f^{n}||^{2}_{\epsilon,l+\gamma/2} \leq C(||f_{0}||_{L^{1}_{2l+1}}, ||f_{0}||_{L\log L}).
\eeno
By Gronwall's inequality, we obtain
\beno
\sup_{0 \leq t \leq T^{*}(2l+1)}||f^{n}(t)||_{L^{2}_{l}} + \int^{T^{*}(2l+1)}_{0}||f^{n}(s)||^{2}_{\epsilon,l+\gamma/2}ds \leq C(||f_{0}||_{L^{1}_{2l+1}}, ||f_{0}||_{L^{2}_{l}}).
\eeno



With the help of uniform $L^{2}_{l}$ norm and the above inequality,  we can prove in a similar manner as in the second step in the proof of theorem \ref{main2},
\beno
\sup_{0 \leq t \leq T^{*}(\phi(s,l))}||f^{n}(t)||_{H^{s}_{l}} \leq C(||f_{0}||_{L^{1}_{\phi(s,l)}}, ||f_{0}||_{H^{s}_{l}}).
\eeno
where $\phi(s,l) = \frac{(2l+4)(2+s)-2l}{s}$.

Now we turn to higher order regularity. Taking into account the fact $W^{\epsilon}(\xi) \leq \langle \xi \rangle$, for the fixed $\epsilon$, by (\ref{goodtermbycoer}) and (\ref{btbyupperandcommu}), we have
\beno
\frac{d}{dt}(\frac{1}{2}||f^{n}||^{2}_{H^{1}_{l}}) + \frac{C_{1}(f_{0})}{2}\epsilon^{2-2s}||f^{n}||^{2}_{H^{2}_{l+\gamma/2}} &\lesssim& ||f^{n}||^{2}_{H^{1}_{l+\gamma/2}} + ||f^{n-1}||_{H^{1}_{2l+3}}||f^{n}||^{2}_{H^{1}_{l+\gamma/2}}  \\
&&  + ||f^{n-1}||_{H^{1}_{6}}||f^{n}||_{H^{1}_{l+2+\gamma/2}}||f^{n}||_{H^{2}_{l+\gamma/2}}.
\eeno
Thanks interpolation theory and Young's inequality, one has
\beno
||f^{n}||^{2}_{H^{1}_{l+\gamma/2}} &\leq& ||f^{n}||_{H^{2}_{l+\gamma/2}}||f^{n}||_{L^{2}_{l+\gamma/2}} \\
&\leq& \frac{C_{1}(f_{0})}{8}\epsilon^{2-2s}||f^{n}||^{2}_{H^{2}_{l+\gamma/2}} + \frac{2}{C_{1}(f_{0})\epsilon^{2-2s}}||f^{n}||^{2}_{L^{2}_{l+\gamma/2}},
\eeno
\beno
||f^{n-1}||_{H^{1}_{2l+3}}||f^{n}||^{2}_{H^{1}_{l+\gamma/2}} &\leq& ||f^{n-1}||^{1/2}_{H^{2}_{l+\gamma/2}}||f^{n-1}||^{1/2}_{L^{2}_{3l+6-\gamma/2}}
||f^{n}||_{H^{2}_{l+\gamma/2}}||f^{n}||_{L^{2}_{l+\gamma/2}} \\
&\leq& \frac{C_{1}(f_{0})}{8}\epsilon^{2-2s}||f^{n}||^{2}_{H^{2}_{l+\gamma/2}} \\&&+ \frac{2}{C_{1}(f_{0})\epsilon^{2-2s}}||f^{n-1}||_{H^{2}_{l+\gamma/2}}||f^{n-1}||_{L^{2}_{3l+6-\gamma/2}}||f^{n}||^{2}_{L^{2}_{l+\gamma/2}}
\\ &\leq& \frac{C_{1}(f_{0})}{8}\epsilon^{2-2s}||f^{n}||^{2}_{H^{2}_{l+\gamma/2}} + \frac{C_{1}(f_{0})}{32}\epsilon^{2-2s}||f^{n-1}||^{2}_{H^{2}_{l+\gamma/2}} \\ && + \frac{32}{(C_{1}(f_{0})\epsilon^{2-2s})^{3}}||f^{n-1}||^{2}_{L^{2}_{3l+6-\gamma/2}}||f^{n}||^{4}_{L^{2}_{l+\gamma/2}},
\eeno
and finally
\beno
&&||f^{n-1}||_{H^{1}_{6}}||f^{n}||_{H^{1}_{l+2+\gamma/2}}||f^{n}||_{H^{2}_{l+\gamma/2}}\\ &\leq& ||f^{n-1}||^{\frac{1-s}{2-s}}_{H^{2}_{6}}||f^{n-1}||^{\frac{1}{2-s}}_{H^{s}_{6}}
||f^{n}||^{3/2}_{H^{2}_{l+\gamma/2}}||f^{n}||^{1/2}_{L^{2}_{l+4+\gamma/2}} \\
&\leq& \frac{C_{1}(f_{0})}{8}\epsilon^{2-2s}||f^{n}||^{2}_{H^{2}_{l+\gamma/2}} \\&&+ \frac{1}{4}(\frac{32}{3C_{1}(f_{0})}\epsilon^{2-2s})^{3}||f^{n-1}||^{\frac{4(1-s)}{2-s}}_{H^{2}_{6}}||f^{n-1}||^{\frac{4}{2-s}}_{H^{s}_{6}}||f^{n}||^{2}_{L^{2}_{l+4+\gamma/2}}
\\ &\leq& \frac{C_{1}(f_{0})}{8}\epsilon^{2-2s}||f^{n}||^{2}_{H^{2}_{l+\gamma/2}} + \frac{C_{1}(f_{0})}{32}\epsilon^{2-2s}||f^{n-1}||^{2}_{H^{2}_{6}} \\ && +
\frac{1}{p}(q \eta)^{-\frac{p}{q}}(\frac{1}{4}(\frac{32}{3C_{1}(f_{0})}\epsilon^{2-2s})^{3})^{\frac{2-s}{s}}
||f^{n-1}||^{\frac{4}{s}}_{H^{s}_{6}}||f^{n}||^{\frac{2(2-s)}{s}}_{L^{2}_{l+4+\gamma/2}},
\eeno
we have used the Young's inequality (\ref{basicineq}) with $p = \frac{2-s}{s}, q = \frac{2-s}{s(1-s)}$ and $\eta = \frac{C_{1}(f_{0})}{32}\epsilon^{2-2s}$.
Thus we arrive at for any $n \geq 1$,
\beno
\frac{d}{dt}||f^{n}||^{2}_{H^{1}_{l}} + \frac{1}{4} C_{1}(f_{0})\epsilon^{2-2s}||f^{n}||^{2}_{H^{2}_{l+\gamma/2}} \leq \frac{1}{8} C_{1}(f_{0})\epsilon^{2-2s}||f^{n-1}||^{2}_{H^{2}_{l+\gamma/2}} + M,
\eeno
where $M$ is the uniform upper bound of $||f^{n}||_{H^{s}_{6}}, ||f^{n}||_{L^{2}_{l+5}}$ and $||f^{n}||_{L^{1}_{3l+6}}$ with respect to $n$ on the time interval $[0, T^{*}]$. Here $T^{*} = T^{*}(\max\{\phi(s,6), 3l+6\})$.  With the same technique as in dealing with (\ref{fnl1lpartial2}), we obtain
\beno
&&||f^{n}(t)||^{2}_{H^{1}_{l}} + \frac{1}{4} C_{1}(f_{0})\epsilon^{2-2s}\int_{0}^{t}||f^{n}(r)||^{2}_{H^{2}_{l+\gamma/2}}dr \\&\leq& \frac{1}{8} C_{1}(f_{0})\epsilon^{2-2s}||f_{0}||^{2}_{H^{2}_{l+\gamma/2}}t + 2(M t + ||f_{0}||^{2}_{H^{1}_{l}}).
\eeno
The above inequality is true for any $n \geq 1$ and $t \in [0, T^{*}]$, so we have the desired result
\beno
\sup_{n}\sup_{0 \leq t \leq T^{*}} ||f^{n}(t)||_{H^{1}_{l}} \leq C(||f_{0}||_{H^{2}_{l+\gamma/2}}, ||f_{0}||_{L^{1}_{\max\{\phi(s,6), 3l+6\}}}, T^{*}).
\eeno
Continuing the argument, there will be a function $q: \N \times \R_{+} \rightarrow \R$, such that
\begin{eqnarray}\label{hn1uniformupperbound}
\sup_{n}\sup_{0 \leq t \leq T^{*}(q(N,l))} ||f^{n}(t)||_{H^{N}_{l}} \leq C(||f_{0}||_{H^{N+1}_{l+\gamma/2}}, ||f_{0}||_{L^{1}_{q(N,l)+\gamma}}).
\end{eqnarray}

\noindent {\textit {Step 3}:} (Cauchy Sequence)\\
Now we are ready to prove $\{f^{n}\}_{n}$ is a Cauchy sequence in $L^{\infty}([0,T^{*}]; L^{1}_{l})$.
Set $h^{n} = f^{n+1} - f^{n}$ for $n \geq 0$.  Then for $n \geq 1$, $h^{n}$ is the solution to the following equation
\begin{equation}\label{differencefunction2} \left\{ \begin{aligned}
&\partial_{t}h^{n} = M^{\epsilon}(f^{n},h^{n}) + M^{\epsilon}(h^{n-1},f^{n}),\\
&h^{n}|_{t=0} = 0.
 \end{aligned} \right.
\end{equation}
As the same as (\ref{lastfnl1lqep}), we have
\begin{eqnarray}\label{hncauchyfirstterm}
\langle Q^{\epsilon }(f^{n},h^{n}), sgn(h^{n})\langle v \rangle^{l} \rangle &\leq& \langle Q^{\epsilon }(f^{n},|h^{n}|), \langle v \rangle^{l} \rangle \\
&\leq& -\frac{9}{8}A^{\epsilon}_{2}||f^{n}||_{L^{1}_{}}||h^{n}||_{L^{1}_{l+\gamma}}+\frac{3}{2}A^{\epsilon}_{2}||f^{n}||_{L^{1}_{\gamma}}||h^{n}||_{L^{1}_{l}} \nonumber \\
&& + A^{\epsilon}_{2}||f^{n}||_{L^{1}_{l+\gamma}}||h^{n}||_{L^{1}_{}} + A^{\epsilon}_{2}||f^{n}||_{L^{1}_{l}}||h^{n}||_{L^{1}_{\gamma}} \nonumber \\&&+ C(l)A^{\epsilon}_{2}||f^{n}||_{L^{1}_{l}}||h^{n}||_{L^{1}_{l}}. \nonumber
\end{eqnarray}
As the same as (\ref{hnl1lthirdterm}), we have
\begin{eqnarray}\label{hncauchysecondterm}
\langle Q_{L}(f^{n},h^{n}) ,  sgn(h^{n})\langle v \rangle^{l} \rangle_{v} &\leq& -\frac{\Lambda}{16}||f^{n}||_{L^{1}}||h^{n}||_{L^{1}_{l+\gamma}}+ \frac{\Lambda}{8}||f^{n}||_{L^{1}_{\gamma}}||h^{n}||_{L^{1}_{l}} \\&& + C(l)\Lambda||f^{n}||_{L^{1}_{l}}||h^{n}||_{L^{1}_{l}}. \nonumber
\end{eqnarray}
Applying proposition \ref{lemma1} again, we obtain
\begin{eqnarray}\label{hncauchythirdterm}
&&\langle Q^{\epsilon}(h^{n-1},f^{n}) ,  sgn(h^{n})\langle v \rangle^{l} \rangle \\&\leq& \int_{\R^{6}}\int_{\SS^{2}}b^{\epsilon}|v-v_{*}|^{\gamma}|h^{n-1}_{*}|f^{n}(\langle v^{\prime} \rangle^{l}+\langle v \rangle^{l})dvdv_{*}d\sigma \nonumber \\
&\leq& A^{\epsilon}_{2}(||f^{n}||_{L^{1}_{}}||h^{n-1}||_{L^{1}_{l+\gamma}}+||f^{n}||_{L^{1}_{\gamma}}||h^{n-1}||_{L^{1}_{l}}) \nonumber \\&&+ C(l)A^{\epsilon}_{2}||f^{n}||_{L^{1}_{l}}||h^{n-1}||_{L^{1}_{l}} \nonumber \\&&+
A^{\epsilon}||f^{n}||_{L^{1}_{l+\gamma}}||h^{n-1}||_{L^{1}_{\gamma}}. \nonumber
\end{eqnarray}
Recalling the Landau operator $Q_{L}$ can be rewritten as:
\beno
Q_{L}(g,h) = \sum_{i,j=1}^{3}(a_{ij}*g)\partial_{ij}h - (c*g)h,
\eeno
we have
\begin{eqnarray}\label{hncauchyfourthterm}
\langle Q_{L}(h^{n-1},f^{n}) ,  sgn(h^{n})\langle v \rangle^{l} \rangle &=& \sum_{i,j=1}^{3} \langle (a_{ij}*h^{n-1})\partial_{ij}f^{n}, sgn(h^{n})\langle v \rangle^{l} \rangle \\&&- \langle (c*h^{n-1})f^{n}, sgn(h^{n})\langle v \rangle^{l}\rangle \nonumber \\&\leq& \Lambda||h^{n-1}||_{L^{1}_{\gamma+2}}||f^{n}||_{H^{2}_{l+\gamma+4}} \nonumber \\&&+ 2\Lambda(\gamma+3)||h^{n-1}||_{L^{1}_{\gamma}}||f^{n}||_{L^{1}_{l+\gamma}}. \nonumber
\end{eqnarray}
Patch all together inequalities (\ref{hncauchyfirstterm}),(\ref{hncauchysecondterm}),(\ref{hncauchythirdterm}) and (\ref{hncauchyfourthterm}), we obtain
\begin{eqnarray}\label{finalforhncauchy}
&&\frac{d}{dt}||h^{n}||_{L^{1}_{l}} + \frac{9}{8}A^{\epsilon}_{2}m(0)||h^{n}||_{L^{1}_{l+\gamma}} \\
 &\leq& A^{\epsilon}_{2}m(0)||h^{n-1}||_{L^{1}_{l+\gamma}}+ K_{1}||h^{n}||_{L^{1}_{l}} + K_{2}||h^{n-1}||_{L^{1}_{l}}, \nonumber
\end{eqnarray}
where $K_{1}$ and $K_{2}$ are some constants depending at most on the uniform upper bound of $||f^{n}||_{H^{2}_{l+\gamma+4}}$, which is bounded by a constant depending on $||f_{0}||_{H^{3}_{l+3\gamma/2+4}}, ||f_{0}||_{L^{1}_{q(2,l+\gamma+4)+\gamma}}$.
With a similar argument as in the previous lemma, for any $t \in [0, T^{*}(q(2,l+\gamma+4))]$, we can conclude
\begin{eqnarray}\label{l1cauchysequence}
||h^{n}(t)||_{L^{1}_{l}} \leq (\frac{8}{9})^{n}M(t)\exp(\frac{9K_{2}t}{8}+K_{1}t),
\end{eqnarray}
where $M(t) = \frac{9}{8}A^{\epsilon}_{2}m(0)\int_{0}^{t} e^{-K_{1}s}||h^{0}(s)||_{L^{1}_{l+\gamma}}ds + 22m(l)$.
Thus $\sum_{n}||h^{n}(t)||_{L^{1}_{l}}$ is finite and $\{f^{n}(t)\}_{n \in \N}$ is a Cauchy sequence in $L^{1}_{l}$. Due to the arbitrariness of $t \in [0, T^{*}(q(2,l+\gamma+4))]$, there is a function $f \in L^{\infty}([0,T^{*}]; L^{1}_{l})$ such that
\beno
\lim_{n \rightarrow \infty} \sup_{0\leq t \leq T^{*}}||f^{n}(t)-f(t)||_{L^{1}_{l}}  = 0
\eeno

In the following, we prove  $\{f^{n}(t)\}_{n \in \N}$ is a Cauchy sequence in $H^{N}_{l}$.  \\
Fix an $\alpha$ with $|\alpha| \leq N$, one has
\beno
\partial_{t}\partial^{\alpha}_{v}h^{n} = \sum_{\alpha_{1}+\alpha_{2}=\alpha} \binom{\alpha}{\alpha_{1}}[M^{\epsilon}(\partial^{\alpha_{1}}_{v}f^{n},\partial^{\alpha_{2}}_{v}h^{n})
+M^{\epsilon}(\partial^{\alpha_{1}}_{v}h^{n-1},\partial^{\alpha_{2}}_{v}f^{n})].
\eeno
Then we have
\beno
\frac{d}{dt}(\frac{1}{2}||\partial^{\alpha}_{v}h^{n}||^{2}_{L^{2}_{l}})&=& \sum_{\alpha_{1}+\alpha_{2}=\alpha} \binom{\alpha}{\alpha_{1}}[\langle M^{\epsilon}(\partial^{\alpha_{1}}_{v}f^{n},\partial^{\alpha_{2}}_{v}h^{n}), \partial^{\alpha}_{v}h^{n} \langle v \rangle^{2l} \rangle \\
&&+\langle M^{\epsilon}(\partial^{\alpha_{1}}_{v}h^{n-1},\partial^{\alpha_{2}}_{v}f^{n}), \partial^{\alpha}_{v}h^{n} \langle v \rangle^{2l} \rangle]\\
&\eqdefa& \sum_{\alpha_{1}+\alpha_{2}=\alpha} \binom{\alpha}{\alpha_{1}} [\mathfrak{I}_{1}(\alpha_{1},\alpha_{2}) + \mathfrak{I}_{2}(\alpha_{1},\alpha_{2})].
\eeno
As the same as (\ref{goodtermbycoer}), on the time interval $[0, T^{*}(2l+1)]$, we have
\beno
&&\mathfrak{I}_{1}(0,\alpha) + C_{1}(f_{0})||\partial^{\alpha}_{v}h^{n}||^{2}_{\epsilon,l+\gamma/2} \\&\lesssim&
C_{2}(f_{0})||\partial^{\alpha}_{v}h^{n}||^{2}_{L^{2}_{l+\gamma/2}}
+||f^{n}||_{L^{1}_{2l+1}}||\partial^{\alpha}_{v}h^{n}||_{\epsilon,l+\gamma/2}
||\partial^{\alpha}_{v}h^{n}||_{L^{2}_{l+\gamma/2}},
\eeno
and thus
\beno
&&\mathfrak{I}_{1}(0,\alpha) + \frac{C_{1}(f_{0})}{2}||\partial^{\alpha}_{v}h^{n}||^{2}_{\epsilon,l+\gamma/2} \lesssim
C(||f_{0}||_{L^{1}_{2l+1}},||f_{0}||_{L^{2}})||\partial^{\alpha}_{v}h^{n}||_{L^{2}_{l+\gamma/2}}.
\eeno
As the same as (\ref{btbyupperandcommu}), for $|\alpha_{2}| \leq |\alpha|-1 \leq N-1$ and any $\eta > 0$, on the time interval $[0, T^{*}(q(N,2l+3))]$,  we have
\beno
\mathfrak{I}_{1}(\alpha_{1},\alpha_{2}) &\lesssim& ||\partial^{\alpha_{1}}_{v}f^{n}||_{L^{1}_{\gamma+2}}||\partial^{\alpha_{2}}_{v}h^{n}||_{H^{s}_{l+\gamma/2+2}}
||\partial^{\alpha}_{v}h^{n}||_{H^{s}_{l+\gamma/2}} \\&&
+\epsilon^{2-2s}||\partial^{\alpha_{1}}_{v}f^{n}||_{L^{1}_{\gamma+2}}||\partial^{\alpha_{2}}_{v}h^{n}||_{H^{1}_{l+\gamma/2+2}}
||\partial^{\alpha}_{v}h^{n}||_{H^{1}_{l+\gamma/2}} \\
&&+||\partial^{\alpha_{1}}_{v}f^{n}||_{L^{1}_{2l+1}}||\partial^{\alpha_{2}}_{v}h^{n}||_{H^{s}_{l+\gamma/2}}
||\partial^{\alpha}_{v}h^{n}||_{L^{2}_{l+\gamma/2}} \\&&+ \epsilon^{2-2s}||\partial^{\alpha_{1}}_{v}f^{n}||_{L^{1}_{\gamma+3}}||\partial^{\alpha_{2}}_{v}h^{n}||_{H^{1}_{l+\gamma/2}}
||\partial^{\alpha}_{v}h^{n}||_{L^{2}_{l+\gamma/2}}\\
&\lesssim& ||f^{n}||_{H^{N}_{2l+3}}||h^{n}||_{H^{N}_{l+\gamma/2+2}}||\partial^{\alpha}_{v}h^{n}||_{\epsilon,l+\gamma/2},
\eeno
which implies, for any $\eta > 0$,
\beno
\mathfrak{I}_{1}(\alpha_{1},\alpha_{2}) - \eta ||\partial^{\alpha}_{v}h^{n}||^{2}_{\epsilon,l+\gamma/2} \lesssim C(\eta,||f_{0}||_{H^{N+1}_{2l+3+\gamma/2}}, ||f_{0}||_{L^{1}_{q(N,2l+3)+\gamma}})||h^{n}||^{2}_{H^{N}_{l+\gamma/2+2}}.
\eeno
Similarly, on the time interval $[0, T^{*}(q(N+1,l+\gamma/2+2))]$, we have
\beno
\mathfrak{I}_{2}(\alpha_{1},\alpha_{2}) \lesssim ||h^{n-1}||_{H^{N}_{2l+3}}||f^{n}||_{H^{N+1}_{l+\gamma/2+2}}||\partial^{\alpha}_{v}h^{n}||_{\epsilon,l+\gamma/2},
\eeno
and so for any $\eta > 0$,
\beno
\mathfrak{I}_{2}(\alpha_{1},\alpha_{2}) - \eta ||\partial^{\alpha}_{v}h^{n}||^{2}_{\epsilon,l+\gamma/2} \lesssim C(\eta,||f_{0}||_{H^{N+2}_{l+\gamma+2}}, ||f_{0}||_{L^{1}_{q(N+1,l+\gamma/2+2)+\gamma}})||h^{n-1}||^{2}_{H^{N}_{2l+3}}.
\eeno
Taking a suitable $\eta$, we obtain
\beno
&&\frac{d}{dt}(\frac{1}{2}||\partial^{\alpha}_{v}h^{n}||^{2}_{L^{2}_{l}}) + \frac{C_{1}(f_{0})}{4} ||\partial^{\alpha}_{v}h^{n}||^{2}_{\epsilon,l+\gamma/2}  \\ &\lesssim&  C(||f_{0}||_{H^{N+1}_{2l+3+\gamma/2}}, ||f_{0}||_{L^{1}_{q(N,2l+3)+\gamma}})||h^{n}||^{2}_{H^{N}_{l+\gamma/2+2}} \\ &&+ C(||f_{0}||_{H^{N+2}_{l+\gamma+2}}, ||f_{0}||_{L^{1}_{q(N+1,l+\gamma/2+2)+\gamma}})||h^{n-1}||^{2}_{H^{N}_{2l+3}}.
\eeno
Now taking sum over $|\alpha| \leq N$, we arrive at
\beno
&& \frac{d}{dt}||h^{n}||^{2}_{H^{N}_{l}} + \frac{C_{1}(f_{0})}{2} ||h^{n}||^{2}_{H^{N+s}_{l+\gamma/2}} \\ &\lesssim&   C(||f_{0}||_{H^{N+1}_{2l+3+\gamma/2}}, ||f_{0}||_{L^{1}_{q(N,2l+3)+\gamma}})||h^{n}||^{2}_{H^{N}_{l+\gamma/2+2}} \\&&+ C(||f_{0}||_{H^{N+2}_{l+\gamma+2}}, ||f_{0}||_{L^{1}_{q(N+1,l+\gamma/2+2)+\gamma}})||h^{n-1}||^{2}_{H^{N}_{2l+3}}.
\eeno
By interpolation theory, one has
\beno
||h^{n}||^{2}_{H^{N}_{l+\gamma/2+2}} \leq \eta ||h^{n}||^{2}_{H^{N+s}_{l+\gamma/2}} + c(\eta) ||h^{n}||^{2}_{L^{1}_{w_{1}(N,l,s,\gamma)}},
\eeno
and
\beno
||h^{n-1}||^{2}_{H^{N}_{2l+3}} \leq \lambda ||h^{n-1}||^{2}_{H^{N+s}_{l+\gamma/2}} + c(\lambda) ||h^{n-1}||^{2}_{L^{1}_{w_{2}(N,l,s,\gamma)}},
\eeno
where $w_{1}(N,l,s,\gamma)=l+\gamma/2+\frac{2(N+s+2)}{s}$ and $w_{2}(N,l,s,\gamma) = \frac{(N+s+2)(2l+3)-(N+2)(l+\gamma/2)}{s}$. It is easy to check $w_{1} \leq w_{2}$.  Choosing suitable $\eta$ and $\lambda$, we have
\beno
&&\frac{d}{dt}||h^{n}||^{2}_{H^{N}_{l}} + \frac{C_{1}(f_{0})}{4} ||h^{n}||^{2}_{H^{N+s}_{l+\gamma/2}} \\ &\leq &  (\frac{8}{9})\frac{C_{1}(f_{0})}{4} ||h^{n-1}||^{2}_{H^{N+s}_{l+\gamma/2}} \\&&+ C(||f_{0}||_{H^{N+1}_{2l+3+\gamma/2}}, ||f_{0}||_{L^{1}_{q(N,2l+3)+\gamma}})||h^{n}||^{2}_{L^{1}_{w_{1}(N,l,s,\gamma)}}\\&&+ C(||f_{0}||_{H^{N+2}_{l+\gamma+2}}, ||f_{0}||_{L^{1}_{q(N+1,l+\gamma/2+2)+\gamma}})||h^{n-1}||^{2}_{L^{1}_{w_{2}(N,l,s,\gamma)}}.
\eeno
Thanks to (\ref{l1cauchysequence}), on the time interval $[0, T^{*}(q(2,w_{2}+\gamma+4))]$, there holds
\beno
||h^{n}(t)||_{L^{1}_{w_{2}}} \leq (\frac{8}{9})^{n}M(T^{*})\exp(\frac{9K_{2}T^{*}}{8}+K_{1}T^{*})  \eqdefa (\frac{8}{9})^{n}M
\eeno
where $T^{*} = T^{*}(q(w_{2}+\gamma+4,2))$ and $M(T^{*}) = \frac{9}{8}A^{\epsilon}_{2}m(0)\int_{0}^{T^{*}} e^{-K_{1}s}||h^{0}(s)||_{L^{1}_{w_{2}+\gamma}}ds + 22m(w_{2})$. For ease of notation, let $K_{3}=C(\eta,||f_{0}||_{H^{N+1}_{2l+3+\gamma/2}}, ||f_{0}||_{L^{1}_{q(N,2l+3)+\gamma}})$ and \\$K_{4} = C(\lambda,||f_{0}||_{H^{N+2}_{l+\gamma+2}}, ||f_{0}||_{L^{1}_{q(N+1,l+\gamma/2+2)+\gamma}})$. Then we have
\beno
&&\frac{d}{dt}||h^{n}||^{2}_{H^{N}_{l}} + \frac{C_{1}(f_{0})}{4} ||h^{n}||^{2}_{H^{N+s}_{l+\gamma/2}} \\&\leq&   (\frac{8}{9})\frac{C_{1}(f_{0})}{4} ||h^{n-1}||^{2}_{H^{N+s}_{l+\gamma/2}}+ M ( K_{3} + \frac{9}{8}K_{4})(\frac{8}{9})^{n}.
\eeno
Integrating both sides with respect to time over $[0, t]$ for any $t \in [0, T^{*}]$, we have
\beno
&&||h^{n}(t)||^{2}_{H^{N}_{l}} + \frac{C_{1}(f_{0})}{4} \int_{0}^{t}||h^{n}(r)||^{2}_{H^{N+s}_{l+\gamma/2}}dr \\&\leq&   (\frac{8}{9})\frac{C_{1}(f_{0})}{4} \int_{0}^{t}||h^{n-1}(r)||^{2}_{H^{N+s}_{l+\gamma/2}}dr+ M T^{*} ( K_{3} + \frac{9}{8}K_{4})(\frac{8}{9})^{n} \\
\\&\leq&   (\frac{8}{9})^{n}\frac{C_{1}(f_{0})}{4} \int_{0}^{t}||h^{0}(r)||^{2}_{H^{N+s}_{l+\gamma/2}}dr+ M T^{*} ( K_{3} + \frac{9}{8}K_{4})(\frac{8}{9})^{n}n.
\eeno
Thus $\sum_{n}||h^{n}(t)||_{H^{N}_{l}}$ is finite and $\{f^{n}(t)\}_{n \in \N}$ is a Cauchy sequence in $H^{N}_{l}$. So there is a function $f \in L^{\infty}([0,T^{*}]; H^{N}_{l})$ such that
\beno
\lim_{n \rightarrow \infty} \sup_{0\leq t \leq T^{*}}||f^{n}(t)-f(t)||_{H^{N}_{l}}  = 0.
\eeno
The condition on $f_{0}$ can be summarized by the definitions of $K_{3}, K_{4}$ and the previous step as
\beno
f_{0} \in H^{(N+2)\vee3}_{w_{H}(N,l)} \cap L^{1}_{w_{L}(N,l)}.
\eeno
Under this condition, actually $\{f^{n}(t)\}_{n \in \N}$ is a Cauchy sequence in $H^{N}_{l}\cap L^{1}_{w(N,l)}$.

It is obvious that $f$ is the solution to \eqref{approximated}. Because $f^{n}$ is non-negative, the limit function $f$ is also non-negative.\\
\noindent {\textit {Step 4}:} (Uniqueness)\\
Suppose $f, g \in L^{\infty}([0,T]; H^{2}_{l+\gamma+4})$ are two non-negative solutions to \eqref{approximated}.  Set $F = f-g$ and $G = f+g$. Then $F$ is a solution to the following equation,
\begin{equation}\label{forunique1} \left\{ \begin{aligned}
&\partial_{t}F = M^{\epsilon}(G,F) + M^{\epsilon}(F,G)\\
&F|_{t=0} = 0.
\end{aligned} \right.
\end{equation}
Note that the above equation is as the same as equation (\ref{differencefunction2}) if  $h^{n} = h^{n-1}$. Thus following the same argument until inequality (\ref{finalforhncauchy}), we have
\beno
\frac{d}{dt}||F||_{L^{1}_{l}} + \frac{1}{8}A^{\epsilon}_{2}||G||_{L^{1}_{}}||F||_{L^{1}_{l+\gamma}}
\leq  K||F||_{L^{1}_{l}}
\eeno
where $K$ is some constant depending on the uniform upper bound of $||G||_{H^{2}_{l+\gamma+4}}$. Note that the previous estimate holds true for $l \geq 4$. Therefore, our approximate equation \eqref{approximated} has at most one solution in the space $L^{\infty}([0,T]; H^{N}_{l})$ if $N \geq 2$ and $l \geq 8+\gamma$.


\end{proof}

\subsection{Improvement of the well-posedness result of approximate equation \eqref{approximated}}
In this subsection, by using the symmetric property of the collision operators, we will prove the propagation of $L^{1}_{l}$ and $H^{N}_{l}$ norms of the solution $f$ to \eqref{approximated} and then extend the lifespan $T^*$ in Lemma \ref{lemmapositive3} to be global. Thanks to Lemma \ref{lemmapositive3},   we may assume that solution $f^{\epsilon}$ to our approximate equation is non-negative and smooth. It means that in this subsection we only need to give the {\it a priori} estimates to the equation.

In order to prove the propagation of $L^{1}_{l}$ of the solution $f^{\epsilon}$, we first give two propositions. The first proposition is related to the Boltzmann operator, while the second deals with the Landau operator.
\begin{prop}\label{propsition1} Let $p \geq 3$ and $k_{p} = [\frac{p+1}{2}]$. Suppose
\begin{eqnarray}\label{newbinomial}
\Theta(v,v_{*}) \eqdefa \int_{\SS^2}b(\cos\theta)|v-v_{*}|^{\gamma}
(\langle v^{\prime} \rangle^{2p} + \langle v^{\prime}_{*} \rangle^{2p}- \langle v \rangle^{2p}-\langle v_{*} \rangle^{2p})d\sigma,
\end{eqnarray}
then one has
\beno
\Theta(v,v_{*})&\leq& -\frac{1}{4}A_{2}(\langle v \rangle^{2p+\gamma}+\langle v_{*} \rangle^{2p+\gamma}) +
\frac{1}{2}A_{2}(\langle v \rangle^{2p}\langle v_{*} \rangle^{\gamma}+\langle v \rangle^{\gamma}\langle v_{*} \rangle^{2p}) \\
&&+A_{2}\sum^{k_{p}}_{k=1}\binom{p}{k}\{\langle v \rangle^{2k+\gamma}\langle v_{*} \rangle^{2(p-k)}+
\langle v \rangle^{2(p-k)+\gamma}\langle v_{*} \rangle^{2k}+\\
&&+\langle v \rangle^{2(p-k)}\langle v_{*} \rangle^{2k+\gamma}+
\langle v \rangle^{2k}\langle v_{*} \rangle^{2(p-k)+\gamma}\}\\
&&+2p(p-1)A_{2}\sum^{k_{p}-1}_{k=0}\binom{p-2}{k}\{\langle v \rangle^{2(k+1)+\gamma}\langle v_{*} \rangle^{2(p-k-1)}+
\langle v \rangle^{2(p-k-1)+\gamma}\langle v_{*} \rangle^{2(k+1)}\\
&&+\langle v \rangle^{2(p-k-1)}\langle v_{*} \rangle^{2(k+1)+\gamma}+
\langle v \rangle^{2(k+1)}\langle v_{*} \rangle^{2(p-k-1)+\gamma}\}\\
&\leq&-\frac{1}{4}A_{2}(\langle v \rangle^{2p+\gamma}+\langle v_{*} \rangle^{2p+\gamma})+
2^{2p+1}A_{2}\langle v \rangle^{2p}\langle v_{*} \rangle^{2p}.
\eeno
\end{prop}
\begin{proof} One may refer to Lemma 3.6 in \cite{lm} for the proof.
\end{proof}
\begin{rmk}\label{casetwotofour}
Lemma 3.6 in \cite{lm} only deals with the case $p \geq 3$, however, the conclusion is also valid in the case $2 \leq p < 3$ but with a different and smaller coefficient coming out instead of the constant $\frac{1}{4}$ before the highest order $2p+\gamma$.
\end{rmk}
\begin{prop}\label{propsition2} Let $p > 2$ and $f$ be a non-negative function, then
\begin{eqnarray}\label{landaumoment}
\langle Q_{L}(f,f),\langle v \rangle^{p} \rangle \leq -\Lambda p ||f||_{L^{1}_{}}||f||_{L^{1}_{p+\gamma}}+ \Lambda p(4p+2)||f||_{L^{1}_{2}}||f||_{L^{1}_{p}}.
\end{eqnarray}
\end{prop}
\begin{proof} One may refer to \cite{dv1} for the proof.
\end{proof}

Now we are ready to prove the propagation of moments and the smoothness.

\noindent {\bf Proof of  Theorem \ref{main2}:}  The proof will be divided into four steps.

\noindent {\it Step 1: Propagation of the moments.}

We consider the $2l$ moment. Assume $l \geq 3$,  for the case $2 \leq l <3$, the proof is similar thanks to remark \ref{casetwotofour}. By the case By the definition of $M^{\epsilon}$, we have
\beno
\frac{d}{dt}||f^{\epsilon}||_{L^{1}_{2l}}  &=& \langle  Q^{\epsilon}(f^{\epsilon},f^{\epsilon}),  \langle v \rangle^{2l}  \rangle+
\epsilon^{2-2s} \langle Q_{L}(f^{\epsilon},f^{\epsilon}) ,  \langle v \rangle^{2l}  \rangle \\
&\eqdefa&\mathfrak{I}_{1}+\mathfrak{I}_{2}.
\eeno
The term $\mathfrak{I}_{1}$ can be written as:
\beno
\mathfrak{I}_{1}&=&\int_{\R^6\times\SS^2}b^{\epsilon}(\cos\theta)|v-v_{*}|^{\gamma}f^{\epsilon}_{*}f^{\epsilon}
(\langle v^{\prime} \rangle^{2l} - \langle v \rangle^{2l})d\sigma dv_{*}dv \\
&=& \frac{1}{2} \int_{\R^6\times\SS^2}b^{\epsilon}(\cos\theta)|v-v_{*}|^{\gamma}f^{\epsilon}_{*}f^{\epsilon}
(\langle v^{\prime} \rangle^{2l}+\langle v^{\prime}_{*} \rangle^{2l}- \langle v \rangle^{2l}-\langle v_{*} \rangle^{2l})d\sigma dv_{*}dv
\eeno
Let $A^{\epsilon}_{2} = \int_{\SS^2}b^{\epsilon}(\cos\theta)\sin^{2}\theta d\sigma$, then by proposition \ref{propsition1}, we have
\beno
\mathfrak{I}_{1}&\leq& -\frac{A^{\epsilon}_{2}}{4}||f^{\epsilon}||_{L^{1}_{0}}||f^{\epsilon}||_{L^{1}_{2l+\gamma}} +
\frac{A^{\epsilon}_{2}}{2}||f^{\epsilon}||_{L^{1}_{2}}||f^{\epsilon}||_{L^{1}_{2l}} \\
&&+A^{\epsilon}_{2}\sum^{k_{l}}_{k=1}\binom{l}{k}\{||f^{\epsilon}||_{L^{1}_{2k+2}}||f^{\epsilon}||_{L^{1}_{2(l-k)}}+
||f^{\epsilon}||_{L^{1}_{2k}}||f^{\epsilon}||_{L^{1}_{2(l-k)+2}}\}\\
&&+2l(l-1)A^{\epsilon}_{2}\sum^{k_{l}-1}_{k=0}\binom{l-2}{k}\{||f^{\epsilon}||_{L^{1}_{2(k+1)+2}}||f^{\epsilon}||_{L^{1}_{2(l-k-1)}}
+||f^{\epsilon}||_{L^{1}_{2(l-k)}}||f^{\epsilon}||_{L^{1}_{2(k+1)}}\},\\
\eeno
where we have used the assumption $\gamma \leq 2$.
By interpolation, for any $2 \leq p,q \leq 2l$ with $ p+q = 2l + 2$, we have
\beno
||f^{\epsilon}||_{L^{1}_{p}}||f^{\epsilon}||_{L^{1}_{q}} \leq ||f^{\epsilon}||_{L^{1}_{2}}||f^{\epsilon}||_{L^{1}_{2l}}.
\eeno
Using the fact $2 \sum^{k_{l}}_{k=1}\binom{l}{k} \leq 2^{l}$, we can conclude:
\begin{eqnarray}\label{uifl1i1}
\mathfrak{I}_{1}&\leq& - \frac{A^{\epsilon}_{2}}{4}||f^{\epsilon}||_{L^{1}_{0}}||f^{\epsilon}||_{L^{1}_{2l+\gamma}} +
2^{2l+1} A^{\epsilon}_{2}||f^{\epsilon}||_{L^{1}_{2}}||f^{\epsilon}||_{L^{1}_{2l}}.
\end{eqnarray}
For the term $\mathfrak{I}_{2}$, we apply proposition \ref{propsition2} with $p=2l$ and obtain
\beno
\mathfrak{I}_{2} &\leq&  - 2l \Lambda\epsilon^{2-2s}||f||_{L^{1}_{}}||f||_{L^{1}_{2l+\gamma}}+  2l(8l+2)\Lambda\epsilon^{2-2s}||f||_{L^{1}_{2}}||f||_{L^{1}_{2l}}.
\eeno
Let $0< \epsilon_{*}< \frac{\sqrt{2}}{2}$ be the point such that $A^{\epsilon_{*}}_{2} = \frac{A_{2}}{2}$, then for any $0< \epsilon \leq \epsilon_{*}$, we have
\beno
\frac{d}{dt}||f^{\epsilon}||_{L^{1}_{2l}} \leq - \frac{A_{2}}{8}||f^{\epsilon}||_{L^{1}}||f^{\epsilon}||_{L^{1}_{2l+\gamma}} +
(2^{2l+1} A_{2} + 2l(8l+2)\Lambda)||f^{\epsilon}||_{L^{1}_{2}}||f^{\epsilon}||_{L^{1}_{2l}}.
\eeno
For any $\eta>0$, there exists a constant $K_{1}(\eta, l)$ such that
 \beno
\langle v \rangle^{2l} \leq \eta \langle v \rangle^{2l+\gamma} + K_{1}(\eta, l).
\eeno
Thus we have $||f^{\epsilon}||_{L^{1}_{2l}} \leq \eta||f^{\epsilon}||_{L^{1}_{2l+\gamma}} + K_{1}(\eta, l)||f^{\epsilon}||_{L^{1}}$.
With the preservation of mass and energy, by denoting $K_{2}(l) = 2^{2l+1} A_{2} + 2l(8l+2)\Lambda$ and taking $\eta(f_{0}) = \frac{A||f_{0}||_{L^{1}}}{16K_{2}(l)||f_{0}||_{L^{1}_{2}}}$, we have
\beno
\frac{d}{dt}||f^{\epsilon}||_{L^{1}_{2l}} \leq - \frac{A_{2}}{16}||f_{0}||_{L^{1}}||f^{\epsilon}||_{L^{1}_{2l+\gamma}} +
K_{2}(l)K_{1}(\eta(f_{0}),l)||f_{0}||_{L^{1}_{2}}||f_{0}||_{L^{1}}.
\eeno
Let $a = K_{2}(l)K_{1}(\eta(f_{0}),l)||f_{0}||_{L^{1}_{2}}||f_{0}||_{L^{1}}$ and $b = - \frac{A_{2}}{16}||f_{0}||_{L^{1}}$, by Gronwall's inequality (\ref{gronwall}), we have the following:
\beno
||f^\epsilon(t)||_{L^{1}_{2l}} \leq ||f_{0}||_{L^{1}_{2l}}+\frac{a}{|b|} \eqdefa  ||f_{0}||_{L^{1}_{2l}} + K(f_{0},l).
\eeno
The constant $K(f_{0},l)$ depends only on $l$, $||f_{0}||_{L^{1}}$ and $||f_{0}||_{L^{1}_{2}}$.

\vskip 0.3cm

\noindent {\textit {Step 2: Propagation of $L^{2}_{l}$ norm.}}

By the definition of $M^{\epsilon}$, we have
\beno
\frac{d}{dt}(\frac{1}{2}||f^{\epsilon}||^{2}_{L^{2}_{l}})  &=& \langle M^{\epsilon}(f^{\epsilon},f^{\epsilon}\langle v \rangle^{l}),  f^{\epsilon} \langle v \rangle^{l}  \rangle +  \{\langle M^{\epsilon}(f^{\epsilon},f^{\epsilon})\langle v \rangle^{l} - M^{\epsilon}(f^{\epsilon},f^{\epsilon}\langle v \rangle^{l}), f^{\epsilon} \langle v \rangle^{l}\rangle\}\\
&\eqdefa&\mathfrak{I}_{1}+\mathfrak{I}_{2}.
\eeno
Applying coercivity  estimates of (\ref{coercivityboltz}) with $g = f^{\epsilon}, f=  f^{\epsilon} \langle v \rangle^{l}$, we have
\begin{eqnarray}\label{l2li1}
\mathfrak{I}_{1}\leq
-C_{1}(f_{0})||f^{\epsilon}||^{2}_{\epsilon,l+\gamma/2}+C_{2}(f_{0})||f^{\epsilon}||^{2}_{L^{2}_{l+\gamma/2}}.
\end{eqnarray}
Applying commutator estimates (\ref{commcboltz}) with $g = f^{\epsilon}, h = f^{\epsilon}, f=  f^{\epsilon} \langle v \rangle^{l}, N_{2} = l+\gamma/2, N_{3} = \gamma/2$ and $N_{1} = 2l+5$, we have
\beno
\mathfrak{I}_{2}&\lesssim& ||f^{\epsilon}||_{L^{1}_{2l+5}}(||f^{\epsilon}||_{H^{s}_{l+\gamma/2}} +\epsilon^{2-2s}||f^{\epsilon}||_{H^{1}_{l+\gamma/2}})||f^{\epsilon}||_{L^{2}_{l+\gamma/2}}.
\eeno
Thanks to the facts $||\cdot||^{2}_{H^{s}_{l+\gamma/2}} \leq ||\cdot||^{2}_{\epsilon,l+\gamma/2}$ and $\epsilon^{2-2s}||\cdot||^{2}_{H^{1}_{l+\gamma/2}} \leq ||\cdot||^{2}_{\epsilon,l+\gamma/2}$, we have
\begin{eqnarray}\label{l2li2}
\mathfrak{I}_{2} - \frac{C_{1}(f_{0})}{2}||f^{\epsilon}||^{2}_{\epsilon,l+\gamma/2} &\lesssim&  \frac{1}{C_{1}(f_{0})}||f^{\epsilon}||^{2}_{L^{1}_{2l+5}}||f^{\epsilon}||^{2}_{L^{2}_{l+\gamma/2}}.
\end{eqnarray}

Now patching together (\ref{l2li1}),and (\ref{l2li2}), we get
\beno
&&\frac{d}{dt}(\frac{1}{2}||f^{\epsilon}||^{2}_{L^{2}_{l}}) + \frac{C_{1}(f_{0})}{2}||f^{\epsilon}||^{2}_{\epsilon,l+\gamma/2}
\\&\lesssim& (C_{2}(f_{0})+\frac{1}{C_{1}(f_{0})}||f^{\epsilon}||^{2}_{L^{1}_{2l+5}}) ||f^{\epsilon}||^{2}_{L^{2}_{l+\gamma/2}}
\\&\lesssim&  C_{3}(||f_{0}||_{L^{1}_{2l+5}}, ||f_{0}||_{L \log L})
||f^{\epsilon}||^{2}_{L^{2}_{l+\gamma/2}},
\eeno
where the existence of $C_{3}(f_{0},l)=C_{3}(||f_{0}||_{L^{1}_{2l+5}}, ||f_{0}||_{L \log L}, l)$ is ensured by the previous step. \\
By applying (\ref{l2hsl1}) with $\lambda = \frac{C_{1}(f_{0})}{4C_{3}(f_{0},l)}$, we have
\beno
\frac{d}{dt}(\frac{1}{2}||f^{\epsilon}||^{2}_{L^{2}_{l}}) + \frac{C_{1}(f_{0})}{4}||f^{\epsilon}||^{2}_{\epsilon,l+\gamma/2} \lesssim C_{3}(f_{0},l)(\frac{C_{1}(f_{0})}{4C_{3}(f_{0},l)})^{-\frac{3}{2s}}||f^{\epsilon}||^{2}_{L^{1}_{l+\gamma/2}}.
\eeno
Thanks to Gronwall's inequality, there exists a constant $C(||f_{0}||_{L^{1}_{2l+5}}, ||f_{0}||_{L^{2}_{l}})$ such that for any $t \geq 0$,
\begin{eqnarray}\label{l22lclosef}
||f^{\epsilon}(t)||^{2}_{L^{2}_{l}} + \int^{t+1}_{t} ||f^{\epsilon}(r)||^{2}_{\epsilon,l+\gamma/2} dr &\leq&
C(||f_{0}||_{L^{1}_{2l+5}}, ||f_{0}||_{L^{2}_{l}}).
\end{eqnarray}
Inequality (\ref{uniformhn}) is obtained in the case of $N=0$.

\vskip 0.3cm

\noindent {\textit {Step 3: Propagation of $H^{s}_{l}$ norm.}}

We first introduce some notations for the fractional derivative. We set
\beno
\triangle_s f= \frac{(\tau_hf)(v)-f(v)}{|h|^{\f32+s}},\eeno
with $ (\tau_h f)(v)=f(v+h)$ and $0<s<1$. Then there holds
\beno
\triangle_s(fg)&=&\triangle_s f g +\tau_h f \triangle_s g\\&=& f \triangle_s g+ \triangle_s f\tau_h g.
\eeno
Due to the definition of the fractional Sobolev space, we observe that:
\ben\label{eqv1} \|g\|_{H^s}^2\sim \int_{|h|\le \f12} \|\triangle_s g\|_{L^2}^2 dh+\|g\|_{L^2}^2.\een
Moreover, we also have, for $|h|\le \f12$ and $ m \in \R$,
\ben\label{e4} \|  g \langle v\rangle^k\triangle_s \langle v\rangle^l\|_{H^m}\lesssim |h|^{-(\f12+s)}\|g  \langle v\rangle^{l+k}\|_{H^m}, \een
\ben\label{e5} \|  \tau_h g \|_{H^{m}_{l}}\sim  \|  g  \|_{H^{m}_{l}}, \een
and
\ben\label{eqv2}  &&\|g\|_{H^m_l}^2+ \int_{|h|\le \f12} \|\langle v\rangle^l \triangle_s g\|_{H^m}^2 dh\nonumber\\&&\sim
\|g\|_{H^m_l}^2+ \int_{|h|\le \f12} \|\triangle_s (g\langle v\rangle^l)\|_{H^m}^2 dh\sim \|g\|_{H^{m+s}_l}^2. \een
One may check the proof of (\ref{e4}), (\ref{e5}) and (\ref{eqv2}) in the appendix of \cite{he3}.

Let $g^{\epsilon} = f^{\epsilon}\langle v \rangle^{l}$. It is easy to check that $\triangle_s g^\epsilon$ solves the following equation:
\beno &&\pa_t(\triangle_s g^\epsilon)-M^\epsilon(f^\epsilon, \triangle_s g^\epsilon)\\&=&M^\epsilon(\triangle_s f^\epsilon,
\tau_h g^\epsilon)+[M^\epsilon(\triangle_s f^\epsilon, f^\epsilon)\langle v\rangle^l-M^\epsilon(\triangle_s f^\epsilon, f^\epsilon\langle v\rangle^l)]\\&&+[M^\epsilon(\tau_hf^\epsilon, \triangle_s f^\epsilon)\langle v\rangle^l-M^\epsilon(\tau_h f^\epsilon, \triangle_s f^\epsilon\langle v\rangle^l)]]\\&&+[M^\epsilon( \tau_h f^\epsilon, \tau_h f^\epsilon)\triangle_s \langle v\rangle^l-M^\epsilon(\tau_h f^\epsilon, \tau_hf^\epsilon\triangle_s \langle v\rangle^l)]\\&\eqdefa& \sum_{i=1}^4 F_i. \eeno

By the upper bound estimate (\ref{upcboltz}), noting $\gamma \leq 2$, we have
\beno
\langle F_{1}, \triangle_{s}g^{\epsilon} \rangle &\lesssim& ||\triangle_{s}f^{\epsilon}||_{L^{1}_{4}}(||g^{\epsilon}||_{H^{s}_{3}}
||\triangle_{s}g^{\epsilon}||_{H^{s}_{\gamma/2}} +
 \epsilon^{2-2s}||g^{\epsilon}||_{H^{1}_{3}}
||\triangle_{s}g^{\epsilon}||_{H^{1}_{\gamma/2}}),
\eeno
which implies, for any $\eta_{1}> 0$,
\begin{eqnarray}\label{hsf1}
\quad\quad\langle F_{1}, \triangle_{s}g^{\epsilon} \rangle - \eta_{1} ||\triangle_{s}g^{\epsilon}||^{2}_{\epsilon,\gamma/2} \lesssim  \frac{1}{2\eta_{1}}(||\triangle_{s}f^{\epsilon}||^{2}_{L^{1}_{4}}||g^{\epsilon}||^{2}_{H^{s}_{3}}
+\epsilon^{2-2s}||\triangle_{s}f^{\epsilon}||^{2}_{L^{1}_{4}}||g^{\epsilon}||^{2}_{H^{1}_{3}}).
\end{eqnarray}
By the commutator estimates (\ref{commcboltz1}) and (\ref{commlandau}), we have
\beno
\langle F_{2}, \triangle_{s}g^{\epsilon} \rangle &\lesssim& ||\triangle_{s}f^{\epsilon}||_{L^{1}_{2l+1}}||f^{\epsilon}||_{H^{s}_{l+\gamma/2}}
||\triangle_{s}g^{\epsilon}||_{L^{2}_{\gamma/2}}\\
&& + \epsilon^{2-2s}||\triangle_{s}f^{\epsilon}||_{L^{1}_{\gamma+3}}||f^{\epsilon}||_{H^{1}_{l+\gamma/2}}
||\triangle_{s}g^{\epsilon}||_{L^{2}_{\gamma/2}},
\eeno
which implies, for any $\eta_{1}> 0$,
\begin{eqnarray}\label{hsf2}
&& \langle F_{2}, \triangle_{s}g^{\epsilon} \rangle  - \eta_{1} ||\triangle_{s}g^{\epsilon}||^{2}_{\epsilon,\gamma/2}  \\&\lesssim&
\frac{1}{2\eta_{1}}(||\triangle_{s}f^{\epsilon}||^{2}_{L^{1}_{2l+1}}||f^{\epsilon}||^{2}_{H^{s}_{l+1}}
+\epsilon^{2-2s}||\triangle_{s}f^{\epsilon}||^{2}_{L^{1}_{5}}||f^{\epsilon}||^{2}_{H^{1}_{l+1}}). \nonumber
\end{eqnarray}
Similarly, we have
\beno
\langle F_{3}, \triangle_{s}g^{\epsilon} \rangle &\lesssim& ||f^{\epsilon}||_{L^{1}_{2l+1}}||\triangle_{s}f^{\epsilon}||_{H^{s}_{l+\gamma/2}}
||\triangle_{s}g^{\epsilon}||_{L^{2}_{\gamma/2}}\\
&& + \epsilon^{2-2s}||f^{\epsilon}||_{L^{1}_{\gamma+3}}||\triangle_{s}f^{\epsilon}||_{H^{1}_{l+\gamma/2}}
||\triangle_{s}g^{\epsilon}||_{L^{2}_{\gamma/2}},
\eeno
which implies, for any $\eta_{2}> 0$,
\begin{eqnarray}\label{hsf3}
&&\langle F_{3}, \triangle_{s}g^{\epsilon} \rangle - \eta_{2} ||\triangle_{s}f^{\epsilon}||^{2}_{\epsilon,l+\gamma/2} \\&\lesssim&
\frac{1}{2\eta_{2}}(||f^{\epsilon}||^{2}_{L^{1}_{2l+1}}||\triangle_{s}g^{\epsilon}||^{2}_{L^{2}_{1}}
+\epsilon^{2-2s}||f^{\epsilon}||^{2}_{L^{1}_{5}}||\triangle_{s}g^{\epsilon}||^{2}_{L^{2}_{1}}). \nonumber
\end{eqnarray}
Also by the upper bound estimate (\ref{upcboltz}), we have
\beno
\langle F_{4}, \triangle_{s}g^{\epsilon} \rangle &\lesssim&
||f^{\epsilon}||_{L^{1}_{l+5}}||f^{\epsilon}||_{H^{s}_{l+3}}
||(\triangle_{s}g^{\epsilon})(\triangle_{s}\langle v\rangle^l)||_{H^{s}_{-l+\gamma/2}}\\
&& +  \epsilon^{2-2s} ||f^{\epsilon}||_{L^{1}_{l+5}}||f^{\epsilon}||_{H^{1}_{l+3}}
||(\triangle_{s}g^{\epsilon})(\triangle_{s}\langle v\rangle^l)||_{H^{1}_{-l+\gamma/2}}  \\
&& +  ||f^{\epsilon}||_{L^{1}_{4}}||(\tau_{h}f^{\epsilon})(\triangle_{s}\langle v\rangle^l)||_{H^{s}_{3}}
||\triangle_{s}g^{\epsilon}||_{H^{s}_{\gamma/2}}  \\
&& +  \epsilon^{2-2s} ||f^{\epsilon}||_{L^{1}_{4}}||(\tau_{h}f^{\epsilon})(\triangle_{s}\langle v\rangle^l)||_{H^{1}_{3}}
||\triangle_{s}g^{\epsilon}||_{H^{1}_{\gamma/2}}  \\,
\eeno
which implies, for any $\eta_{1}> 0$,
\begin{eqnarray}\label{hsf4}
&&\langle F_{4}, \triangle_{s}g^{\epsilon} \rangle
- \eta_{1} ||\triangle_{s}g^{\epsilon}||^{2}_{\epsilon,\gamma/2} \\&\lesssim&
\frac{1}{\eta_{1}}|h|^{-(1+2s)}(||f^{\epsilon}||^{2}_{L^{1}_{l+5}}||f^{\epsilon}||^{2}_{H^{s}_{l+3}}
+\epsilon^{2-2s}||f^{\epsilon}||^{2}_{L^{1}_{l+5}}||f^{\epsilon}||^{2}_{H^{1}_{l+3}}). \nonumber
\end{eqnarray}
By the coercivity estimate (\ref{coercivityboltz}), we have
\begin{eqnarray}\label{conegative}
\langle M^\epsilon(f^\epsilon, \triangle_s g^\epsilon), \triangle_{s}g^{\epsilon} \rangle &\leq& -C_{1}(f_{0}) ||\triangle_{s}g^{\epsilon}||^{2}_{\epsilon,\gamma/2}+C_{2}(f_{0})||\triangle_{s}g^{\epsilon}||^{2}_{L^{2}_{\gamma/2}}.
\end{eqnarray}

Thanks to
\beno
\frac{d}{dt}(\frac{1}{2}||\triangle_s g^\epsilon||^{2}_{L^{2}}) &=&
\langle \partial_{t}\triangle_s g^\epsilon, \triangle_{s}g^{\epsilon} \rangle
= \langle M^\epsilon(f^\epsilon, \triangle_s g^\epsilon), \triangle_{s}g^{\epsilon} \rangle + \sum_{i=1}^4 \langle F_{i}, \triangle_{s}g^{\epsilon} \rangle,
\eeno
patching together all the above estimates, taking $\eta_{1} = \frac{C_{1}(f_{0})}{6}$ in (\ref{hsf1}),(\ref{hsf2}),(\ref{hsf4}),
we arrive at, for $|h|\leq \frac{1}{2}$,
\beno
&&\frac{d}{dt}(\frac{1}{2}||\triangle_s g^\epsilon||^{2}_{L^{2}}) + \frac{C_{1}(f_{0})}{2} ||\triangle_{s}g^{\epsilon}||^{2}_{\epsilon,\gamma/2} - \eta_{2}||\triangle_{s}f^{\epsilon}||^{2}_{\epsilon,l+\gamma/2} \\&\lesssim&
+C_{2}(f_{0})||\triangle_{s}g^{\epsilon}||^{2}_{L^{2}_{\gamma/2}} +\frac{1}{2\eta_{1}}(||\triangle_{s}f^{\epsilon}||^{2}_{L^{1}_{4}}||g^{\epsilon}||^{2}_{H^{s}_{3}}
+\epsilon^{2-2s}||\triangle_{s}f^{\epsilon}||^{2}_{L^{1}_{4}}||g^{\epsilon}||^{2}_{H^{1}_{3}})\\
&&+\frac{1}{2\eta_{1}}(||\triangle_{s}f^{\epsilon}||^{2}_{L^{1}_{2l+1}}||f^{\epsilon}||^{2}_{H^{s}_{l+1}}
+\epsilon^{2-2s}||\triangle_{s}f^{\epsilon}||^{2}_{L^{1}_{5}}||f^{\epsilon}||^{2}_{H^{1}_{l+1}})\\
&&+\frac{1}{2\eta_{2}}(||f^{\epsilon}||^{2}_{L^{1}_{2l+1}}||\triangle_{s}g^{\epsilon}||^{2}_{L^{2}_{1}}
+\epsilon^{2-2s}||f^{\epsilon}||^{2}_{L^{1}_{5}}||\triangle_{s}g^{\epsilon}||^{2}_{L^{2}_{1}})\\
&&+\frac{1}{\eta_{1}}|h|^{-(1+2s)}(||f^{\epsilon}||^{2}_{L^{1}_{l+5}}||f^{\epsilon}||^{2}_{H^{s}_{l+3}}
+\epsilon^{2-2s}||f^{\epsilon}||^{2}_{L^{1}_{l+5}}||f^{\epsilon}||^{2}_{H^{1}_{l+3}})
\\&\lesssim&
+C_{2}(f_{0})||\triangle_{s}g^{\epsilon}||^{2}_{L^{2}_{\gamma/2}}
+\frac{1}{\eta_{1}}||\triangle_{s}f^{\epsilon}||^{2}_{L^{2}_{6}}||g^{\epsilon}||^{2}_{\epsilon,3}\\
&&+\frac{1}{\eta_{1}}(||\triangle_{s}f^{\epsilon}||^{2}_{L^{2}_{2l+3}}||f^{\epsilon}||^{2}_{H^{s}_{l+1}}
+\epsilon^{2-2s}||\triangle_{s}f^{\epsilon}||^{2}_{L^{2}_{7}}||f^{\epsilon}||^{2}_{H^{1}_{l+1}})\\
&&+\frac{1}{\eta_{2}}||f^{\epsilon}||^{2}_{L^{1}_{2l+5}}||\triangle_{s}g^{\epsilon}||^{2}_{L^{2}_{1}}
+\frac{1}{\eta_{1}}|h|^{-(1+2s)}||f^{\epsilon}||^{2}_{L^{1}_{l+5}}||f^{\epsilon}||^{2}_{\epsilon,l+3}.
\eeno
where we have used the fact $||\langle \cdot \rangle^{-2}||_{L^{2}} \leq \sqrt{2} \pi$.
Integrating both sides from $0$ to $t$ with respect to time, we obtain
\begin{eqnarray}\label{integralwitht}
&&||\triangle_s g^\epsilon(t)||^{2}_{L^{2}}+ C_{1}(f_{0})\int^{t}_{0} ||\triangle_{s}g^{\epsilon}(r)||^{2}_{\epsilon,\gamma/2}dr
-2\eta_{2}\int^{t}_{0}||\triangle_{s}f^{\epsilon}(r)||^{2}_{\epsilon,l+\gamma/2}dr
\\&\lesssim& ||\triangle_s g^\epsilon(0)||^{2}_{L^{2}}+
C_{2}(f_{0})\int^{t}_{0}||\triangle_{s}g^{\epsilon}(r)||^{2}_{L^{2}_{\gamma/2}}dr
+ \frac{1}{\eta_{1}}\int^{t}_{0}||\triangle_{s}f^{\epsilon}(r)||^{2}_{L^{2}_{6}}||g^{\epsilon}(r)||^{2}_{\epsilon,3}dr
\nonumber
\\&&+\frac{1}{\eta_{1}}\int^{t}_{0}(||\triangle_{s}f^{\epsilon}(r)||^{2}_{L^{2}_{2l+3}}||f^{\epsilon}(r)||^{2}_{H^{s}_{l+1}}
+\epsilon^{2-2s}||\triangle_{s}f^{\epsilon}(r)||^{2}_{L^{2}_{7}}||f^{\epsilon}(r)||^{2}_{H^{1}_{l+1}})dr \nonumber \\
&&+ \frac{1}{\eta_{2}}\int^{t}_{0}||f^{\epsilon}(r)||^{2}_{L^{1}_{2l+5}}||\triangle_{s}g^{\epsilon}(r)||^{2}_{L^{2}_{1}}dr
+\frac{1}{\eta_{1}}|h|^{-(1+2s)}\int^{t}_{0}||f^{\epsilon}(r)||^{2}_{L^{1}_{l+5}}||f^{\epsilon}(r)||^{2}_{\epsilon,l+3}dr.
\nonumber
\end{eqnarray}
Integrating both sides on the Ball $B(0, \f12)$ with respect to the variable $h$, noting that \\$\int_{|h|\leq \frac{1}{2}} |h|^{-(1+2s)} dh$ is finite, thanks to the facts \eqref{eqv1} and \eqref{eqv2},  taking a small enough $\eta_{2}$, we derive that
\begin{eqnarray}\label{integralwithth}
&&|| g^\epsilon(t)||^{2}_{H^{s}}+\frac{C_{1}(f_{0})}{2}\int^{t}_{0} \int_{|h|\leq\frac{1}{2}}||\triangle_{s}g^{\epsilon}(r)||^{2}_{\epsilon,\gamma/2}dhdr\\
&\lesssim& || g^\epsilon(0)||^{2}_{H^{s}}+|| g^\epsilon(t)||^{2}_{L^{2}}
+C_{2}(f_{0})\int^{t}_{0}||g^{\epsilon}(r)||^{2}_{H^{s}_{1}}dr \nonumber \\
&& + \frac{1}{\eta_{1}}\int^{t}_{0}||f^{\epsilon}(r)||^{2}_{H^{s}_{6}}||g^{\epsilon}(r)||^{2}_{\epsilon,3}dr
 \nonumber\\
&&+\frac{1}{\eta_{1}}\int^{t}_{0}(||f^{\epsilon}(r)||^{2}_{H^{s}_{2l+3}}||f^{\epsilon}(r)||^{2}_{H^{s}_{l+1}}
+\epsilon^{2-2s}||f^{\epsilon}(r)||^{2}_{H^{s}_{7}}||f^{\epsilon}(r)||^{2}_{H^{1}_{l+1}})dr \nonumber \\
&&+ \frac{1}{\eta_{2}}\int^{t}_{0}||f^{\epsilon}(r)||^{2}_{L^{1}_{2l+5}}||g^{\epsilon}(r)||^{2}_{H^{s}_{1}}dr
+\frac{1}{\eta_{1}}\int^{t}_{0}||f^{\epsilon}(r)||^{2}_{L^{1}_{l+5}}||f^{\epsilon}(r)||^{2}_{\epsilon,l+3}dr.
\nonumber
\end{eqnarray}
Using the fact $||f^{\epsilon}||^{2}_{H^{s}_{2l+3}}||f^{\epsilon}||^{2}_{H^{s}_{l+1}} \leq ||f^{\epsilon}||^{2}_{H^{s}_{l}}||f^{\epsilon}||^{2}_{H^{s}_{2l+4}}$, substituting into the uniform bound of $||f^{\epsilon}||_{L^{1}_{2l+5}}$ and $||f^{\epsilon}||_{L^{2}_{l}}$, we have
\begin{eqnarray}\label{hslclose1}
&&|| f^\epsilon(t)||^{2}_{H^{s}_{l}}+\frac{C_{1}(f_{0})}{2}\int^{t}_{0} \int_{|h|\leq\frac{1}{2}}||\triangle_{s}g^{\epsilon}(r)||^{2}_{\epsilon,\gamma/2}dhdr
\\&\lesssim& || f_{0}||^{2}_{H^{s}_{l}}+C(||f_{0}||_{L^{2}_{l}},||f_{0}||_{L^{1}_{2l+5}})
+C(||f_{0}||_{L^{1}_{2l+5}},||f_{0}||_{L \log L})\int^{t}_{0}||f^{\epsilon}(r)||^{2}_{\epsilon,l+3}dr \nonumber \\
&& + C(||f_{0}||_{L^{1}_{1}},||f_{0}||_{L \log L})\int^{t}_{0}||f^{\epsilon}(r)||^{2}_{H^{s}_{l}}||f^{\epsilon}(r)||^{2}_{\epsilon,2l+4}dr.
 \nonumber
\end{eqnarray}
Actually, inequality (\ref{hslclose1}) holds true on any bounded interval. Therefore,  for any $t_{1}<t_{2}$ with $t_{2}-t_{1} \leq 2$, we have
\begin{eqnarray}\label{hslclose3}
&& || f^\epsilon(t_{2})||^{2}_{H^{s}_{l}}+\frac{C_{1}(f_{0})}{2}\int^{t_{2}}_{t_{1}} \int_{|h|\leq\frac{1}{2}}||\triangle_{s}g^{\epsilon}(r)||^{2}_{\epsilon,\gamma/2}dhdr\\
&\lesssim& || f^\epsilon(t_{1})||^{2}_{H^{s}_{l}}+C(||f_{0}||_{L^{2}_{l}},||f_{0}||_{L^{1}_{2l+5}})
+C(||f_{0}||_{L^{1}_{2l+5}},||f_{0}||_{L \log L})\int^{t_{2}}_{t_{1}}||f^{\epsilon}(r)||^{2}_{\epsilon,l+3}dr \nonumber \\
&&+ C(||f_{0}||_{L^{1}_{1}},||f_{0}||_{L \log L})\int^{t_{2}}_{t_{1}}||f^{\epsilon}(r)||^{2}_{H^{s}_{l}}||f^{\epsilon}(r)||^{2}_{\epsilon,2l+4}dr.
 \nonumber
\end{eqnarray}
By Gronwall's inequality (\ref{gronwall2}) and uniform upper bound (\ref{l22lclosef}) for integral of $||f^{\epsilon}||^{2}_{\epsilon,l}$ on any bounded interval, we arrive at
\begin{eqnarray}\label{hslclose4}
|| f^\epsilon(t_{2})||^{2}_{H^{s}_{l}}  &\lesssim& C(||f_{0}||_{L^{1}_{4l+13}},||f_{0}||_{L^{2}_{2l+4}})
\{|| f^\epsilon(t_{1})||^{2}_{H^{s}_{l}}+C(||f_{0}||_{L^{1}_{2l+11}},||f_{0}||_{L^{2}_{l+3}}) \}.
\end{eqnarray}
Also from (\ref{l22lclosef}), we conclude that, in any unit interval $[t,t+1]$, there exists at least one point $t_{*}$ such that
\begin{eqnarray}\label{hslclose5}
|| f^\epsilon(t_{*})||^{2}_{H^{s}_{l}}  &\lesssim& C(||f_{0}||_{L^{1}_{2l+5}},||f_{0}||_{L^{2}_{l}}).
\end{eqnarray}
Combining (\ref{hslclose4}) and (\ref{hslclose5}), we have
\begin{eqnarray}\label{hslclose6}
|| f^\epsilon(t)||^{2}_{H^{s}_{l}}  &\lesssim& C(||f_{0}||_{H^{s}_{l}},||f_{0}||_{L^{1}_{4l+13}},||f_{0}||_{L^{2}_{2l+4}}).
\end{eqnarray}
Together with (\ref{hslclose3}), we finally arrive at
\begin{eqnarray}\label{hslclose7}
&&|| f^\epsilon(t)||^{2}_{H^{s}_{l}}+\frac{C_{1}(f_{0})}{2}\int^{t+1}_{t} \int_{|h|\leq\frac{1}{2}}||\triangle_{s}g^{\epsilon}(r)||^{2}_{\epsilon,\gamma/2}dhdr \\
&\lesssim& C(||f_{0}||_{H^{s}_{l}},||f_{0}||_{L^{1}_{4l+13}},||f_{0}||_{L^{2}_{2l+4}}). \nonumber
\end{eqnarray}
By interpolation theory, there holds
\beno
||f_{0}||_{L^{2}_{2l+4}} \lesssim ||f_{0}||_{H^{s}_{l}} + ||f_{0}||_{H^{-2}_{\phi(s,l)}} \lesssim ||f_{0}||_{H^{s}_{l}} + ||f_{0}||_{L^{1}_{\phi(s,l)}},
\eeno
where $\phi(s,l) = \frac{(2l+4)(2+s)-2l}{s} \geq 4l+13$. Therefore we have
\begin{eqnarray}\label{hslclose8}
|| f^\epsilon(t)||^{2}_{H^{s}_{l}} &\lesssim& C(||f_{0}||_{H^{s}_{l}},||f_{0}||_{L^{1}_{\phi(s,l)}}).
\end{eqnarray}

\vskip 0.3cm

\noindent {\textit {Step 4: Propagation of $H^{N}_{l}$ norm when $N \geq 1$.}}

We prove the propagation by induction on $N$. Let $m \geq 1$ be an integer. Suppose inequality (\ref{uniformhn}) holds true for all $N \leq m-1$, we now prove that it is also valid for $N = m$.

Set $g^\epsilon =\pa^\alpha_{v} f^\epsilon \langle v\rangle^l$ with $|\alpha|
\leq m$, then $g^\epsilon$ solves
\begin{eqnarray}\label{hmlequation}
\pa_t g^\epsilon &=&M^{\epsilon}(f^\epsilon, g^\epsilon )+[M^{\epsilon}(f^\epsilon, \pa^{\alpha}_v f^\epsilon)\langle v\rangle^l-M^{\epsilon}(f^\epsilon,
\pa^{\alpha}_v f^\epsilon\langle v\rangle^l)]\\
&&+\sum_{|\alpha_1|\geq1, \alpha_1+\alpha_2=\alpha}
\binom{\alpha}{\alpha_{1}} M^{\epsilon}(\pa^{\alpha_1}_v f^\epsilon, \pa^{\alpha_2}_v f^\epsilon)
\langle v \rangle^l. \nonumber
\end{eqnarray}
By the coercivity estimate (\ref{coercivityboltz}), we have
\begin{eqnarray}\label{hmls1}
\langle M^{\epsilon}(f^\epsilon, g^\epsilon ), g^{\epsilon} \rangle \leq -C_{1}(f_{0})||g^{\epsilon}||^{2}_{\epsilon,\gamma/2} +
C_{2}(f_{0})||g^{\epsilon}||^{2}_{L^{2}_{\gamma/2}}.
\end{eqnarray}
By the commutator estimate (\ref{commcboltz}), we have
\beno
|\langle M^{\epsilon}(f^\epsilon, \pa^{\alpha}_v f^\epsilon)\langle v\rangle^l-M^{\epsilon}(f^\epsilon,
\pa^{\alpha}_v f^\epsilon\langle v\rangle^l), g^{\epsilon} \rangle| \lesssim ||f^{\epsilon}||_{L^{1}_{2l+5}}||g^{\epsilon}||_{\epsilon,\gamma/2}
||g^{\epsilon}||_{L^{2}_{\gamma/2}},
\eeno
which implies, for any $\eta_{1} > 0$,
\begin{eqnarray}\label{hmls2}
&&|\langle M^{\epsilon}(f^\epsilon, \pa^{\alpha}_v f^\epsilon)\langle v\rangle^l-M^{\epsilon}(f^\epsilon,
\pa^{\alpha}_v f^\epsilon\langle v\rangle^l), g^{\epsilon} \rangle| - \eta_{1} ||g^{\epsilon}||^{2}_{\epsilon,\gamma/2}
\\&\lesssim&  \frac{1}{\eta_{1}}||f^{\epsilon}||^{2}_{L^{1}_{2l+5}}||g^{\epsilon}||^{2}_{L^{2}_{\gamma/2}}. \nonumber
\end{eqnarray}

For the remaining terms in the right hand of (\ref{hmlequation}) with $|\alpha_{1}| \geq 1$, we split each of them into two terms:
\beno
M^{\epsilon}(\pa^{\alpha_1}_v f^\epsilon, \pa^{\alpha_2}_v f^\epsilon)
\langle v \rangle^l &=& \{M^{\epsilon}(\pa^{\alpha_1}_v f^\epsilon, \pa^{\alpha_2}_v f^\epsilon)
\langle v \rangle^l - M^{\epsilon}(\pa^{\alpha_1}_v f^\epsilon, \pa^{\alpha_2}_v f^\epsilon \langle v \rangle^l)\} \\
&&+M^{\epsilon}(\pa^{\alpha_1}_v f^\epsilon, \pa^{\alpha_2}_v f^\epsilon \langle v \rangle^l)\\
&\eqdefa& \mathfrak{I}_{1}+\mathfrak{I}_{2}.
\eeno
By the commutator estimate (\ref{commcboltz}), for the case $|\alpha_{1}|=|\alpha| \leq m$,  we have
\beno
|\langle \mathfrak{I}_{1}, g^{\epsilon} \rangle| \lesssim ||\pa^{\alpha_1}_v f^{\epsilon}||_{L^{1}_{2l+5}}||f^{\epsilon}||_{\epsilon,l+\gamma/2}
||g^{\epsilon}||_{L^{2}_{\gamma/2}}
\lesssim ||f^{\epsilon}||_{H^{m}_{2l+7}}||f^{\epsilon}||_{\epsilon,l+\gamma/2}
||g^{\epsilon}||_{L^{2}_{\gamma/2}},
\eeno
which implies, for any $\eta_{2} > 0$,
\begin{eqnarray}\label{hmls4}
|\langle \mathfrak{I}_{1}, g^{\epsilon} \rangle| - \eta_{2}||g^{\epsilon}||^{2}_{\epsilon,\gamma/2}
\lesssim \frac{1}{\eta_{2}}||f^{\epsilon}||^{2}_{H^{m}_{2l+7}}||f^{\epsilon}||^{2}_{\epsilon,l+\gamma/2}
\end{eqnarray}
For the case $1 \leq |\alpha_{1}|\leq |\alpha| - 1 \leq  m-1$, we have
\beno
|\langle \mathfrak{I}_{1}, g^{\epsilon} \rangle| &\lesssim& ||\pa^{\alpha_1}_v f^{\epsilon}||_{L^{1}_{2l+5}}||\pa^{\alpha_2}_v f^{\epsilon}||_{\epsilon,l+\gamma/2}
||g^{\epsilon}||_{L^{2}_{\gamma/2}} \\
&\lesssim&(||f^{\epsilon}||_{H^{1}_{2l+7}}||f^{\epsilon}||_{H^{m}_{l+\gamma/2}}\textbf{1}_{m \geq 2}+
||f^{\epsilon}||_{H^{m-1}_{2l+7}}||f^{\epsilon}||_{H^{m-1}_{l+\gamma/2}})
||g^{\epsilon}||_{L^{2}_{\gamma/2}},
\eeno
which implies, for any $\eta_{2} > 0$,
\begin{eqnarray}\label{hmls5}
&&|\langle \mathfrak{I}_{1}, g^{\epsilon} \rangle| - \eta_{2}||g^{\epsilon}||^{2}_{\epsilon,\gamma/2} \nonumber
\\&\lesssim&  \frac{1}{\eta_{2}}
(||f^{\epsilon}||^{2}_{H^{1}_{2l+7}}||f^{\epsilon}||^{2}_{H^{m}_{l+\gamma/2}}\textbf{1}_{m \geq 2}+
||f^{\epsilon}||^{2}_{H^{m-1}_{2l+7}}||f^{\epsilon}||^{2}_{H^{m-1}_{l+\gamma/2}}). \nonumber
\end{eqnarray}
By the upper bound estimate (\ref{upcboltz}), for the case $|\alpha_{1}|=|\alpha| \leq m$, we have,
\beno
|\langle \mathfrak{I}_{2}, g^{\epsilon} \rangle| \lesssim ||\pa^{\alpha_1}_v f^{\epsilon}||_{L^{1}_{4}}
(||f^{\epsilon}||_{H^{s}_{l+3}}||g^{\epsilon}||_{H^{s}_{\gamma/2}}+
\epsilon^{2-2s}||f^{\epsilon}||_{H^{1}_{l+3}}||g^{\epsilon}||_{H^{1}_{\gamma/2}}),
\eeno
which implies, for any $\eta_{3}>0$,
\begin{eqnarray}\label{hmls6}
|\langle \mathfrak{I}_{2}, g^{\epsilon} \rangle| - \eta_{3} ||g^{\epsilon}||^{2}_{\epsilon,\gamma/2} \lesssim \frac{1}{\eta_{3}}||f^{\epsilon}||^{2}_{H^{m}_{6}}||f^{\epsilon}||^{2}_{\epsilon,l+3}.
\end{eqnarray}
While for the case $1 \leq |\alpha_{1}|\leq |\alpha| - 1 \leq  m-1$, we similarly have, for any $\eta_{3}>0$,
\begin{eqnarray}\label{hmls7}
|\langle \mathfrak{I}_{2}, g^{\epsilon} \rangle| - \eta_{3} ||g^{\epsilon}||^{2}_{\epsilon,\gamma/2} \lesssim   \frac{1}{\eta_{3}}(||f^{\epsilon}||^{2}_{H^{1}_{6}}||f^{\epsilon}||^{2}_{H^{m}_{l+3}}\textbf{1}_{m \geq 2}+
||f^{\epsilon}||^{2}_{H^{m-1}_{6}}||f^{\epsilon}||^{2}_{H^{m-1}_{l+3}}).
\end{eqnarray}
Now choosing suitable $\eta_{1}$ in (\ref{hmls2}), $\eta_{2}$ in (\ref{hmls4}) and (\ref{hmls5}), and $\eta_{3}$ in (\ref{hmls6}) and (\ref{hmls7}), we have
\begin{eqnarray}\label{hmlsclose1}
&&\frac{d}{dt}||f^{\epsilon}||^{2}_{H^{m}_{l}} + \frac{C_{1}(f_{0})}{2}||f^{\epsilon}||^{2}_{\epsilon, m, l+\gamma/2} \\&\lesssim&
C(||f_{0}||_{L^{1}_{2l+5}}, ||f_{0}||_{L^{2}}) ||f^{\epsilon}||^{2}_{H^{m}_{l+\gamma/2}}  \nonumber\\
&&+ C(C_{1}(f_{0}))\{||f^{\epsilon}||^{2}_{H^{m}_{2l+7}}||f^{\epsilon}||^{2}_{\epsilon,l+3} + ||f^{\epsilon}||^{2}_{H^{1}_{2l+7}}||f^{\epsilon}||^{2}_{H^{m}_{l+3}}\textbf{1}_{m \geq 2}+
||f^{\epsilon}||^{2}_{H^{m-1}_{2l+7}}||f^{\epsilon}||^{2}_{H^{m-1}_{l+3}} \}. \nonumber
\end{eqnarray}

When $m=1$, inequality (\ref{hmlsclose1}) reduces to
\beno
&&\frac{d}{dt}||f^{\epsilon}||^{2}_{H^{1}_{l}} + \frac{C_{1}(f_{0})}{2}||f^{\epsilon}||^{2}_{\epsilon, 1, l+\gamma/2} \\&\lesssim&
C(||f_{0}||_{L^{1}_{2l+5}}, ||f_{0}||_{L^{2}}) ||f^{\epsilon}||^{2}_{H^{1}_{l+\gamma/2}} \\
&& + C(C_{1}(f_{0}))\{||f^{\epsilon}||^{2}_{H^{1}_{2l+7}}||f^{\epsilon}||^{2}_{\epsilon,l+3} +
||f^{\epsilon}||^{2}_{L^{2}_{2l+7}}||f^{\epsilon}||^{2}_{L^{2}_{l+3}} \}.
\eeno
Remembering that
\beno
||f^{\epsilon}||^{2}_{\epsilon,l+3} = ||f^{\epsilon}||^{2}_{H^{s}_{l+3}} + \epsilon^{2-2s}||f^{\epsilon}||^{2}_{H^{1}_{l+3}},
\eeno
by interpolation theory and the basic inequality (\ref{basicineq}), for any $\eta>0$, we have
\beno
||f^{\epsilon}||^{2}_{H^{1}_{2l+7}} &\leq& ||f^{\epsilon}||^{2(1-s)}_{H^{1+s}_{l+\gamma/2}}||f^{\epsilon}||^{2s}_{H^{s}_{\psi(l)}}\\
&\leq& \eta ||f^{\epsilon}||^{2}_{H^{1+s}_{l+\gamma/2}} + s(\frac{\eta}{1-s})^{-\frac{1-s}{s}}||f^{\epsilon}||^{2}_{H^{s}_{x(l)}},
\eeno
where $x(l) = \frac{2l+7}{s} - \frac{1-s}{s}(l+\frac{\gamma}{2})$,
and
\beno
||f^{\epsilon}||^{2}_{H^{1}_{2l+7}}||f^{\epsilon}||^{2}_{H^{1}_{l+3}} &\leq& ||f^{\epsilon}||^{4}_{H^{1}_{2l+7}} \\
&\leq& ||f^{\epsilon}||^{\frac{4-4s}{2-s}}_{H^{2}_{l+\gamma/2}}||f^{\epsilon}||^{\frac{4}{2-s}}_{H^{s}_{\psi^{\prime}(l)}} \\
&\leq&\eta||f^{\epsilon}||^{2}_{H^{2}_{l+\gamma/2}} + \frac{s}{2-s}(\frac{2
\eta-s\eta}{2-2s})^{-\frac{2-2s}{s}}||f^{\epsilon}||^{4/s}_{H^{s}_{\tilde{x}(l)}},
\eeno
where $\tilde{x}(l) = (2l+7)(2-s) - (1-s)(l+\gamma/2) \leq x(l)$.
Taking small enough $\eta$, we finally have
\beno
\frac{d}{dt}||f^{\epsilon}||^{2}_{H^{1}_{l}} + \frac{C_{1}(f_{0})}{4}||f^{\epsilon}||^{2}_{\epsilon, 1, l+\gamma/2} &\lesssim&
C(||f^{\epsilon}||_{H^{s}_{x(l)}}, ||f_{0}||_{L^{1}_{2l+5}}).
\eeno
Then by Gronwall's inequality and the uniform upper bound (\ref{hslclose8}) of $H^{s}$ norm, we arrive at
\beno
&&||f^{\epsilon}(t)||^{2}_{H^{1}_{l}} + \frac{C_{1}(f_{0})}{4}\int^{t+1}_{t}||f^{\epsilon}(r)||^{2}_{\epsilon, 1, l+\gamma/2}dr
\lesssim C(||f_{0}||_{L^{1}_{\phi(s,x(l))}},||f_{0}||_{H^{s}_{x(l)}},||f_{0}||_{H^{1}_{l}}).
\eeno
Once again by interpolation theory, there holds
\beno
||f_{0}||_{H^{s}_{x(l)}} &\lesssim& ||f_{0}||_{H^{1}_{l}} + ||f_{0}||_{L^{1}_{y(l)}},
\eeno
where $y(l) = \frac{3x(l) - (s+2)l}{1-s}$. By setting $\phi(1,l) = \max\{\phi(s,x(l)),y(l)\}$, we have
\beno
||f^{\epsilon}(t)||^{2}_{H^{1}_{l}} + \frac{C_{1}(f_{0})}{4}\int^{t+1}_{t}||f^{\epsilon}(r)||^{2}_{\epsilon, 1, l+\gamma/2}dr
\lesssim C(||f_{0}||_{L^{1}_{\phi(1,l)}},||f_{0}||_{H^{1}_{l}}).
\eeno

When $m \geq 2$, $||f^{\epsilon}||^{2}_{H^{1}_{2l+7}}$ has uniform bound by assumption. According to the interpolation inequality and the basic inequality (\ref{basicineq}), one has
\begin{eqnarray}\label{hmhmsmminus1}
||f^{\epsilon}||^{2}_{H^{m}_{2l+7}} \leq \eta ||f^{\epsilon}||^{2}_{H^{m+s}_{l+\gamma/2}} + (\frac{1+s}{s}\eta)^{-\frac{1}{s}} ||f^{\epsilon}||^{2}_{H_{z(l)}^{m-1}},
\end{eqnarray}
where $z(l) = 2l+7+\frac{l+7}{s}$.
With the fact $||f^{\epsilon}||_{\epsilon,l+3}\lesssim ||f^{\epsilon}||_{H^{1}_{l+3}}$, we finally arrive at
\beno
&&\frac{d}{dt}||f^{\epsilon}||^{2}_{H^{m}_{l}} + \frac{C_{1}(f_{0})}{4}||f^{\epsilon}||^{2}_{\epsilon, m, l+\gamma/2} \lesssim
C(||f_{0}||_{L^{1}_{2l+5}},||f^{\epsilon}||_{H_{z(l)}^{m-1}}).
\eeno
Then by Gronwall's inequality and the assumed uniform bound of  $H^{m-1}$ norm,
\beno
||f^{\epsilon}(t)||^{2}_{H^{m}_{l}} + \frac{C_{1}(f_{0})}{4}\int^{t+1}_{t}||f^{\epsilon}(r)||^{2}_{\epsilon, m, l+\gamma/2}dr
\lesssim C(||f_{0}||_{L^{1}_{\phi(m-1,z(l))}},||f_{0}||_{H^{m-1}_{z(l)}},||f_{0}||_{H^{m}_{l}}).
\eeno
By interpolation theory, there holds
\beno
||f_{0}||_{H^{m-1}_{z(l)}} &\lesssim& ||f_{0}||_{H^{m}_{l}} + ||f_{0}||_{L^{1}_{u(m,l)}},
\eeno
where $u(m,l) = (m+2)z(l)-(m+1)l$. Now by setting $\phi(m,l) = \max\{u(m,l),\phi(m-1,z(l))\}$, we arrive at
\begin{eqnarray}\label{hmlsclosef2}
||f^{\epsilon}(t)||^{2}_{H^{m}_{l}} + \frac{C_{1}(f_{0})}{4}\int^{t+1}_{t}||f^{\epsilon}(r)||^{2}_{\epsilon, m, l+\gamma/2}dr
\lesssim C(||f_{0}||_{L^{1}_{\phi(m,l)}},||f_{0}||_{H^{m}_{l}}).
\end{eqnarray}
The proof of theorem \ref{main2} is complete now.

\begin{rmk}
Since $L^{1} \subset H^{-m}$ if $m>3/2$, one can obtain lower weight requirement in the space $L^{1}$.
We use $H^{-2}$ as one interpolation space just for a neat expression. For the same reason, we replace $\gamma$ with $2$.
\end{rmk}

\section{Error estimates to the approximation}
In this section, we prove the last two theorems stated in section 1. We first give a proof to theorem \ref{main3}.

\noindent {\bf Proof of Theorem \ref{main3}:} For each $0 < \epsilon \leq \sqrt{2}/2$, we define $F^{\epsilon}_{R}$ and $Q_{\epsilon}$ respectively as follows:
\begin{eqnarray}\label{errorf}
F^{\epsilon}_{R} = \frac{f^{\epsilon}-f}{\epsilon^{3-2s}},
\end{eqnarray}
\begin{eqnarray}\label{qlow}
Q_{\epsilon} = Q - Q^{\epsilon}.
\end{eqnarray}
Take the difference between equations (\ref{homb}) and (\ref{approximated}), and divide both sides by $\epsilon^{3-2s}$, we have
\ben\label{errorequation}
\partial _t F^{\epsilon}_{R} = \Upsilon(f^{\epsilon}) + Q(f^{\epsilon},F^{\epsilon}_{R}) + Q(F^{\epsilon}_{R},f)
\een
where
\begin{eqnarray}\label{upsilonf}
\Upsilon(f^{\epsilon}) = \frac{1}{\epsilon}[Q_{L}(f^{\epsilon},f^{\epsilon})-\epsilon^{2s-2}Q_{\epsilon}(f^{\epsilon},f^{\epsilon})].
\end{eqnarray}
We now show that $L^{1}_{2l}$ norm of $F^{\epsilon}_{R}$   is bounded by the initial datum $f_{0}$ and time $t$. According to (\ref{errorequation}), we have
\ben \label{l1norm}
\frac{d}{dt}||F^{\epsilon}_{R}||_{L^{1}_{2l}} &=& \langle  \Upsilon(f^{\epsilon}) + Q(f^{\epsilon},F^{\epsilon}_{R}) + Q(F^{\epsilon}_{R},f), sgn(F^{\epsilon}_{R}) \langle v \rangle^{2l}  \rangle \\
&\eqdefa&\mathfrak{I}_1+\mathfrak{I}_2+\mathfrak{I}_3. \nonumber
\een
Thanks to lemma (7.1) in the Appendix of \cite{he1}, we have
\begin{eqnarray}\label{l1I1f}
\mathfrak{I}_1  \leq C(7.1)||f^{\epsilon}||^{2}_{H^{5}_{2l + \gamma+12}}.
\end{eqnarray}

Now we deal with $\mathfrak{I}_{2}$, note that
\beno
\mathfrak{I}_{2}&=&\int_{\R^6\times\SS^2}b(\cos\theta)|v-v_{*}|^{\gamma}f^{\epsilon}_{*}F^{\epsilon}_{R}
(sgn(F^{\epsilon}_{R}(v^{\prime}))\langle v^{\prime} \rangle^{2l} - sgn(F^{\epsilon}_{R}(v))\langle v \rangle^{2l})d\sigma dv_{*}dv \\
&\leq& \int_{\R^6\times\SS^2}b(\cos\theta)|v-v_{*}|^{\gamma}f^{\epsilon}_{*}|F^{\epsilon}_{R}|
(\langle v^{\prime} \rangle^{2l} - \langle v \rangle^{2l})d\sigma dv_{*}dv\\
&\eqdefa&\mathfrak{I}_{2,1}+\mathfrak{I}_{2,2},
\eeno
where
\beno
\mathfrak{I}_{2,1}=\int_{\R^6\times\SS^2}b(\cos\theta)|v-v_{*}|^{\gamma}f^{\epsilon}_{*}|F^{\epsilon}_{R}|
(\langle v^{\prime} \rangle^{2l} + \langle v^{\prime}_{*} \rangle^{2l}- \langle v \rangle^{2l}-\langle v_{*} \rangle^{2l})d\sigma dv_{*}dv,
\eeno
and
\beno
\mathfrak{I}_{2,2}=-\int_{\R^6\times\SS^2}b(\cos\theta)|v-v_{*}|^{\gamma}f^{\epsilon}_{*}|F^{\epsilon}_{R}|
(\langle v^{\prime}_{*} \rangle^{2l}- \langle v_{*} \rangle^{2l})d\sigma dv_{*}dv.
\eeno
According to proposition \ref{propsition1}, we have
\begin{eqnarray}\label{l1I21f}
\quad\quad\quad \mathfrak{I}_{2,1} \leq -\frac{1}{4}A_{2}(||f^{\epsilon}||_{L^{1}}||F^{\epsilon}_{R}||_{L^{1}_{2l+\gamma}}+
||f^{\epsilon}||_{L^{1}_{2l+\gamma}}||F^{\epsilon}_{R}||_{L^{1}})+
2^{2l+1}A_{2}||f^{\epsilon}||_{L^{1}_{2l}}||F^{\epsilon}_{R}||_{L^{1}_{2l}}.
\end{eqnarray}
Now we turn to $\mathfrak{I}_{2,2}$. Recall that
\beno
\langle v^{\prime}_{*} \rangle^{2l}- \langle v_{*} \rangle^{2l} &=& (v^{\prime}_{*}- v_{*})\cdot(\nabla\langle \cdot \rangle^{2l})(v_{*}) \\
&&+\int^{1}_{0}\frac{1-\kappa}{2}(v^{\prime}_{*}- v_{*})\otimes(v^{\prime}_{*}- v_{*}):(\nabla^{2}\langle \cdot \rangle^{2l})(v(\kappa))d\kappa,
\eeno
where $v(\kappa) = v_{*}+\kappa(v^{\prime}_{*}- v_{*}) = v_{*}-\kappa(v^{\prime}- v)$.
By symmetry,
\begin{eqnarray}\label{symmtry}
\int_{\SS^2}b(\cos\theta)(v^{\prime}_{*}- v_{*})d\sigma = (v-v_{*})\int_{\SS^{2}}b(\cos\theta)\sin^{2}\frac{\theta}{2}d\sigma.
\end{eqnarray}
Observe that the matrix $\nabla^{2}\langle \cdot \rangle^{2l}$ is positive definite, we are only left with
\begin{eqnarray}\label{l1I22f}
\mathfrak{I}_{2,2}&\leq& 2l\int_{\R^6\times\SS^2}b(\cos\theta)\sin^{2}\frac{\theta}{2}|v-v_{*}|^{\gamma+1}\langle v_{*} \rangle^{2l-1}
f^{\epsilon}_{*}|F^{\epsilon}_{R}|d\sigma dv_{*}dv \\
&\leq& A_{2}l||f^{\epsilon}||_{L^{1}_{2l+\gamma}}||F^{\epsilon}_{R}||_{L^{1}_{\gamma+1}}. \nonumber
\end{eqnarray}

Split $\mathfrak{I}_{3}$ into two parts:
\beno
\mathfrak{I}_{3}&=&\int_{\R^6\times\SS^2}b(\cos\theta)|v-v_{*}|^{\gamma}F^{\epsilon}_{R}(v_{*})f
(sgn(F^{\epsilon}_{R}(v^{\prime}))\langle v^{\prime} \rangle^{2l} - sgn(F^{\epsilon}_{R}(v))\langle v \rangle^{2l})d\sigma dv_{*}dv \\
&=& \int_{\R^6\times\SS^2}B\textbf{1}_{\theta\leq|v-v_{*}|^{-\alpha}}F^{\epsilon}_{R}(v_{*})f
(sgn(F^{\epsilon}_{R}(v^{\prime}))\langle v^{\prime} \rangle^{2l} - sgn(F^{\epsilon}_{R}(v))\langle v \rangle^{2l})d\sigma dv_{*}dv\\
&&+\int_{\R^6\times\SS^2}B\textbf{1}_{\theta\geq|v-v_{*}|^{-\alpha}}F^{\epsilon}_{R}(v_{*})f
(sgn(F^{\epsilon}_{R}(v^{\prime}))\langle v^{\prime} \rangle^{2l} - sgn(F^{\epsilon}_{R}(v))\langle v \rangle^{2l})d\sigma dv_{*}dv\\
&\eqdefa&\mathfrak{I}_{3,1}+\mathfrak{I}_{3,2},
\eeno
where $\alpha = \frac{\gamma+2}{2-2s}$. \\
For $\mathfrak{I}_{3,1}$, we have
\beno
\mathfrak{I}_{3,1}&=&\int_{\R^6\times\SS^2}B\textbf{1}_{\theta\leq|v-v_{*}|^{-\alpha}}F^{\epsilon}_{R}(v_{*})
(sgn(F^{\epsilon}_{R}(v^{\prime}))f^{\prime}\langle v^{\prime} \rangle^{2l} - sgn(F^{\epsilon}_{R}(v))f\langle v \rangle^{2l})d\sigma dv_{*}dv\\
&&+\int_{\R^6\times\SS^2}b(\cos\theta)\textbf{1}_{\theta\leq|v-v_{*}|^{-\alpha}}|v-v_{*}|^{\gamma}F^{\epsilon}_{R}(v_{*})(f-f^{\prime})
sgn(F^{\epsilon}_{R}(v^{\prime}))\langle v^{\prime} \rangle^{2l} d\sigma dv_{*}dv\\
&\eqdefa&\mathfrak{I}_{3,1,1}+\mathfrak{I}_{3,1,2}.
\eeno
By cancellation lemma,
\begin{eqnarray}\label{l1I311f}
|\mathfrak{I}_{3,1,1}|\leq C(cancel)||f||_{L^{1}_{2l+\gamma}}||F^{\epsilon}_{R}||_{L^{1}_{\gamma}},
\end{eqnarray}
where $C(cancel) = 2^{\frac{5+\gamma}{2}}A_{2}$.
For the term $\mathfrak{I}_{3,1,2}$, apply Taylor expansion:
\beno
f(v)-f(v^{\prime}) &=& (v-v^{\prime})\cdot \nabla_{v} f(v^{\prime})
+\int^{1}_{0}\frac{1-\kappa}{2}(v-v^{\prime})\otimes(v-v^{\prime}):\nabla^{2}_{v}f(v(\kappa))d\kappa,
\eeno
where $v(\kappa) = v^{\prime}+\kappa(v-v^{\prime})$.
For fixed $v_{*}$, it is easy to check
\beno
\int_{\R^3\times\SS^2}b(\cos\theta)\textbf{1}_{\theta\leq|v-v_{*}|^{-\alpha}}|v-v_{*}|^{\gamma}(v-v^{\prime})\cdot \nabla_{v} f(v^{\prime})
sgn(F^{\epsilon}_{R}(v^{\prime}))\langle v^{\prime} \rangle^{2l} d\sigma dv=0.
\eeno
Thus we are only left with
\begin{eqnarray}\label{l1I312}
|\mathfrak{I}_{3,1,2}|&\leq&\int^{1}_{0}\int_{\R^6\times\SS^2}\frac{1-\kappa}{2}b(\cos\theta)\sin^{2}\frac{\theta}{2}\textbf{1}_{\theta\leq|v-v_{*}|^{-\alpha}}
|v-v_{*}|^{\gamma+2}  \\
&&\times F^{\epsilon}_{R}(v_{*})|\nabla^{2}_{v}f(v(\kappa))|\langle v^{\prime} \rangle^{2l} d\kappa d\sigma dv_{*}dv. \nonumber
\end{eqnarray}
Set $u = v^{\prime}+\kappa(v-v^{\prime})$, then we have
\beno
\langle v^{\prime} \rangle^{2} &=&1+|v^{\prime}|^{2}=1+|v^{\prime}+\kappa(v-v^{\prime})-\kappa(v-v^{\prime})|^{2} \\
&\leq & 1+2|u|^{2}+2\kappa^{2}|v-v^{\prime}|^{2} \leq 2\langle u \rangle^{2} + 2\kappa^{2}|u-v_{*}|^{2},
\eeno
and
\beno
\langle v^{\prime} \rangle^{2l} \leq 2^{2l-1} \langle u \rangle^{2l} + 2^{2l-1}\kappa^{2l}|u-v_{*}|^{2l}\leq 2^{2l} \langle u \rangle^{2l} \langle v_{*} \rangle^{2l}.
\eeno
By the change of variable: $v \rightarrow u$, the Jacobian matrix is
\beno
\frac{du}{dv} = \frac{1+k}{2}(I+\frac{1-k}{1+k}\frac{v-v_{*}}{|v-v_{*}|}\otimes\sigma),
\eeno
with its Jacobian
\beno
|\frac{du}{dv}| = \frac{(1+k)^{3}}{8}(1+\frac{1-k}{1+k}\frac{v-v_{*}}{|v-v_{*}|}\cdot\sigma)\geq \frac{1}{8}.
\eeno
Thanks to $|u-v_{*}| \leq |v-v_{*}| \leq \sqrt{2}|u-v_{*}|$, we obtain
\begin{eqnarray}\label{l1I312f}
|\mathfrak{I}_{3,1,2}|&\leq& 2^{2l+3}\pi K \int_{\R^6}\int^{|u-v_{*}|^{-\alpha}\wedge\pi/2}_{0}\theta^{1-2s}|u-v_{*}|^{\gamma+2}\\
&&\times F^{\epsilon}_{R}(v_{*})|\nabla^{2}_{v}f(u)|\langle u \rangle^{2l}\langle v_{*} \rangle^{2l} d\theta dv_{*}du. \nonumber \\
&\leq& 2^{2l+2} \frac{\pi K}{1-s} ||\nabla^{2}_{v}f||_{L^{1}_{2l}}||F^{\epsilon}_{R}||_{L^{1}_{2l}}\nonumber \\
&\leq&2^{2l+\frac{5}{2}} \frac{\pi^{2} K}{1-s} ||f||_{H^{2}_{2l+2}}||F^{\epsilon}_{R}||_{L^{1}_{2l}}, \nonumber
\end{eqnarray}
where we have used the fact $||\langle \cdot \rangle^{-2}||_{L^{2}} \leq \sqrt{2} \pi$.

Now we turn to $\mathfrak{I}_{3,2}$. Note that
\beno
\mathfrak{I}_{3,2} &\leq& \int_{\R^6\times\SS^2}b(\cos\theta)\textbf{1}_{\theta\geq|v-v_{*}|^{-\alpha}}|v-v_{*}|^{\gamma}|F^{\epsilon}_{R}(v_{*})|f
(\langle v^{\prime} \rangle^{2l} + \langle v \rangle^{2l})d\sigma dv_{*}dv \\
&= & \int_{\R^6\times\SS^2}b(\cos\theta)\textbf{1}_{\theta\geq|v-v_{*}|^{-\alpha}}|v-v_{*}|^{\gamma}|F^{\epsilon}_{R}(v_{*})|f
(\langle v^{\prime} \rangle^{2l} - \langle v \rangle^{2l})d\sigma dv_{*}dv \\
&&+2\int_{\R^6\times\SS^2}b(\cos\theta)\textbf{1}_{\theta\geq|v-v_{*}|^{-\alpha}}|v-v_{*}|^{\gamma}|F^{\epsilon}_{R}(v_{*})|f
 \langle v \rangle^{2l}d\sigma dv_{*}dv\\
 &\eqdefa& \mathfrak{I}_{3,2,1}+\mathfrak{I}_{3,2,2}.
\eeno
First look at the term $\mathfrak{I}_{3,2,1}$. Recall that $j = \frac{u-(u \cdot n)n}{|u-(u \cdot n)n|}$ in lemma \ref{lemma1},
then we have $j\cdot n = 0$, and thus
\beno
\int_{\SS^{2}}b(\cos\theta)\textbf{1}_{\theta\geq|v-v_{*}|^{-\alpha}}(E(\theta))^{p-1}h(j\cdot\omega)\sin\theta d\sigma = 0.
\eeno
Applying proposition \ref{lemma1} and the above equality, we obtain
\beno
\mathfrak{I}_{3,2,1} &\leq& \int_{\R^6\times\SS^2}b(\cos\theta)\sin^{2l}\frac{\theta}{2}|v-v_{*}|^{\gamma}|F^{\epsilon}_{R}(v_{*})|f
\langle v_{*} \rangle^{2l}d\sigma dv_{*}dv\\
&&+ c_{l}\int_{\R^6\times\SS^2}b(\cos\theta)\sin^{2}\theta|v-v_{*}|^{\gamma}|F^{\epsilon}_{R}(v_{*})|f
\langle v_{*} \rangle^{2l-2}\langle v \rangle^{2l-2}d\sigma dv_{*}dv\\
&\eqdefa& \mathfrak{I}_{3,2,1,1}+\mathfrak{I}_{3,2,1,2},
\eeno
where $c_{l} = 2^{l-3}(l(l-1)+4)$.
Thanks to the following fact:
\beno
 \int_{\SS^2}b(\cos\theta)\sin^{2l}\frac{\theta}{2}d\sigma &\leq& 2\pi K \int^{\pi/2}_{0}\theta^{-1-2s}\sin^{2l}\frac{\theta}{2}d\theta\\
&=&2^{1-2s}\pi K \int^{\pi/4}_{0}\eta^{-1-2s}\sin^{2l}\eta d\eta \\
&\leq& \frac{2^{-2s} \pi K}{l-s}(\frac{\pi}{4})^{2l-2s},
\eeno
and $|v-v_{*}|^{\gamma} \leq 2(\langle v \rangle^{\gamma}+\langle v_{*} \rangle^{\gamma})$, we have
\begin{eqnarray}\label{l1I3211f}
\mathfrak{I}_{3,2,1,1}\leq \frac{2^{1-2s} \pi K}{l-s}(\frac{\pi}{4})^{2l-2s}(||f||_{L^{1}}||F^{\epsilon}_{R}||_{L^{1}_{2l+\gamma}}
+||f||_{L^{1}_{\gamma}}||F^{\epsilon}_{R}||_{L^{1}_{2l}}).
\end{eqnarray}
Due to $|v-v_{*}|^{\gamma} \leq \langle v \rangle^{2} \langle v_{*} \rangle^{2}$, we obtain
\begin{eqnarray}\label{l1I3212f}
\mathfrak{I}_{3,2,1,2}\leq c_{l}A_{2}||f||_{L^{1}_{2l}}||F^{\epsilon}_{R}||_{L^{1}_{2l}}.
\end{eqnarray}

For the term $\mathfrak{I}_{3,2,2}$, we have
\begin{eqnarray}\label{l1I322f}
|\mathfrak{I}_{3,2,2}|&\leq& 4\pi K \int_{\R^6}\int^{\pi/2}_{|v-v_{*}|^{-\alpha}\wedge\pi/2} \theta^{-1-2s}
|v-v_{*}|^{\gamma}|F^{\epsilon}_{R}(v_{*})|f \langle v \rangle^{2l}d\theta dv_{*}dv \\
&\leq&\frac{2 \pi K}{s} ||f||_{L^{1}_{4l}}||F^{\epsilon}_{R}||_{L^{1}_{2l}}, \nonumber
\end{eqnarray}
provided $2 \alpha s + \gamma \leq 2l$.

Patch the above inequalities (\ref{l1I1f}),(\ref{l1I21f}),(\ref{l1I22f}),(\ref{l1I311f}), and (\ref{l1I312f})-(\ref{l1I322f}), for those $l$ such that $\frac{2^{1-2s} \pi K}{l-s}(\frac{\pi}{4})^{2l-2s} \leq \frac{A_{2}}{8}$, we have the following desired result:
\beno
\frac{d}{dt}||F^{\epsilon}_{R}||_{L^{1}_{2l}} &\leq& -\frac{A_{2}}{8} ||f_{0}||_{L^{1}}||F^{\epsilon}_{R}||_{L^{1}_{2l+\gamma}}
+C(7.1)||f^{\epsilon}||^{2}_{H^{5}_{2l + \gamma+12}}
\\&&+\big\{2^{2l+1}A_{2}||f^{\epsilon}||_{L^{1}_{2l}}+A_{2}l||f^{\epsilon}||_{L^{1}_{2l+\gamma}}
+C(cancel)||f||_{L^{1}_{2l+\gamma}} \\
&&+2^{2l+\frac{1}{2}} \frac{\pi^{2} K}{1-s} ||f||_{H^{2}_{2l+2}} +\frac{A_{2}}{8}||f||_{L^{1}_{\gamma}}\\
&&+c_{l}A_{2}||f||_{L^{1}_{2l}}+ \frac{2 \pi K}{s} ||f||_{L^{1}_{4l}}\}||F^{\epsilon}_{R}||_{L^{1}_{2l}},
\eeno
where we have used the mass conservation property: $||f_{t}||_{L^{1}}= ||f^{\epsilon}_{t}||_{L^{1}}=||f_{0}||_{L^{1}}$.
The propagation of $L$ and $H$ norms of $f$ and $f^{\epsilon}$ allows us to conclude:
\begin{eqnarray}\label{l1close2}
\frac{d}{dt}||F^{\epsilon}_{R}||_{L^{1}_{2l}} &\leq& -\frac{A_{2}}{8} ||f_{0}||_{L^{1}}||F^{\epsilon}_{R}||_{L^{1}_{2l+\gamma}}
+C(||f_{0}||_{L^{1}_{\phi(5,2l+\gamma+12)}},||f_{0}||_{H^{5}_{2l + \gamma+12}})\\
&&+C(||f_{0}||_{L^{1}_{\max\{41,\phi(2,2l+2)\}}},||f_{0}||_{H^{2}_{2l+2}})||F^{\epsilon}_{R}||_{L^{1}_{2l}}. \nonumber
\end{eqnarray}
Applying Gronwall's inequality (\ref{gronwall}) with $a = C(||f_{0}||_{L^{1}_{\phi(5,2l+\gamma+12)}},||f_{0}||_{H^{5}_{2l + \gamma+12}})$ \\and $b = C(||f_{0}||_{L^{1}_{\max\{41,\phi(2,2l+2)\}}},||f_{0}||_{H^{2}_{2l+2}})$, we have
\begin{eqnarray}\label{l1closef}
||F^{\epsilon}_{R}(t)||_{L^{1}_{2l}} &\leq& \frac{a}{b}(e^{bt}-1)\eqdefa C(f_{0},t). \nonumber
\end{eqnarray}

\vskip 0.5cm
We now prove theorem \ref{main4} in the rest of this section.

\noindent {\bf Proof of Theorem \ref{main4}:}\\
\noindent {\textit {Step 1}:} (Case $N=0$)\\
Taking the difference between equations (\ref{homb}) and (\ref{approximated}), and dividing both sides by $\epsilon^{3-2s}$, we have
\ben\label{errorequation2}
\partial _t F^{\epsilon}_{R} = \Upsilon(f) + M^{\epsilon}(f^{\epsilon},F^{\epsilon}_{R}) + M^{\epsilon}(F^{\epsilon}_{R},f) \nonumber
\een
Then we have
\ben \label{l2normweightl}
\frac{d}{dt}(\frac{1}{2}||F^{\epsilon}_{R}||^{2}_{L^{2}_{l}}) &=& \langle  \Upsilon(f) + M^{\epsilon}(f^{\epsilon},F^{\epsilon}_{R}) + M^{\epsilon}(F^{\epsilon}_{R},f), F^{\epsilon}_{R}\langle v \rangle^{2l} \rangle \nonumber \\
&\eqdefa&\mathfrak{I}_1+\mathfrak{I}_2+\mathfrak{I}_3 \nonumber
\een
Thanks to lemma 7.1 in the Appendix of \cite{he1}, we have
\beno
\mathfrak{I}_1  \lesssim ||f||^{2}_{H^{5}_{l+\gamma+10}}||F^{\epsilon}_{R}||_{L^{2}_{l}},
\eeno
which implies, for any $\eta > 0$,
\begin{eqnarray}\label{l2lI1}
\mathfrak{I}_1  - \eta ||F^{\epsilon}_{R}||^{2}_{L^{2}} \lesssim \frac{1}{\eta}||f||^{4}_{H^{5}_{l+\gamma+10}}.
\end{eqnarray}
Splitting $\mathfrak{I}_2$ into two terms
\begin{eqnarray}\label{l2lI2to2}
\mathfrak{I}_2  &=& \langle M^{\epsilon}(f^{\epsilon},F^{\epsilon}_{R}\langle v \rangle^{l}), F^{\epsilon}_{R}\langle v \rangle^{l} \rangle + \{\langle M^{\epsilon}(f^{\epsilon},F^{\epsilon}_{R})\langle v \rangle^{l}-M^{\epsilon}(f^{\epsilon},F^{\epsilon}_{R}\langle v \rangle^{l}), F^{\epsilon}_{R}\langle v \rangle^{l} \rangle\} \nonumber \\
&\eqdefa& \mathfrak{I}_{2,1}+\mathfrak{I}_{2,2}. \nonumber
\end{eqnarray}
By coercivity estimate (\ref{coercivityboltz}), we have
\begin{eqnarray}\label{l2lI21}
\mathfrak{I}_{2,1} \leq -C_{1}(f_{0})||F^{\epsilon}_{R}||^{2}_{\epsilon,l+\gamma/2}+ C_{2}(f_{0})||F^{\epsilon}_{R}||^{2}_{L^{2}_{l+\gamma/2}}.
\end{eqnarray}
By commutator estimate (\ref{commcboltz}) with $N_{2} = l + \gamma/2, N_{3} = \gamma/2$, we have,
\beno
\mathfrak{I}_{2,2} \lesssim  ||f^{\epsilon}||_{L^{1}_{2l+5}} ||F^{\epsilon}_{R}||_{\epsilon,l+\gamma/2} ||F^{\epsilon}_{R}||_{L^{2}_{l+\gamma/2}},
\eeno
which implies,  for any $\eta > 0$,
\begin{eqnarray}\label{l2lI22}
\mathfrak{I}_{2,2} -\eta ||F^{\epsilon}_{R}||^{2}_{\epsilon,l+\gamma/2} \lesssim \frac{1}{\eta} ||f^{\epsilon}||^{2}_{L^{1}_{2l+5}} ||F^{\epsilon}_{R}||^{2}_{L^{2}_{l+\gamma/2}}.
\end{eqnarray}
Splitting $\mathfrak{I}_3$ into two terms
\begin{eqnarray}\label{l2lI3to2}
\mathfrak{I}_3  &=& \langle M^{\epsilon}(F^{\epsilon}_{R},f \langle v \rangle^{l}), F^{\epsilon}_{R}\langle v \rangle^{l} \rangle_{v} + \{\langle M^{\epsilon}(F^{\epsilon}_{R},f)\langle v \rangle^{l}-M^{\epsilon}(F^{\epsilon}_{R},f \langle v \rangle^{l}), F^{\epsilon}_{R}\langle v \rangle^{l} \rangle_{v}\} \nonumber \\
&\eqdefa& \mathfrak{I}_{3,1}+\mathfrak{I}_{3,2}. \nonumber
\end{eqnarray}
Applying upper bound estimate (\ref{upcboltz}) with $ w_{1} = \gamma/2 + 2, w_{2} = \gamma/2$, we have
\beno
\mathfrak{I}_{3,1} \lesssim ||F^{\epsilon}_{R}||_{L^{1}_{\gamma+2}} ||f||_{H^{1}_{l+3}}||F^{\epsilon}_{R}||_{\epsilon, l+\gamma/2},
\eeno
which implies,  for any $\eta > 0$,
\begin{eqnarray}\label{l2lI31}
\mathfrak{I}_{3,1}-\eta ||F^{\epsilon}_{R}||^{2}_{\epsilon, l+\gamma/2} \lesssim \frac{1}{\eta} ||F^{\epsilon}_{R}||^{2}_{L^{1}_{\gamma+2}} ||f||^{2}_{H^{1}_{l+3}}.
\end{eqnarray}
By commutator estimate (\ref{commcboltz}), we have
\beno
\mathfrak{I}_{3,2}
\lesssim  ||F^{\epsilon}_{R}||_{L^{1}_{2l+5}} ||f||_{H^{1}_{l+\gamma/2}} ||F^{\epsilon}_{R}||_{L^{2}_{l+\gamma/2}} \nonumber \\
\eeno
which implies, for any $\eta>0$,
\begin{eqnarray}\label{l2lI32}
\mathfrak{I}_{3,2} - \eta ||F^{\epsilon}_{R}||^{2}_{L^{2}_{l+\gamma/2}} \lesssim \frac{1}{\eta}||F^{\epsilon}_{R}||^{2}_{L^{1}_{2l+5}} ||f||^{2}_{H^{1}_{l+\gamma/2}}.
\end{eqnarray}
Now setting $\eta = \frac{C_{1}(f_{0})}{8}$ in (\ref{l2lI1}),(\ref{l2lI22}),(\ref{l2lI31}),(\ref{l2lI32}), and combining with (\ref{l2lI21}), we have
\begin{eqnarray}\label{l2lclose1}
&&\frac{d}{dt}||F^{\epsilon}_{R}||^{2}_{L^{2}_{l}}  + C_{1}(f_{0})||F^{\epsilon}_{R}||^{2}_{\epsilon,l+\gamma/2} \\&\lesssim &
\{C_{2}(f_{0})+\frac{1}{\eta} ||f^{\epsilon}||^{2}_{L^{1}_{2l+5}}\}||F^{\epsilon}_{R}||^{2}_{L^{2}_{l+\gamma/2}} +\frac{1}{\eta}||f||^{4}_{H^{5}_{l+\gamma+10}} \nonumber \\
&&+\frac{1}{\eta} ||F^{\epsilon}_{R}||^{2}_{L^{1}_{\gamma+2}}||f||^{2}_{H^{1}_{l+3}}
+\frac{1}{\eta}||F^{\epsilon}_{R}||^{2}_{L^{1}_{2l+5}} ||f||^{2}_{H^{1}_{l+\gamma/2}} \nonumber
\end{eqnarray}
Now choosing $\lambda = \frac{C_{1}(f_{0})}{2}(C_{2}(f_{0})+\frac{1}{\eta} ||f^{\epsilon}||^{2}_{L^{1}_{2l+5}})^{-1}$ in (\ref{l2hsl1}), we have
\begin{eqnarray}\label{l2lclose2}
&&\frac{d}{dt}||F^{\epsilon}_{R}||^{2}_{L^{2}_{l}}  + \frac{C_{1}(f_{0})}{2}||F^{\epsilon}_{R}||^{2}_{\epsilon,l+\gamma/2}   \\&\lesssim&
\{C_{2}(f_{0})+\frac{1}{\eta} ||f^{\epsilon}||^{2}_{L^{1}_{2l+5}}\} \lambda ^{-\frac{3}{2s}}||F^{\epsilon}_{R}||^{2}_{L^{1}_{l+\gamma/2}}
+\frac{1}{\eta}||f||^{4}_{H^{5}_{l+\gamma+10}}
\nonumber \\&&+\frac{1}{\eta} ||F^{\epsilon}_{R}||^{2}_{L^{1}_{\gamma+2}}||f||^{2}_{H^{1}_{l+3}}
+\frac{1}{\eta}||F^{\epsilon}_{R}||^{2}_{L^{1}_{2l+5}} ||f||^{2}_{H^{1}_{l+\gamma/2}} \nonumber
\end{eqnarray}
According to theorem \ref{main3}, we have
\beno
||F^{\epsilon}_{R}(t)||_{L^{1}_{2l+5}} \leq C(||f_{0}||_{L^{1}_{\phi(5,2l+\gamma+17)}},||f_{0}||_{H^{5}_{2l+\gamma+17}},t).
\eeno
The other terms of the right hand side of (\ref{l2lclose2}) are also bounded by some lower order or lower weight norm of initial datum $f_{0}$, thus we arrive at
\beno
||F^{\epsilon}_{R}(t)||_{L^{2}_{l}} \leq C(||f_{0}||_{L^{1}_{\phi(5,2l+\gamma+17)}},||f_{0}||_{H^{5}_{2l+\gamma+17}},t).
\eeno
We remark that the dependence on $t$ is also at most exponential.

\noindent {\textit {Step 2}:} (Case $N \geq 1$)\\
Suppose inequality (\ref{errorl2}) holds true for all $N \leq m-1$, we now prove that it is also valid for $N = m$.

Let $g^{\epsilon}_{\alpha, l} = \langle v \rangle^{l} \partial^{\alpha}_{v}F^{\epsilon}_{R}$ with $|\alpha| \leq m$, then $g^{\epsilon}_{\alpha, l}$ solves
\begin{eqnarray}\label{galpha}
\pa_t g^{\epsilon}_{\alpha, l} = \sum_{\alpha_1+\alpha_2=\alpha} \binom{\alpha}{\alpha_{1}}
[M^{\epsilon}(\pa^{\alpha_1}_v f^\epsilon, \pa^{\alpha_2}_v F^{\epsilon}_{R}) +
M^{\epsilon}(\pa^{\alpha_1}_v F^{\epsilon}_{R} , \pa^{\alpha_2}_v f) + \Upsilon(\pa^{\alpha_1}_v f, \pa^{\alpha_2}_v f)]\langle v \rangle^{l} \nonumber
\end{eqnarray}
Therefore we have
\begin{eqnarray}\label{falphalequation}
\frac{d}{dt}(\frac{1}{2}||g^{\epsilon}_{\alpha, l}||^{2}_{L^{2}}) &=& \langle \pa_t g^{\epsilon}_{\alpha, l}, g^{\epsilon}_{\alpha, l}\rangle \\
&=&  \sum_{\alpha_1+\alpha_2=\alpha} \binom{\alpha}{\alpha_{1}} \{ \langle M^{\epsilon}(\pa^{\alpha_1}_v f^\epsilon, \pa^{\alpha_2}_v F^{\epsilon}_{R})\langle v \rangle^{l} , g^{\epsilon}_{\alpha, l} \rangle \nonumber \\
&&+ \langle M^{\epsilon}(\pa^{\alpha_1}_v F^{\epsilon}_{R}, \pa^{\alpha_2}_v f)\langle v \rangle^{l} , g^{\epsilon}_{\alpha, l} \rangle \nonumber \\
&&+ \langle \Upsilon(\pa^{\alpha_1}_v f, \pa^{\alpha_2}_v f) \langle v \rangle^{l} , g^{\epsilon}_{\alpha, l} \rangle
\} \nonumber \\
&\eqdefa& \sum_{\alpha_1+\alpha_2=\alpha} \binom{\alpha}{\alpha_{1}} \{\mathfrak{I}_{1}(\alpha_{1},\alpha_{2})+\mathfrak{I}_{2}(\alpha_{1},\alpha_{2})+
\mathfrak{I}_{3}(\alpha_{1},\alpha_{2})\}. \nonumber
\end{eqnarray}

Again by lemma 7.1 in the Appendix of \cite{he1}, we have
\beno
\langle \Upsilon(\pa^{\alpha_1}_v f, \pa^{\alpha_2}_v f) \langle v \rangle^{l} , g^{\epsilon}_{\alpha, l} \rangle \lesssim
||f||^{2}_{H^{m+5}_{l+\gamma+10}}||g^{\epsilon}_{\alpha, l} ||_{L^{2}},
\eeno
which implies, for any $\eta > 0$,
\begin{eqnarray}\label{hmli3}
\langle \Upsilon(\pa^{\alpha_1}_v f, \pa^{\alpha_2}_v f) \langle v \rangle^{l} , g^{\epsilon}_{\alpha, l} \rangle  - \eta ||g^{\epsilon}_{\alpha, l} ||_{L^{2}} \lesssim \frac{1}{\eta}||f||^{4}_{H^{m+5}_{l+\gamma+10}}.
\end{eqnarray}

Splitting $\mathfrak{I}_{1}(\alpha_{1},\alpha_{2})$ into two terms, we have
\beno
\mathfrak{I}_{1}(\alpha_{1},\alpha_{2}) &=& \langle M^{\epsilon}(\pa^{\alpha_1}_vf^\epsilon, \langle v \rangle^{l} \pa^{\alpha_2}_vF^{\epsilon}_{R}) , g^{\epsilon}_{\alpha, l} \rangle \\ &&+\{\langle M^{\epsilon}(\pa^{\alpha_1}_v f^\epsilon, \pa^{\alpha_2}_v F^{\epsilon}_{R})\langle v \rangle^{l} , g^{\epsilon}_{\alpha, l} \rangle -\langle M^{\epsilon}(\pa^{\alpha_1}_vf^\epsilon, \langle v \rangle^{l} \pa^{\alpha_2}_vF^{\epsilon}_{R} ) , g^{\epsilon}_{\alpha, l} \rangle\} \\
&\eqdefa& \mathfrak{I}_{1,1}(\alpha_{1},\alpha_{2}) + \mathfrak{I}_{1,2}(\alpha_{1},\alpha_{2}).
\eeno
By coercivity estimate (\ref{coercivityboltz}), we have
\begin{eqnarray}\label{hmli110m}
\mathfrak{I}_{1,1}(0,\alpha) \leq -C_{1}(f_{0})||g^{\epsilon}_{\alpha, l}||^{2}_{\epsilon,\gamma/2}+ C_{2}(f_{0})||g^{\epsilon}_{\alpha, l}||^{2}_{L^{2}_{\gamma/2}}.
\end{eqnarray}
For $1 \leq |\alpha_{1}| \leq |\alpha| \leq m$,  by upper bound estimate (\ref{upcboltz}) with $ w_{1} = \gamma/2 + 2, w_{2} = \gamma/2$, we have
\beno
\mathfrak{I}_{1,1}(\alpha_{1},\alpha_{2}) &\lesssim& ||\pa^{\alpha_1}_vf^\epsilon||_{L^{1}_{4}} ||\pa^{\alpha_2}_vF^{\epsilon}_{R} ||_{H^{1}_{l+\gamma/2 + 2}} ||g^{\epsilon}_{\alpha, l}||_{\epsilon,\gamma/2} \\
&\lesssim& ||f^\epsilon||_{H^{m}_{6}} ||F^{\epsilon}_{R}||_{H^{m}_{l+\gamma/2 + 2}}||g^{\epsilon}_{\alpha, l}||_{\epsilon,\gamma/2},
\eeno
which implies, for any $\eta_{1}>0$,
\begin{eqnarray}\label{hmli111m1}
\mathfrak{I}_{1,1}(\alpha_{1},\alpha_{2}) - \eta_{1} ||g^{\epsilon}_{\alpha, l}||^{2}_{\epsilon,\gamma/2} \lesssim \frac{1}{\eta_{1}}||f^\epsilon||^{2}_{H^{m}_{6}} ||F^{\epsilon}_{R}||^{2}_{H^{m}_{l+\gamma/2 + 2}}.
\end{eqnarray}
By commutator estimates (\ref{commcboltz}) with $N_{2} = l + \gamma/2, N_{3} = \gamma/2$, we have
\beno
\mathfrak{I}_{1,2}(\alpha_{1},\alpha_{2}) \lesssim  ||f^{\epsilon}||_{H^{m}_{2l+7}} ||\pa^{\alpha_2}_v F^{\epsilon}_{R}||_{\epsilon,l+\gamma/2} ||g^{\epsilon}_{\alpha, l}||_{L^{2}_{\gamma/2}},
\eeno
which implies, for any $\eta_{2} > 0$,
\begin{eqnarray}\label{hmli12anym}
\mathfrak{I}_{1,2}(\alpha_{1},\alpha_{2})
-\eta_{2} ||\pa^{\alpha_2}_v F^{\epsilon}_{R}||^{2}_{\epsilon,l+\gamma/2} \lesssim \frac{1}{\eta_{2}} ||f^{\epsilon}||^{2}_{H^{m}_{2l+7}} ||g^{\epsilon}_{\alpha, l}||^{2}_{L^{2}_{\gamma/2}}.
\end{eqnarray}

Splitting $\mathfrak{I}_{2}(\alpha_{1},\alpha_{2})$ into two terms, we have
\beno
\mathfrak{I}_{2}(\alpha_{1},\alpha_{2}) &=& \langle M^{\epsilon}(\pa^{\alpha_1}_v F^{\epsilon}_{R}, \langle v \rangle^{l} \pa^{\alpha_2}_v f ) , g^{\epsilon}_{\alpha, l} \rangle \\ &&+\{\langle M^{\epsilon}(\pa^{\alpha_1}_v F^{\epsilon}_{R}, \pa^{\alpha_2}_v f)\langle v \rangle^{l} , g^{\epsilon}_{\alpha, l} \rangle -\langle M^{\epsilon}(\pa^{\alpha_1}_v F^{\epsilon}_{R}, \langle v \rangle^{l} \pa^{\alpha_2}_v f ) , g^{\epsilon}_{\alpha, l} \rangle\} \\
&\eqdefa& \mathfrak{I}_{2,1}(\alpha_{1},\alpha_{2}) + \mathfrak{I}_{2,2}(\alpha_{1},\alpha_{2}).
\eeno
Applying upper bound estimate (\ref{upcboltz}) with $ w_{1} = \gamma/2 + 2, w_{2} = \gamma/2$, we may have
\beno
\mathfrak{I}_{2,1}(\alpha_{1},\alpha_{2}) &\lesssim& ||\pa^{\alpha_1}_v F^{\epsilon}_{R}||_{L^{1}_{4}} ||\pa^{\alpha_2}_v f ||_{H^{1}_{l+\gamma/2 + 2}} ||g^{\epsilon}_{\alpha, l}||_{\epsilon,\gamma/2} \\
&\lesssim& ||F^{\epsilon}_{R}||_{H^{m}_{6}} ||f||_{H^{m+1}_{l+\gamma/2 + 2}}||g^{\epsilon}_{\alpha, l}||_{\epsilon,\gamma/2},
\eeno
which implies, for any $\eta_{1}>0$,
\begin{eqnarray}\label{hmli21anym}
\mathfrak{I}_{2,1}(\alpha_{1},\alpha_{2}) - \eta_{1} ||g^{\epsilon}_{\alpha, l}||^{2}_{\epsilon,\gamma/2} \lesssim \frac{1}{\eta_{1}}||f||^{2}_{H^{m+1}_{l+\gamma/2 + 2}} ||F^{\epsilon}_{R}||^{2}_{H^{m}_{6}}.
\end{eqnarray}
By commutator estimate (\ref{commcboltz}) with $N_{2} = l + \gamma/2, N_{3} = \gamma/2$, we have
\beno
\mathfrak{I}_{2,2}(\alpha_{1},\alpha_{2})
\lesssim  ||F^{\epsilon}_{R}||_{H^{m}_{2l+7}} ||f||_{H^{m+1}_{l+\gamma/2}} ||g^{\epsilon}_{\alpha, l}||_{L^{2}_{\gamma/2}},
\eeno
which implies, for any $\eta_{1} > 0$,
\begin{eqnarray}\label{hmli22anym}
\mathfrak{I}_{2,2}(\alpha_{1},\alpha_{2}) - \eta_{1} ||g^{\epsilon}_{\alpha, l}||^{2}_{L^{2}_{\gamma/2}} \lesssim \frac{1}{\eta_{1}}||f||^{2}_{H^{m+1}_{l+\gamma/2}}||F^{\epsilon}_{R}||^{2}_{H^{m}_{2l+7}}.
\end{eqnarray}

Patching all together (\ref{hmli3}),(\ref{hmli110m}),(\ref{hmli111m1}),(\ref{hmli12anym}),(\ref{hmli21anym}),(\ref{hmli22anym}), and  taking $\eta_{1}$ small enough, we arrive at
\begin{eqnarray}\label{errorhmlclose1}
&&\frac{d}{dt}(\frac{1}{2}||g^{\epsilon}_{\alpha, l}||^{2}_{L^{2}}) + \frac{3 C_{1}(f_{0})}{4}||g^{\epsilon}_{\alpha, l}||^{2}_{\epsilon,\gamma/2}-\eta_{2} ||F^{\epsilon}_{R}||^{2}_{\epsilon,m,l+\gamma/2} \\&\lesssim& ||f||^{4}_{H^{m+5}_{l+\gamma+10}}
+ ( C_{2}(f_{0})+ \frac{1}{\eta_{2}} ||f^{\epsilon}||^{2}_{H^{m}_{2l+7}} )||g^{\epsilon}_{\alpha, l}||^{2}_{L^{2}_{\gamma/2}} \nonumber \\
&& + \frac{1}{\eta_{1}}(||f^\epsilon||^{2}_{H^{m}_{6}} +  ||f||^{2}_{H^{m+1}_{l+\gamma/2 + 2}}+||f||^{2}_{H^{m+1}_{l+\gamma/2}})||F^{\epsilon}_{R}||^{2}_{H^{m}_{2l+7}}. \nonumber
\end{eqnarray}
Let $a(m) = \sum_{r=0}^{m}\binom{r+2}{r}$. Summing over $|\alpha| \leq m$, by taking $\eta_{2} = \frac{C_{1}(f_{0})}{4}\frac{1}{a(m)}$, we have
\begin{eqnarray}\label{errorhmlclose2}
&&\frac{d}{dt}(\frac{1}{2}||F^{\epsilon}_{R}||^{2}_{H^{m}_{l}})  + \frac{C_{1}(f_{0})}{2}||F^{\epsilon}_{R}||^{2}_{\epsilon, m ,l+\gamma/2} \\ &\lesssim& ||f||^{4}_{H^{m+5}_{l+\gamma+10}} + (C_{2}(f_{0})+ \frac{1}{\eta_{2}} ||f^{\epsilon}||^{2}_{H^{m}_{2l+7}}  )||F^{\epsilon}_{R}||^{2}_{H^{m}_{l+\gamma/2}} \nonumber \\
&& + \frac{1}{\eta_{1}}(||f^\epsilon||^{2}_{H^{m}_{6}} + ||f||^{2}_{H^{m+1}_{l+\gamma/2 + 2}}+||f||^{2}_{H^{m+1}_{l+\gamma/2}})||F^{\epsilon}_{R}||^{2}_{H^{m}_{2l+7}}.  \nonumber
\end{eqnarray}
Thanks to (\ref{hmhmsmminus1}), we may conclude
\begin{eqnarray}\label{errorhmlclose3}
\frac{d}{dt}||F^{\epsilon}_{R}||^{2}_{H^{m}_{l}}    + \frac{C_{1}(f_{0})}{2}||F^{\epsilon}_{R}||^{2}_{\epsilon, m ,l+\gamma/2} \lesssim C(||f||_{H^{m+5}_{l+\gamma+10}}, ||F^{\epsilon}_{R}||_{H^{m-1}_{z(l)}}, ||f_{0}||_{LlogL},||f_{0}||_{L^{1}_{1}}). \nonumber
\end{eqnarray}
By theorem \ref{main2} and remark \ref{alsotrue}, for any $t \geq 0$, we have
\beno
||f(t)||_{H^{m+5}_{l+\gamma+10}} \lesssim C(||f_{0}||_{L^{1}_{\phi(m+5,l+\gamma+10)}},||f_{0}||_{H^{m+5}_{l+\gamma+10}}).
\eeno
By assumption, there holds
\beno
||F^{\epsilon}_{R}(t)||_{H^{m-1}_{z(l)}} \lesssim C(||f_{0}||_{L^{1}_{\varphi(m-1,z(l))}}, ||f_{0}||_{H^{m+4}_{\psi(m-1,z(l))}}, t).
\eeno
On the other hand, by interpolation, we have
\beno
||f_{0}||_{H^{m+4}_{\psi(m-1,z(l))}} \lesssim ||f_{0}||_{H^{m+5}_{l+\gamma+10}}+||f_{0}||_{L^{1}_{\rho(m,l)}},
\eeno
where $\rho(m,l) = (m+7)\psi(m-1,z(l)) - (m+6)(l+\gamma+10)$. By defining $\varphi(m,l) = \max\{\varphi(m-1,z(l)), \rho(m,l)\}$,  we have
\beno
||F^{\epsilon}_{R}(t)||^{2}_{H^{m}_{l}} \leq C(||f_{0}||_{L^{1}_{\varphi(m,l)}}, ||f_{0}||_{H^{m+5}_{l+\gamma+10}}, t).
\eeno

\bigskip

{\bf Acknowledgments.} The authors thank the first 2016 exchange program between Hong Kong and the mainland of China, supported by the Ministry of education of the People's Republic of China. They also express their gratitude to the Department of Mathematical Sciences of Tsinghua University and the Department of Mathematics of Hong Kong Baptist University for the kindly hospitality.


\begin{thebibliography}{99}

\bibitem{advw} R. Alexandre, L. Desvillettes, C. Villani, and B. Wennberg, Entropy dissipation and long-
range interactions, {\it Arch. Ration. Mech. Anal.} 152 (2000), no. 4, 327-355.


\bibitem{amuxy6} R. Alexandre, Y. Morimoto, S. Ukai, C.-J. Xu, and T.
Yang, Smoothing effect of weak solutions for the spatially homogeneous Boltzmann equation without angular cutoff, {\it Kyoto J. Math.}  52 (2012), no. 3, 433-463.



\bibitem{ar1}  L. Arkeryd,  On the Boltzmann equation, {\it Arch. Rational Mech. Anal.} 45 (1972), 1-34.


\bibitem {chenhe} Y.~Chen and L.~He, Smoothing estimates for   Boltzmann equation with full-range
interactions  I: spatially homogeneous case,  {\it Arch. Ration. Mech. Anal.} 201 (2011), no. 2, 501-548.



\bibitem{ld1} L. Desvillettes, On asymptotics of the Boltzmann equation when the collisions become grazing, {\it Transp.
Theory Stat. Phys.} 21(3)  (1992), 259–276.


\bibitem{dv1} L. Desvillettes, C. Villani, On the spatially homogeneous landau equation for hard potentials part i: existence, uniqueness and smoothness., {\it Comm. Partial Differential Equations} 25.1-2 (2000), 179-259.





\bibitem{fg3} N. Fournier and D. Godinho, Asymptotic of grazing collisions and particle approximation for the Kac equation without cutoff, {\it Commun. Math. Phys.} 316.2 (2012), 307-344.

\bibitem{fg2} N. Fournier and A. Guillin, From a Kac-like particle system to the Landau equation for hard potentials and Maxwell molecules, {\it arXiv preprint arXiv:1510.01123}, (2015).











\bibitem{he1} L. He, Asymptotic analysis of the spatially homogeneous Boltzmann equation: grazing collisions limit, {\it J. Stat. Phys.} 155.1 (2014), 151-210 .

\bibitem{he3} L. He, Well-posedness of spatially homogeneous Boltzmann
equation with full-range interaction, {\it Commun. Math. Phys.}  312 (2012), 447-476.

\bibitem{he2} L. He, Sharp bounds for Boltzmann and Landau collision operators,  {\it arXiv:1604.06981}, (2016).

\bibitem{HeJiang} L.-B.He and J.-C. Jiang, On the Cauchy problem for inhomogeneous Boltzmann equations with Hard potentials: Well-posedness and  Global  stability, in preparation.

\bibitem{hmuy} Z. Huo, Y. Morimoto, S. Ukai and T. Yang,
Regularity of solutions for spatially homogeneous Boltzmann equation
without angular cutoff, {\it Kinet. Relat. Models} 1 (2008), no. 3,
453-489.



\bibitem{lm}  X. Lu and  C. Mouhot,   On Measure Solutions of the Boltzmann Equation, part I: Moment Production and Stability Estimates, {\it J. Differential Equations} 4 (2012), 3305-3363.

\bibitem{mm}
S. Mischler and C. Mouhot, Kac’s program in kinetic theory, {\it Inventiones mathematicae} (1) 193 (2013), 1-47.

\bibitem{mmw}  S. Mischler, C. Mouhot, and B. Wennberg, A new approach to quantitative propagation of chaos for drift, diffusion and jump processes, {\it Probability Theory and Related Fields} 161.1-2 (2015), 1-59.

\bibitem{mw} S. Mischler and  B. Wennberg,  On the spatially homogeneous Boltzmann equation, {\it Ann. Inst. H. Poincar\'{e}Anal. Non Lin\'{e}aire} 16 (1999), 467-501.

\bibitem{mv} C. Mouhot and C. Villani, Regularity theory for the spatially homogeneous Boltzmann equation with cut-off,
{\it Arch. Ration. Mech. Anal.} 173 (2004), 169-212.

\bibitem{sl} L. Silvestre, A new regularization mechanism for the Boltzmann equation without cut-off, {\it Commun. Math. Phys.}  348 (2016),  69-100.

\end{thebibliography}
\end{document}